\documentclass[english]{IEEEtran}
\usepackage[T3,T1]{fontenc}
\usepackage[latin9]{inputenc}
\usepackage{color}
\usepackage{babel}
\usepackage{mathtools}
\usepackage{amsmath}
\usepackage{amsthm}
\usepackage{amssymb}
\usepackage{graphicx}
\usepackage[unicode=true,pdfusetitle,
 bookmarks=true,bookmarksnumbered=true,bookmarksopen=true,bookmarksopenlevel=3,
 breaklinks=false,pdfborder={0 0 1},backref=false,colorlinks=true]
 {hyperref}
\hypersetup{
 pdfborderstyle=,pdfborderstyle={},pdfborderstyle={},pdfborderstyle={},pdfborderstyle={},pdfborderstyle={},pdfborderstyle={},pdfborderstyle={},pdfborderstyle={},pdfborderstyle={},pdfborderstyle={},pdfborderstyle={},pdfborderstyle={},pdfborderstyle={},pdfborderstyle={},pdfborderstyle={},pdfborderstyle={},pdfborderstyle={},pdfborderstyle={},pdfborderstyle={},pdfpagelayout=OneColumn,pdfnewwindow=true,pdfstartview=XYZ,plainpages=false,linkcolor=blue,urlcolor=blue,citecolor=red,anchorcolor=blue,linkcolor=blue,urlcolor=blue,citecolor=red,anchorcolor=blue}

\makeatletter
\theoremstyle{plain}
\newtheorem{thm}{\protect\theoremname}
\theoremstyle{definition}
\newtheorem{example}{\protect\examplename}
\theoremstyle{plain}
\newtheorem{cor}{\protect\corollaryname}
\theoremstyle{remark}
\newtheorem{rem}{\protect\remarkname}
\theoremstyle{plain}
\newtheorem{prop}{\protect\propositionname}
\theoremstyle{plain}
\newtheorem{lem}{\protect\lemmaname}
\theoremstyle{definition}
\newtheorem{defn}{\protect\definitionname}

\usepackage{babel}
\usepackage{enumerate}
\usepackage{verbatim}
\usepackage{mathrsfs}
\usepackage{cite}
\usepackage{bbm}

\mathtoolsset{showonlyrefs} 

\usepackage{lineno}

\DeclareSymbolFont{tipa}{T3}{cmr}{m}{n}
\DeclareMathAccent{\inbreve}{\mathalpha}{tipa}{16}

\newcommand{\subsetsim}{\mathrel{%
  \ooalign{\raise0.2ex\hbox{$\subset$}\cr\hidewidth\raise-0.8ex\hbox{\scalebox{0.9}{$\sim$}}\hidewidth\cr}}}
\newcommand{\supsetsim}{\mathrel{%
  \ooalign{\raise0.2ex\hbox{$\supset$}\cr\hidewidth\raise-0.8ex\hbox{\scalebox{0.9}{$\sim$}}\hidewidth\cr}}}

\newcommand{\subsetapprox}{\mathrel{%
  \ooalign{\raise0.4ex\hbox{$\subset$}\cr\hidewidth\raise-0.8ex\hbox{\scalebox{0.9}{$\approx$}}\hidewidth\cr}}}
\allowdisplaybreaks[1]
\newcommand{\bone}{\mathbbm{1}}

\newcommand{\calA}{\mathcal{A}}

\newcommand{\calP}{\mathcal{P}}

\newcommand{\calX}{\mathcal{X}}

\def\cC{{\mathcal C}}

\def\Bern{\mathrm{Bern}}

\def\Unif{\mathrm{Unif}}

\def\supp{{\mathrm{supp}}}

\def\conv{\breve}
\def\conc{\inbreve}

\def\P{\mathsf{P}}
\def\Q{\mathsf{Q}}
\def\R{\mathsf{R}}

\def\C{\mathsf{C}}
\def\W{\mathsf{W}}
\def\S{{\mathsf{S}}}
\def\Levy{  d_{\rm P}}
\def\L{\mathsf{L}} 

\def\1{\mathbf{1}}

\def\d{{\text {\rm d}}}

\providecommand{\corollaryname}{Corollary}
\providecommand{\lemmaname}{Lemma}
\providecommand{\propositionname}{Proposition}
\providecommand{\remarkname}{Remark}
\providecommand{\theoremname}{Theorem}

\providecommand{\definitionname}{Definition}

\makeatother

\providecommand{\corollaryname}{Corollary}
\providecommand{\definitionname}{Definition}
\providecommand{\examplename}{Example}
\providecommand{\lemmaname}{Lemma}
\providecommand{\propositionname}{Proposition}
\providecommand{\remarkname}{Remark}
\providecommand{\theoremname}{Theorem}

\begin{document}
\title{Exact Exponents for Concentration and Isoperimetry in Product Polish
Spaces}
\author{Lei Yu\thanks{L. Yu is with the School of Statistics and Data Science, LPMC, KLMDASR,
and LEBPS, Nankai University, Tianjin 300071, China (e-mail: leiyu@nankai.edu.cn).
This work was supported by the National Key Research and Development
Program of China under grant 2023YFA1009604, the NSFC under grant
62101286, and the Fundamental Research Funds for the Central Universities
of China (Nankai University) under grant 054-63243076.}}
\maketitle
\begin{abstract}
In this paper, we derive variational formulas for the asymptotic exponents
(i.e., convergence rates) of the concentration and isoperimetric functions
in the product Polish probability space under certain mild assumptions.
These formulas are expressed in terms of relative entropies (which
are from information theory) and optimal transport cost functionals
(which are from optimal transport theory). Hence, our results verify
an intimate connection among information theory, optimal transport,
and concentration of measure or isoperimetric inequalities. In the
concentration regime, the corresponding variational formula is in
fact a dimension-free bound in the sense that this bound is valid
for any dimension. A cardinality bound for the alphabet of the auxiliary
random variable in the expression of the asymptotic isoperimetric
exponent is provided, which makes the expression computable by a finite-dimensional
program for the finite alphabet case. We lastly apply our results
to obtain an isoperimetric inequality in the classic isoperimetric
setting, which is asymptotically sharp under certain conditions. The
proofs in this paper are based on information-theoretic and optimal
transport techniques. 
\end{abstract}







\section{Introduction}

Concentration of measure in a probability metric space refers to a
phenomenon that a slight enlargement of any measurable set of not
small probability will always have large probability. In the language
of functional analysis, it is equivalent to a phenomenon that the
value of any Lipschitz function is concentrated around its median.
The concentration of measure phenomenon was pushed forward in the
early 1970s by V. Milman in the study of the asymptotic geometry of
Banach spaces. It was then studied in depth by Milman and many other
authors including Gromov, Maurey, Pisier, Schechtman, Talagrand, Ledoux,
etc. In particular, Talagrand \cite{talagrand1995concentration} studied
the concentration of measure in product spaces equipped with product
probability measures, and derived a variety of concentration of measure
inequalities for these spaces. In information theory, concentration
of measure is known as the blowing-up lemma \cite{Ahls76,Marton86},
which was employed by Gács, Ahlswede, and Körner to prove the strong
converses of two coding problems in information theory.

It is worth mentioning that Marton is the first to introduce information-theoretic
techniques, especially transport-entropy inequalities, in the study
of the concentration of measure \cite{Marton86}, which yields an
elegant and short proof for this phenomenon. By developing a new transport-entropy
inequality, Talagrand extended her idea to the case of Gaussian measure
and Euclidean metric \cite{talagrand1996transportation}. Since then,
such a textbook beautiful argument became popular and emerged in many
books, e.g., \cite{ledoux2001concentration,RagSason,villani2003topics}.
By replacing the ``linear'' transport-entropy inequality in Marton's
argument with a ``nonlinear'' version, Gozlan and Léonard obtained
the sharp dimension-free bound on the concentration function \cite{gozlan2007large}.
In other words, their bound is exponentially tight in the sense that
the exponent of their bound asymptotically coincides with that of
the concentration function. Furthermore, Gozlan \cite{gozlan2009characterization}
also used Marton's argument to prove the equivalence between the Gaussian
bound of the concentration function and Talagrand's transport-entropy
inequality. Dembo \cite{dembo1997information} provided a new kind
of transport-entropy inequalities, and used them to recover several
results of Talagrand \cite{talagrand1995concentration}.

Ahlswede and Zhang \cite{ahlswede1999asymptotical,ahlswede1997identification}
focused on the isoperimetric regime of the concentration problem,
in which they assumed the set to be small enough such that its enlargement
is small as well. In this regime, the problem turns into an isoperimetric
problem where the difference between the enlargement and the original
set is regarded as the ``boundary'' of the set. They characterized
the asymptotic exponents for this problem by using information-theoretic
methods. Their results was used as a key tool to study the identification
problem \cite{ahlswede1997identification}. 

In this paper, we investigated the concentration (or isoperimetric)
problem in the product Polish space. Specifically, we minimize the
probability of the $t$-enlargement (or $t$-neighborhood) $A^{t}$
of a set $A$ under the condition that the probability of $A$ is
given. Here, different from the common setting in concentration of
measure, the probability of $A$ is not necessarily restricted to
be around $1/2$. The probability of $A$ could be small or large.
We now introduce the mathematical formulation.

Let $\mathcal{X}$ and $\mathcal{Y}$ be Polish spaces (i.e., separable
completely metrizable spaces, including Euclidean spaces and countable
metric spaces as special cases). Let $\Sigma(\mathcal{X})$ and $\Sigma(\mathcal{Y})$
be respectively the Borel $\sigma$-algebras on $\mathcal{X}$ and
$\mathcal{Y}$ that are generated by the topologies on $\mathcal{X}$
and $\mathcal{Y}$. Let $\mathcal{P}(\mathcal{X})$ and $\mathcal{P}(\mathcal{Y})$
denote the sets of probability measures (or distributions) on $\mathcal{X}$
and $\mathcal{Y}$ respectively. Let $P_{X}\in\mathcal{P}(\mathcal{X})$
and $P_{Y}\in\mathcal{P}(\mathcal{Y})$. In other words, $P_{X}$
and $P_{Y}$ are respectively the distributions of two random variables
$X$ and $Y$. Let $c:\mathcal{X}\times\mathcal{Y}\to[0,+\infty)$
be lower semi-continuous, which is called a cost function. Denote
$\mathcal{X}^{n}$ as the $n$-fold product space of $\mathcal{X}$.
For the product space $\mathcal{X}^{n}\times\mathcal{Y}^{n}$ and
given $c$, we consider an additive cost function $c_{n}$ on $\mathcal{X}^{n}\times\mathcal{Y}^{n}$
given by 
\[
c_{n}(x^{n},y^{n}):=\sum_{i=1}^{n}c(x_{i},y_{i})\quad\mbox{for }(x^{n},y^{n})\in\mathcal{X}^{n}\times\mathcal{Y}^{n},
\]
where $c$ given above is independent of $n$. Obviously, $c_{n}$
is lower semi-continuous since $c$ is lower semi-continuous.

For a set $A\subseteq\mathcal{X}^{n}$, denote its $t$-enlargement
under $c$ as 
\begin{equation}
A^{t}:=\bigcup_{x^{n}\in A}\{y^{n}\in\mathcal{Y}^{n}:c_{n}(x^{n},y^{n})\leq t\}.\label{eq:Gamma-1}
\end{equation}
To address the measurability of $A^{t}$, we assume that either of
the following two conditions holds throughout this paper. 
\begin{enumerate}
\item For lower semi-continuous $c$, we restrict $A$ to be a closed set. 
\item If $\mathcal{X}$ and $\mathcal{Y}$ are the same Polish space and
$c=d^{p}$, where $p>0$ and $d$ is a complete metric that induces
the topology on this Polish space, then $A$ can be any Borel set. 
\end{enumerate}
For the first case, since $\mathcal{X}^{n}$ and $\mathcal{Y}^{n}$
are Polish and a projection map is continuous, by definition, the
projection of a closed (or open) subset of $\mathcal{X}^{n}\times\mathcal{Y}^{n}$
to $\mathcal{X}^{n}$ is analytic (or Souslin) \cite{bogachev2007measure}.
Note that for closed $A$, $A^{t}$ is the projection of the closed
set $c_{n}^{-1}((-\infty,t])\cap(A\times\mathcal{Y}^{n})$ to $\mathcal{X}^{n}$.
So, the set $A^{t}$ is analytic and hence, universally measurable.
If we extend $P_{Y}^{\otimes n}$ to the collection of analytic sets,
then $P_{Y}^{\otimes n}(A^{t})$ is well defined. Hence, for this
case, we by default adopt this extension to avoid the measurability
problem. For the second case, for any Borel set $A$, $A^{t}$ is
always Borel (since it is countable intersections of Borel sets $\bigcup_{x^{n}\in A}\{y^{n}\in\mathcal{Y}^{n}:c_{n}(x^{n},y^{n})<t+\frac{1}{k}\},k=1,2,...$).

Define the isoperimetric function (or isoperimetric profile) as for
$a\in[0,1],t\ge0$, 
\begin{equation}
\Gamma^{(n)}(a,t):=\inf_{A:P_{X}^{\otimes n}(A)\ge a}P_{Y}^{\otimes n}(A^{t}),\label{eq:Gamma}
\end{equation}
where the set $A$ is assumed to satisfy either of the above two conditions.
We call $(a,t)\mapsto1-\Gamma^{(n)}(a,t)$ as the concentration function,
which reduces to the usual concentration function $t\mapsto1-\Gamma^{(n)}(\frac{1}{2},t)$
in the theory of concentration of measure when $a$ is set to $1/2$.
Throughout this paper, we set 
\begin{equation}
a=e^{-n\alpha},\;t=n\tau.\label{eq:alpha-tau}
\end{equation}
Define the isoperimetric and concentration exponents respectively
as\footnote{Throughout this paper, the base of $\log$ is $e$. Our results are
still true if the bases are chosen to other values, as long as the
bases of the logarithm and exponent are the same. } for $\alpha,\tau\ge0$, 
\begin{align}
E_{0}^{(n)}(\alpha,\tau) & :=-\frac{1}{n}\log\Gamma^{(n)}(e^{-n\alpha},n\tau)\label{eq:-28}\\
E_{1}^{(n)}(\alpha,\tau) & :=-\frac{1}{n}\log\left(1-\Gamma^{(n)}(e^{-n\alpha},n\tau)\right).\label{eq:E1}
\end{align}
In fact, 
\[
\Gamma^{(n)}(e^{-n\alpha},n\tau)=e^{-nE_{0}^{(n)}(\alpha,\tau)}=1-e^{-nE_{1}^{(n)}(\alpha,\tau)}.
\]

In the classic setting, $\mathcal{X}=\mathcal{Y}$ equipped with a
metric $d$ is a Polish metric space, and moreover, $P_{X}=P_{Y}=:P$
and $c=d^{p}$ with $p\ge1$. For a set $A$, its boundary measure
is defined by\footnote{For the discrete metric, this definition does not make sense, since
$(P^{\otimes n})^{+}(A)=0$ for any set $A$. So, in this case, $(P^{\otimes n})^{+}(A)$
can be defined by $(P^{\otimes n})^{+}(A):=P^{\otimes n}(A^{1})-P^{\otimes n}(A)$.} 
\begin{equation}
(P^{\otimes n})^{+}(A):=\liminf_{r\downarrow0}\frac{P^{\otimes n}(A^{r^{p}})-P^{\otimes n}(A)}{r}.\label{eq:perimeter}
\end{equation}
In the classic isoperimetric problem, the objective is to minimize
$(P^{\otimes n})^{+}(A)$ over all sets $A$ with a given probability. 

In this paper, we aim at characterizing the asymptotics of the concentration
and isoperimetric exponents in \eqref{eq:-28} and \eqref{eq:E1},
as well as applying these results to obtain an asymptotically sharp
inequality on the classic isoperimetric problem (under certain conditions). 

\subsection{Our Contributions}

Our contributions in this paper are as follows.  
\begin{enumerate}
\item We characterize the asymptotic concentration exponent $\lim_{n\to\infty}E_{1}^{(n)}(\alpha,\tau)$
(under certain mild assumptions) in terms of two fundamental quantities
from other fields---``relative entropy'' which comes from information
theory (or large deviations theory) and ``optimal transport cost''
which comes from the theory of optimal transport. The (conditional)
empirically typical sets are shown to be optimal in attaining the
asymptotic concentration exponent. Hence, this result further verifies
an intimate connection among concentration of measure, information
theory, and optimal transport. The obtained expression for $\lim_{n\to\infty}E_{1}^{(n)}(\alpha,\tau)$
is shown to be a dimension-free bound on $E_{1}^{(n)}(\alpha,\tau)$.
This bound is tighter than Marton's bound \cite{Marton86,marton1996bounding}
and an improved version by Gozlan and Léonard \cite{gozlan2005principe,gozlan2007large},
especially when the probability of the set is small. It also sharpens
Talagrand's concentration inequality in \cite{talagrand1995concentration}.
The improvement is due to that the single-letterization part in our
proof relies on the subadditivity of optimal transport (OT) costs
(or equivalently, relies on a new and more general transport-entropy
inequality), and bypasses the traditional transport-entropy inequality
in Marton's proof and the nonlinear transport-entropy inequality in
Gozlan and Léonard's proof. As applications, we also consider the
case that $c=d^{p}$ with $p\ge1$ and $d$ denoting a metric and
the case that $c$ is the Hamming metric. We obtain cleaner expressions
for the asymptotic concentration exponents for these two cases, and
also recover existing results for the setting of $a=\frac{1}{2}$,
including Gozlan and Léonard's \cite{gozlan2005principe,gozlan2007large}
and Alon, Boppana, and Spencer's in \cite{alon1998asymptotic}.  
\item We also provide upper and lower bounds for the asymptotic isoperimetric
exponent $\lim_{n\to\infty}E_{0}^{(n)}(\alpha,\tau)$ (under certain
mild assumptions) for Polish spaces. These bounds are also expressed
in terms of the relative entropy and the optimal transport cost. Under
a continuity assumption, the bounds coincide, which yields an exact
characterization of the asymptotic isoperimetric exponent. This result
is a generalization of Ahlswede and Zhang's \cite{ahlswede1999asymptotical}
from finite spaces to Polish spaces. In fact, similar to Ahlswede
and Zhang's proof, our proof also relies on the inherently typical
subset lemma, but requires new techniques  since the spaces are much
more general. 
\item Our another contribution is deriving dual formulas for the bounds
or expressions mentioned above for the asymptotic concentration or
isoperimetric exponents. By our dual formulas, on one hand, we verify
the equivalence between our formula and Alon, Boppana, and Spencer's
in \cite{alon1998asymptotic} for the asymptotic concentration exponent;
on the other hand, we provide a bound on the alphabet size of the
auxiliary random variable in the expression of the asymptotic isoperimetric
exponent. These two observations are not obvious from the perspective
of primal formulas. Previously, there was no bound on the alphabet
size of the auxiliary random variable, even for the finite alphabet
case considered by Ahlswede and Zhang \cite{ahlswede1999asymptotical}.
As explicitly mentioned in \cite[Remark 1 on p. 50]{ahlswede1997identification},
deriving cardinality bounds for the auxiliary random variable is not
easy. Deriving cardinality bounds is also important, since it makes
the expression ``computable'' for the finite alphabet case. That
is, it enables us to evaluate the expression by a finite-dimensional
program when the alphabets are finite. 
\item The isoperimetric problem mentioned above concerns thick boundaries.
In contrast, in the classic isoperimetric problem, the boundary is
extremely thin, as shown in \eqref{eq:perimeter}. We apply our results
to obtain the following isoperimetric inequality: 
\begin{align}
(P^{\otimes n})^{+}(A) & \ge n^{1-1/p}e^{-n\alpha}(\xi(\alpha)+o_{n}(1)),\label{eq:-67-1}
\end{align}
where $\xi(\alpha)$ is a certain function defined in \eqref{eq:xi}.
This inequality is asymptotically sharp under certain conditions.
\end{enumerate}

\subsection{Organization}

This paper is organized as follows. In Section \ref{subsec:Notations},
we introduce the notations used in this paper. In Section \ref{sec:Main-Results},
we state our main results, including a dimension-free bound for the
concentration exponent, the characterizations of the asymptotic concentration
exponent and the asymptotic isoperimetric exponent, and the dual formulas
for our bounds and expressions. We also discuss the connections of
the concentration or isoperimetric problems to the Strassen's optimal
transport problem in Section \ref{sec:Main-Results}. In the same
section, we also apply our results to obtain an isoperimetric inequality
for the classic isoperimetric setting, which is asymptotically sharp
in certain conditions. The proofs of these results are provided in
Section \ref{sec:Proof-of-Theorem-1}-\ref{sec:Proof-of-Theorem}.

\subsection{\label{subsec:Notations}Notations }

\subsubsection{Probability Theory }

Throughout this paper, for a topological space $\mathcal{Z}$, we
use $\Sigma(\mathcal{Z})$ to denote the Borel $\sigma$-algebra on
$\mathcal{Z}$ generated by the topology of $\mathcal{Z}$. Hence
$(\mathcal{Z},\Sigma(\mathcal{Z}))$ forms a measurable space. For
this measurable space, we denote the set of probability measures on
$(\mathcal{Z},\Sigma(\mathcal{Z}))$ as $\mathcal{P}(\mathcal{Z})$.
For a Polish space $\mathcal{Z}$, if $d$ is a complete metric that
induces the topology on this space, then $(\mathcal{Z},d)$ is called
a Polish metric space. For a Polish space $\mathcal{Z}$, if we equip
$\mathcal{P}(\mathcal{Z})$ with the weak topology, then the resultant
space is Polish as well. For brevity, we denote it as $(\mathcal{P}(\mathcal{Z}),\Sigma(\mathcal{P}(\mathcal{Z})))$.

As mentioned at the beginning of the introduction, $\mathcal{X}$
and $\mathcal{Y}$ are Polish spaces, and $P_{X}$ and $P_{Y}$ are
two probability measures defined respectively on $\mathcal{X}$ and
$\mathcal{Y}$. We also use $Q_{X},R_{X}$ to denote another two probability
measures on $\mathcal{X}$. The probability measures $P_{X},Q_{X},R_{X}$
can be thought as the push-forward measures (or the distributions)
induced jointly by the same measurable function $X$ (random variable)
from an underlying measurable space to $\mathcal{X}$ and by different
probability measures $\P,\Q,\R$ defined on the underlying measurable
space. Without loss of generality, we assume that $X$ is the identity
map, and $\P,\Q,\R$ are the same as $P_{X},Q_{X},R_{X}$. So, $P_{X},Q_{X},R_{X}$
could be independently specified to arbitrary probability measures.
We say that all probability measures induced by the underlying measure
$\P$, together with the corresponding measurable spaces, constitute
the $\P$-system. So, $P_{X}$ is in fact the distribution of the
random variable $X$ in the $\P$-system, where the letter ``$P$''
in the notation $P_{X}$ refers to the system and the subscript ``$X$''
refers to the random variable. When emphasizing the random variables,
we write $X\sim P_{X}$ to indicate that $X$ follows the distribution
$P_{X}$ in the $\P$-system. For a random variable (a measurable
function) $f$ from $\mathcal{X}$ to another measurable space $\mathcal{Z}$,
the distribution $P_{f(X)}$ of $f$ in different systems is clearly
different, e.g., it is $P_{X}\circ f^{-1}$ in the $\P$-system, but
it is $Q_{X}\circ f^{-1}$ in the $\Q$-system.

We use $P_{X}\otimes P_{Y}$ to denote the product of $P_{X}$ and
$P_{Y}$, and $P_{X}^{\otimes n}$ (resp. $P_{Y}^{\otimes n}$) to
denote the $n$-fold product of $P_{X}$ (resp. $P_{Y}$). For a probability
measure $P_{X}$ and a transition probability measure (or Markov kernel)
$P_{Y|X}$ from $\mathcal{X}$ to $\mathcal{Y}$, we denote $P_{X}P_{Y|X}$
as the joint probability measure induced by $P_{X}$ and $P_{Y|X}$.
Here $P_{Y|X}$ is called the regular conditional distribution of
$P_{X}P_{Y|X}$. We denote $P_{Y}$ or $P_{X}\circ P_{Y|X}$ as the
marginal distribution on $Y$ of the joint distribution $P_{X}P_{Y|X}$.
Moreover, we can pick up probability measures or transition probabilities
from different probability systems to constitute a joint probability
measure, e.g., $P_{X}Q_{Y|X}$. For a distribution $P_{X}$ on $\mathcal{X}$
and a measurable subset $A\subseteq\mathcal{X}$, $P_{X}(\cdot|A)$
denotes the conditional probability measure given $A$. For brevity,
we write $P_{X}(x):=P_{X}(\{x\}),x\in\mathcal{X}$. In particular,
if $X\sim P_{X}$ is discrete, the restriction of $P_{X}$ to the
set of singletons corresponds to the probability mass function of
$X$ in the $\P$-system. We denote $x^{n}=(x_{1},x_{2},\cdots,x_{n})\in\mathcal{X}^{n}$
as a sequence in $\mathcal{X}^{n}$. Given $x^{n}$, denote $x_{i}^{k}=(x_{i},x_{i+1},\cdots,x_{k})$
as a subsequence of $x^{n}$ for $1\le i\le k\le n$, and $x^{k}:=x_{1}^{k}$.
For a probability measure $P_{X^{n}}$ on $\mathcal{X}^{n}$, we use
$P_{X_{k}|X^{k-1}}$ to denote the regular conditional distribution
of $X_{k}$ given $X^{k-1}$ induced by $P_{X^{n}}$. For a measurable
function $f:\mathcal{X}\to\mathbb{R}$, sometimes we adopt the notation
$P_{X}(f)=\int_{\mathcal{X}}f\ \mathrm{d}P_{X}$.

Given $n\geq1$, the empirical measure (also known as type for the
finite alphabet case in information theory \cite{Csi97,Dembo}) for
a sequence $x^{n}\in\mathcal{X}^{n}$ is 
\[
\L_{x^{n}}:=\frac{1}{n}\sum_{i=1}^{n}\delta_{x_{i}}
\]
where $\delta_{x}$ is Dirac mass at the point $x\in\mathcal{X}$.
Let $\L_{n}:x^{n}\in\mathcal{X}^{n}\mapsto\L_{x^{n}}\in\calP(\mathcal{X})$
be the empirical measure map. For a pair of sequences $(x^{n},y^{n})\in\mathcal{X}^{n}\times\mathcal{Y}^{n}$,
the empirical joint measure $\L_{x^{n},y^{n}}$ and empirical conditional
measure $\L_{y^{n}|x^{n}}$ are defined similarly. Obviously, empirical
measures (or empirical joint measures) for $n$-length sequences are
discrete distributions whose probability masses are multiples of $1/n$.

\subsubsection{Information Theory }

For two distributions $P,Q$ defined on the same space, the relative
entropy {[}or Kullback-Leibler (KL) divergence{]} of $Q$ from $P$
is defined as 
\[
D(Q\|P):=\begin{cases}
\int\log(\frac{\mathrm{d}Q}{\mathrm{d}P})\mathrm{d}Q, & P\ll Q\\
\infty, & \textrm{otherwise}
\end{cases}.
\]
For brevity, we denote binary relative entropy function $D(p\|q):=D(\Bern(p)\|\Bern(q))$
where $p,q\in[0,1]$. Define the conditional version as $D(Q_{X|W}\|P_{X|W}|Q_{W}):=D(Q_{X|W}Q_{W}\|P_{X|W}Q_{W})$.

We use $B_{\delta}(x):=\{x'\in\mathcal{X}:d(x,x')<\delta\}$ and $B_{\delta]}(x):=\{x'\in\mathcal{X}:d(x,x')\leq\delta\}$
to respectively denote an open ball and a closed ball. We use $\overline{A}$,
$A^{o}$, and $A^{c}:=\mathcal{X}\backslash A$ to respectively denote
the closure, interior, and complement of the set $A\subseteq\mathcal{X}$.
Denote the Lévy--Prokhorov metric on $\mathcal{P}(\mathcal{X})$
as 
\begin{align*}
\Levy(Q_{X}',Q_{X}) & =\inf\{\delta>0:Q_{X}'(A)\le Q_{X}(A_{\delta})+\delta,\\
 & \qquad\qquad\forall\textrm{ closed }A\subseteq\mathcal{X}\}
\end{align*}
with $A_{\delta}:=\bigcup_{x\in A}\{x'\in\mathcal{X}:d(x,x')<\delta\}$,
which is compatible with the weak topology for the Polish metric space
$(\mathcal{X},d)$. Denote the total variation (TV) distance as 
\[
\|Q_{X}'-Q_{X}\|_{\mathrm{TV}}:=\sup_{A}\left\{ Q_{X}'(A)-Q_{X}(A)\right\} ,
\]
where the supremum is taken over all measurable $A$ in $\mathcal{P}(\mathcal{X})$.
The supremum here is in fact a maximum. Denote the sublevel set of
the relative entropy (or the divergence ``ball'') as $D_{\epsilon]}(P_{X}):=\{Q_{X}:D(Q_{X}\|P_{X})\le\epsilon\}$
for $\epsilon\ge0$. The Lévy--Prokhorov metric, the TV distance,
and the relative entropy admit the following relation:\footnote{Here, $\frac{1}{2}$ in \eqref{eq:D-TV-L} should be replaced by $\frac{1}{2\log e}$
if the base of the logarithm in the relative entropy is not $e$.
Accordingly, $2\epsilon^{2}$ in \eqref{eq:D-B} should be replaced
by $2\epsilon^{2}\log e$. } For any $Q,P\in\mathcal{P}(\mathcal{X})$, 
\begin{equation}
\Levy(Q,P)\le\|Q-P\|_{\mathrm{TV}}\le\sqrt{\frac{1}{2}D(Q\|P)},\label{eq:D-TV-L}
\end{equation}
which implies for $\epsilon\ge0$, 
\begin{equation}
B_{\epsilon]}(P)\supseteq D_{2\epsilon^{2}]}(P).\label{eq:D-B}
\end{equation}
The first inequality in \eqref{eq:D-TV-L} follows by definition \cite{gibbs2002choosing},
and the second inequality is known as Pinsker's inequality.

For a Polish space $\mathcal{X}$ and an empirical measure $T$ of
an $n$-length sequence in $\mathcal{X}^{n}$, $\L_{n}^{-1}(T)$ is
called the empirical class of $T$. When $\mathcal{X}$ is finite,
an empirical class is also called a type class \cite{Csis00}. For
a Polish space $\mathcal{X}$ and $\epsilon>0$, the empirically $\epsilon$-typical
set of $P$ \cite{mitran2015on} is defined as 
\begin{align}
 & \mathcal{T}_{\epsilon}^{(n)}(P):=\L_{n}^{-1}(B_{\epsilon]}(P)),\label{eq:typ_set}
\end{align}
where $B_{\epsilon]}(P)$ denotes the closed ball of center $P$ and
radius $\epsilon$ under the Lévy--Prokhorov metric. Since the empirical
measure map $\L_{n}$ is continuous under the weak topology, $\mathcal{T}_{\epsilon}^{(n)}(P)$
is closed in $\mathcal{X}^{n}$. Moreover, by Sanov's theorem \cite{Dembo},
the empirically typical set is a high probability set under the product
measure $P^{\otimes n}$. When $\mathcal{X}$ is finite and equips
with a Hamming metric, the Lévy--Prokhorov metric reduces to the
TV distance. So, for this case, 
\begin{align}
 & \mathcal{T}_{\epsilon}^{(n)}(P)=\big\{ x^{n}\in\mathcal{X}^{n}:\sum_{a\in\mathcal{X}}|\L_{x^{n}}(a)-P(a)|\leq2\epsilon\big\}.\label{eqn:typ_set-1}
\end{align}
For a transition probability measure $P_{X|W}$ from a finite set
$\mathcal{W}$ to a Polish space $\mathcal{X}$ and for $\epsilon>0$,
denote $B_{\epsilon]}(P_{X|W}):=\{R_{X|W}:R_{X|W=w}\in B_{\epsilon]}(P_{X|W=w}),\forall w\in\mathcal{W}\}$,
which is a closed ball of radius $\epsilon$ in $\mathcal{P}(\mathcal{X}\times\mathcal{Y}|\mathcal{W})$
equipped with the metric $(R_{X|W},P_{X|W})\mapsto\max_{w}\Levy(R_{X|W=w},P_{X|W=w})$.
Given a sequence $w^{n}$, define the conditional empirically $\epsilon$-typical
set{} of $P_{X|W}$ w.r.t. $w^{n}$ as 
\begin{align*}
 & \mathcal{T}_{\epsilon}^{(n)}(P_{X|W}|w^{n})\\
 & :=\L_{n}^{-1}(B_{\epsilon]}(P_{X|W})|w^{n})\\
 & =\big\{ x^{n}\in\mathcal{X}^{n}:\L_{x^{n}|w^{n}}(\cdot|b)\in B_{\epsilon]}(P_{X|W=b}),\forall b\in\mathcal{W}\big\}.
\end{align*}

For $(X,Y)\sim Q_{XY}$, the mutual information between $X$ and $Y$
is denoted as $I_{Q}(X;Y)=D(Q_{XY}\|Q_{X}\otimes Q_{Y})$. Denote
the conditional mutual information as 
\[
I_{Q}(X;Y|W)=\mathbb{E}_{Q_{W}}[D(Q_{XY|W}\|Q_{X|W}\otimes Q_{Y|W})].
\]
For discrete random variables $(X,Y)\sim Q_{XY}$, the (Shannon) entropy
\[
H_{Q}(X)=-\sum_{x}Q_{X}(x)\log Q_{X}(x),
\]
and the conditional (Shannon) entropy 
\[
H_{Q}(X|Y)=-\sum_{x,y}Q_{XY}(x,y)\log Q_{X|Y}(x|y).
\]
For brevity, we denote the binary entropy function $H(p):=H_{\Bern(p)}(X)=-p\log p-(1-p)\log(1-p)$
where $p\in[0,1]$. In fact, for discrete random variables, $I_{Q}(X;Y)=H_{Q}(X)-H_{Q}(X|Y)$. 

\subsubsection{Optimal Transport }

In this paper, our results involve the OT cost functional, which is
introduced now. The coupling set of $(P_{X},P_{Y})$ is defined as
\[
\cC(P_{X},P_{Y}):=\Biggl\{\begin{array}{l}
P_{XY}\in\mathcal{P}(\mathcal{X}\times\mathcal{Y}):\,\\
P_{XY}(A\times\mathcal{Y})=P_{X}(A),\forall A\in\Sigma(\mathcal{X}),\\
P_{XY}(\mathcal{X}\times B)=P_{Y}(B),\forall B\in\Sigma(\mathcal{Y})\:
\end{array}\Biggr\}.
\]
Distributions in $\cC(P_{X},P_{Y})$ are termed couplings of $(P_{X},P_{Y})$.
The OT cost between $P_{X}$ and $P_{Y}$ is defined as\footnote{The existence of the minimizers are well-known; see, e.g., \cite[Theorem 1.3]{villani2003topics}.
Furthermore, when the (joint) distribution of the random variables
involved in an expectation is clear from context, we will omit the
subscript ``$(X,Y)\sim P_{XY}$''.} 
\begin{equation}
\C(P_{X},P_{Y}):=\min_{P_{XY}\in\cC(P_{X},P_{Y})}\mathbb{E}_{(X,Y)\sim P_{XY}}[c(X,Y)].\label{eq:OT}
\end{equation}
Any $P_{XY}\in\cC(P_{X},P_{Y})$ attaining $\C(P_{X},P_{Y})$ is called
an OT plan. The minimization problem in \eqref{eq:OT} is called the
Monge--Kantorovich's OT problem \cite{villani2003topics}. The functional
$(P_{X},P_{Y})\in\mathcal{P}(\mathcal{X})\times\mathcal{P}(\mathcal{Y})\mapsto\C(P_{X},P_{Y})\in[0,+\infty)$
is called the OT (cost) functional. If $\mathcal{X}=\mathcal{Y}$,
$d$ is a complete metric that induces the topology on this space
{[}i.e., $(\mathcal{X},d)$ is a Polish metric space{]}, and $c=d^{p}$
with $p\ge1$, then $\W_{p}(P_{X},P_{Y}):=(\C(P_{X},P_{Y}))^{1/p}$
is the so-called $p$-th Wasserstein metric between $P_{X}$ and $P_{Y}$.
For the $n$-dimensional case, $\W_{p}(P_{X^{n}},P_{Y^{n}}):=(\C(P_{X^{n}},P_{Y^{n}}))^{1/p}$
with $c_{n}(x^{n},y^{n})=\sum_{i=1}^{n}d^{p}(x_{i},y_{i})$ is the
$p$-th Wasserstein metric between $P_{X^{n}}$ and $P_{Y^{n}}$ for
the product metric $d_{n}(x^{n},y^{n})=c_{n}(x^{n},y^{n})^{1/p}$
where $p\ge1$.

Furthermore, for another distribution $P_{W}$ on a Polish space $\mathcal{W}$,
the conditional coupling set of transition probability measures $P_{X|W}$
and $P_{Y|W}$ is defined as 
\begin{align*}
 & \cC(P_{X|W},P_{Y|W})\\
 & :=\Biggl\{\begin{array}{l}
P_{XY|W}\in\mathcal{P}(\mathcal{X}\times\mathcal{Y}|\mathcal{W}):\\
P_{XY|W=w}\in\cC(P_{X|W=w},P_{Y|W=w}),\\
\qquad\forall w\in\mathcal{W}\:
\end{array}\Biggr\},
\end{align*}
where $\mathcal{P}(\mathcal{X}\times\mathcal{Y}|\mathcal{W})$ denotes
the set of transition probability measures from $\mathcal{W}$ to
$\mathcal{X}\times\mathcal{Y}$. The conditional OT cost between transition
probability measures $P_{X|W}$ and $P_{Y|W}$ given $P_{W}$ is defined
as 
\begin{align}
 & \C(P_{X|W},P_{Y|W}|P_{W})\nonumber \\
 & :=\min_{P_{XY|W}\in\cC(P_{X|W},P_{Y|W})}\mathbb{E}_{(W,X,Y)\sim P_{W}P_{XY|W}}[c(X,Y)],\label{eq:OT-2-1}
\end{align}
where $P_{W}P_{XY|W}$ denotes the joint probability measure induced
by $P_{W}$ and $P_{XY|W}$. The conditional OT cost can be alternatively
expressed as\footnote{In other words, the minimization in \eqref{eq:OT-2-1} can be taken
in a pointwise way for each $w$. For optimal $P_{XY}^{(w)}$ attaining
$\C(P_{X|W=w},P_{Y|W=w})$, the measurability of $w\mapsto P_{XY}^{(w)}(B),\,B\in\Sigma(\mathcal{X}\times\mathcal{Y})$
can be addressed by measurable selection theorems, e.g., \cite[Proposition 7.50]{bertsekas1996stochastic}.} 
\begin{equation}
\C(P_{X|W},P_{Y|W}|P_{W})=\mathbb{E}_{P_{W}}[\C(P_{X|W},P_{Y|W})].\label{eq:OT-1}
\end{equation}

\subsubsection{Others}

We use $f(n)=o_{n}(1)$ to denote that $f(n)\to0$ pointwise as $n\to+\infty$.
When there is no specification, by default, we denote $\inf\emptyset:=+\infty,\;\sup\emptyset:=-\infty$,
and $[k]:=\{1,2,...,k\}$. Denote $\conv{g}$ as the lower convex
envelope of a function $g$, and $\conc g$ as the upper concave envelope
of $g$.


\section{\label{sec:Main-Results}Main Results }

\subsection{Asymptotic Concentration Exponent}

\subsubsection{\label{subsec:General-Cost}General Cost }

We now characterize the asymptotic concentration exponent $\lim_{n\to\infty}E_{1}^{(n)}(\alpha,\tau)$.
To this end, given $P_{X},P_{Y}$, and $c$, we define 
\begin{align}
\phi(\alpha,\tau) & :=\inf_{\substack{Q_{X}\in\mathcal{P}(\mathcal{X}),Q_{Y}\in\mathcal{P}(\mathcal{Y}):\\
D(Q_{X}\|P_{X})\le\alpha,\C(Q_{X},Q_{Y})>\tau
}
}D(Q_{Y}\|P_{Y}).\label{eq:phi}
\end{align}
Denote $\conv{\phi}(\alpha,\tau)$ as the lower convex envelope of
$\phi(\alpha,\tau)$, which can be also expressed as 
\begin{equation}
\conv{\phi}(\alpha,\tau)=\inf_{\substack{Q_{X|W},Q_{Y|W},Q_{W}:\\
D(Q_{X|W}\|P_{X}|Q_{W})\le\alpha,\\
\C(Q_{X|W},Q_{Y|W}|Q_{W})>\tau
}
}D(Q_{Y|W}\|P_{Y}|Q_{W}),\label{eq:phi_lce-1}
\end{equation}
where $W$ is an auxiliary random variable defined on a Polish space.
However, by Carathéodory's theorem, the alphabet size of $Q_{W}$
can be restricted to be no larger than $4$. In fact, the alphabet
size can be further restricted to be no larger than $3$, since it
suffices to consider the boundary points of the convex hull of 
\begin{align*}
 & \big\{\big(D(Q_{X|W=w}\|P_{X}),D(Q_{Y|W=w}\|P_{Y}),\\
 & \qquad\qquad\C(Q_{X|W=w},Q_{Y|W=w})\big)\big\}_{w\in\mathcal{W}}.
\end{align*}

To characterize the asymptotic concentration exponent, we need the
following assumption. Define the $(\mathcal{X},\epsilon)$-smooth
 OT functional as 
\[
\C_{\mathcal{X},\epsilon}(Q_{X},Q_{Y}):=\inf_{Q_{X}':\Levy(Q_{X},Q_{X}')\le\epsilon}\C(Q_{X}',Q_{Y}).
\]
By definition, $\C_{\mathcal{X},\epsilon}(Q_{X},Q_{Y})\le\C_{\mathcal{X},0}(Q_{X},Q_{Y})=\C(Q_{X},Q_{Y})$,
and by the lower semicontinuity of the OT functional, $\lim_{\epsilon\downarrow0}\C_{\mathcal{X},\epsilon}(Q_{X},Q_{Y})\ge\C(Q_{X},Q_{Y})$.
So, $\lim_{\epsilon\downarrow0}\C_{\mathcal{X},\epsilon}(Q_{X},Q_{Y})=\C(Q_{X},Q_{Y})$
pointwise. 

\textbf{Assumption 1} (Uniform Convergence of $(\mathcal{X},\epsilon)$-Smooth
OT Functional): We assume that there is a function $\delta(\epsilon):(0,\infty)\to(0,\infty)$
vanishing as $\epsilon\downarrow0$ such that 
\[
\C_{\mathcal{X},\epsilon}(Q_{X},Q_{Y})\ge\C(Q_{X},Q_{Y})-\delta(\epsilon)
\]
holds for all $(Q_{X},Q_{Y})$. In other words, $\C_{\mathcal{X},\epsilon}(Q_{X},Q_{Y})\to\C(Q_{X},Q_{Y})$
as $\epsilon\downarrow0$ uniformly for all $(Q_{X},Q_{Y})$. 

Obviously, if the optimal transport cost functional $(Q_{X},Q_{Y})\mapsto\C(Q_{X},Q_{Y})$
is uniformly continuous under the Lévy--Prokhorov metric (which was
assumed by the author in \cite{yu2020asymptotics} in studying the
asymptotics of Strassen's optimal transport problem), then Assumption
1 holds. The following two examples satisfying Assumption 1 were provided
in \cite{yu2020asymptotics}. 
\begin{example}[Countable Alphabet and Bounded Cost]
\label{exa:1} $\mathcal{X}$ and $\mathcal{Y}$ are countable sets
and $c$ is bounded (i.e., $\sup_{x,y}c(x,y)<\infty$). 
\end{example}
\begin{example}[Wasserstein Metric Induced by a Bounded Metric\footnote{Example \ref{exa:2} satisfying Assumption 1 follows by the fact that
the Wasserstein metric induced by a bounded metric $d$ is equivalent
to the Lévy--Prokhorov metric in the sense that $\Levy^{p+1}\le\W_{p}^{p}\le\Levy^{p}+d_{\sup}^{p}\Levy$
where $d_{\sup}=\sup_{x,x'\in\mathcal{X}}d(x,x')$ is the diameter
of $\mathcal{X}$ \cite{gibbs2002choosing}.}]
\label{exa:2} $\mathcal{X}=\mathcal{Y}$ equipped with a bounded
metric $d$ is a Polish metric space, i.e., $\sup_{x,y}d(x,y)<\infty$.
The cost function is set to $c=d^{p}$ for $p\ge1$, and hence, $\C=\W_{p}^{p}$. 
\end{example}
The following theorem characterizes the asymptotic concentration exponent.
The proof is provided in Section \ref{sec:Proof-of-Theorem-1}. For
a function $f:[0,\infty)^{k}\to[0,\infty]$ with $k\ge1$, denote
the effective domain of $f$ as 
\[
\mathrm{dom}f=\left\{ x^{k}\in[0,\infty)^{k}:f(z)<\infty\right\} .
\]
By definition, $\mathrm{dom}\conv{f}=\mathrm{dom}\conc f=\mathrm{dom}f$
if $f$ is monotonous in each parameter (given others). 
\begin{thm}[Asymptotics of $E_{1}^{(n)}$ and Dimension-Free Bound\footnote{The terminology ``dimension-free bound'' here denotes that the tuple
of the normalized enlargement parameter $\tau$, the (normalized)
exponent of $P_{X}^{\otimes n}(A)$, and the (normalized) exponent
of $1-P_{Y}^{\otimes n}(A^{n\tau})$ verifies the same inequality
for all $n$. This concept is weaker than that in \cite{gozlan2009characterization}
and reduces to the latter when $P_{X}^{\otimes n}(A)$ is fixed to
be around $1/2$, $c$ is set to $d^{2}$, and the bound on the exponent
of $1-P_{Y}^{\otimes n}(A^{n\tau})$ in the inequality satisfied by
the tuple is the quadratic form. Hence, the ``dimension-free bound''
here could be satisfied by a much larger class of probability metric
spaces.}]
\label{thm:LD} For Polish $\mathcal{X}$ and $\mathcal{Y}$, the
following hold. 
\begin{enumerate}
\item For any $\alpha\ge0,\tau\ge0$ and any positive integer $n$, 
\begin{equation}
E_{1}^{(n)}(\alpha,\tau)\geq\conv{\phi}(\alpha,\tau).\label{eq:bound2}
\end{equation}
\item Under Assumption 1, for any $(\alpha,\tau)$ in the interior of $\mathrm{dom}\conv{\phi}$,
it holds that $\lim_{n\to\infty}E_{1}^{(n)}(\alpha,\tau)=\conv{\phi}(\alpha,\tau).$ 
\item Let $(a_{n})$ be a sequence such that $e^{-o(n)}\le a_{n}\le1-e^{-o(n)}$
(and hence $\alpha_{n}=-\frac{1}{n}\log a_{n}\to0$). Then, under
Assumption 1, it holds that for any $\tau$ in the interior of $\mathrm{dom}\conv{\varphi}$,
\begin{align*}
\lim_{\alpha\downarrow0}\conv{\phi}(\alpha,\tau) & \le\liminf_{n\to\infty}E_{1}^{(n)}(\alpha_{n},\tau)\\
 & \le\limsup_{n\to\infty}E_{1}^{(n)}(\alpha_{n},\tau)\le\conv{\varphi}(\tau),
\end{align*}
where 
\begin{equation}
\varphi(\tau):=\phi(0,\tau)=\inf_{Q_{Y}:\C(P_{X},Q_{Y})>\tau}D(Q_{Y}\|P_{Y}).\label{eq:varphi}
\end{equation}
\end{enumerate}
\end{thm}
The condition $e^{-o(n)}\le a_{n}\le1-e^{-o(n)}$ implies that the
sequence $(a_{n})$ does not approach $0$ or $1$ too fast, in the
sense that the sequence $(a_{n})$ is sandwiched between a sequence
that subexponentially approaches zero and a sequence that subexponentially
approaches one.

The expression $\conv{\phi}(\alpha,\tau)$ for the asymptotic concentration
exponent is elegant in the sense that it is expressed in terms of
two fundamental quantities from other fields---``relative entropy''
which comes from information theory (or large deviations theory) and
``optimal transport cost'' which comes from the theory of optimal
transport. Hence, this verifies an intimate connection among concentration
of measure, information theory, and optimal transport.

The first bound like the one in \eqref{eq:bound2} was derived by
Marton \cite{Marton86,marton1996bounding}, which was improved by
Gozlan and Léonard in \cite{gozlan2005principe,gozlan2007large}.
Our proof relies on the subadditivity of OT costs, instead of traditional
transport-entropy inequalities, leading to that our bound in \eqref{eq:bound2}
is strictly better than Gozlan and Léonard's especially when the measure
of the set is small. When $c=d^{p}$ and $\alpha$ is close to zero,
e.g., $\alpha=\frac{1}{n}\log2$ (i.e., $a=\frac{1}{2}$; recall the
relation $a=e^{-n\alpha}$ in \eqref{eq:alpha-tau}), our bound and
theirs do not differ too much, and as $n\to\infty$, they coincide
asymptotically. However, if $\alpha$ is bounded away from zero, our
bound is usually asymptotically tight but theirs are not. 

The bound in \eqref{eq:bound2} can be expressed as an exponentially
sharp version of Talagrand's concentration inequalities. Given $P_{X},P_{Y}$,
and $c$, we define for $\tau\ge0,\lambda\in[0,1]$, 
\begin{align}
\phi_{\lambda}(\tau) & :=\inf_{Q_{X},Q_{Y}:\C(Q_{X},Q_{Y})>\tau}(1-\lambda)D(Q_{Y}\|P_{Y})\nonumber \\
 & \qquad\qquad\qquad+\lambda D(Q_{X}\|P_{X}),\label{eq:phi-2}
\end{align}
which is a nonlinear variant of the transport-entropy inequalities
in \cite[Definition 4.1]{gozlan2017kantorovich}. Denote $\conv{\phi}_{\lambda}(\tau)$
as the lower convex envelope of $\phi_{\lambda}(\tau)$. 
\begin{cor}[Improved Talagrand's Concentration Inequality]
\label{thm:bound-1} For Polish $\mathcal{X}$ and $\mathcal{Y}$,
it holds that for any $\tau\ge0,\lambda\in[0,1]$, $t=n\tau,$ and
any $A$, 
\begin{equation}
P_{Y}^{\otimes n}((A^{t})^{c})^{1-\lambda}P_{X}^{\otimes n}(A)^{\lambda}\leq e^{-n\conv{\phi}_{\lambda}(\tau)},\label{eq:-5}
\end{equation}
where $\conv{\phi}_{\lambda}$ can be alternatively expressed as 
\begin{equation}
\conv{\phi}_{\lambda}(\tau)=\inf_{\alpha\ge0}\lambda\alpha+(1-\lambda)\conv{\phi}(\alpha,\tau).\label{eq:-36}
\end{equation}
Moreover, under Assumption 1 and given any $\tau$ which together
with the optimal $\alpha$ attaining the infimum in \eqref{eq:-36}
is in the interior of $\mathrm{dom}\conv{\phi}$, the inequality in
\eqref{eq:-5} is exponentially sharp in the sense that there is a
sequence of sets $A_{n}$ such that the induced exponents of two sides
asymptotically coincide. 
\end{cor}
\begin{rem}
The kind of inequalities like the one in \eqref{eq:-5} are the so-called
Talagrand's concentration inequalities; see a weaker version for Hamming
metric in \cite[p. 86]{talagrand1995concentration}. An inequality
weaker than the one in \eqref{eq:-5} was proven by Gozlan et al.
\cite{gozlan2017kantorovich} in which linear bounds on $\phi_{\lambda}(\tau)$,
instead of $\phi_{\lambda}(\tau)$ itself, were applied in the proof.
 
\end{rem}
\begin{rem}
The function $\phi_{\lambda}$ suggests a new and more general class
of transport-entropy inequalities, which plays the same role in our
proof of Theorem \ref{thm:LD} as the traditional transport-entropy
inequalities in Marton's proof \cite{Marton86,marton1996bounding}.
 
\end{rem}
\begin{IEEEproof}
It holds that 
\begin{align}
 & -\frac{1}{n}\log\left(P_{Y}^{\otimes n}((A^{t})^{c})^{1-\lambda}P_{X}^{\otimes n}(A)^{\lambda}\right)\nonumber \\
 & \geq\inf_{\alpha\ge0}\lambda\alpha+(1-\lambda)\conv{\phi}(\alpha,\tau)\label{eq:-6}\\
 & =\inf_{\substack{\alpha\ge0,Q_{X|W},Q_{Y|W},Q_{W}:\\
D(Q_{X|W}\|P_{X}|Q_{W})\le\alpha,\\
\C(Q_{X|W},Q_{Y|W}|Q_{W})>\tau
}
}\lambda\alpha+(1-\lambda)D(Q_{Y|W}\|P_{Y}|Q_{W})\nonumber \\
 & =\inf_{\substack{Q_{X|W},Q_{Y|W},Q_{W}:\\
\C(Q_{X|W},Q_{Y|W}|Q_{W})>\tau
}
}\lambda D(Q_{X|W}\|P_{X}|Q_{W})\nonumber \\
 & \qquad\qquad+(1-\lambda)D(Q_{Y|W}\|P_{Y}|Q_{W})\\
 & =\conv{\phi}_{\lambda}(\tau).\nonumber 
\end{align}
From the alternative expression of $\conv{\phi}_{\lambda}(\tau)$
in \eqref{eq:-6} and for each $\alpha$, choosing $\lambda$ such
that $\frac{-\lambda}{1-\lambda}$ is a subgradient of $\alpha'\mapsto\conv{\phi}(\alpha',\tau)$
at $\alpha$, we obtain the inequality in \eqref{eq:bound2} from
the inequality in \eqref{eq:-6} (or equivalently, the one in \eqref{eq:-5}).
Hence, Theorem \ref{thm:bound-1} is in fact equivalent to the bound
in \eqref{eq:bound2}, and the exponential sharpness of \eqref{eq:-5}
is equivalent to the asymptotic tightness of \eqref{eq:bound2}. 
\end{IEEEproof}

\subsubsection{Complete Metric }

An interesting special case is that $(\mathcal{X},P_{X})$ and $(\mathcal{Y},P_{Y})$
are the same Polish probability space and the cost function $c$ is
set to $d^{p}$ with $p\ge1$ and $d$ denoting a complete metric
compatible with the topology on this space. In other words, $\C=\W_{p}^{p}$.
For this case, we now remove Assumption 1 from Theorem \ref{thm:LD}.
Furthermore, to further simplify Statement 3 of Theorem \ref{thm:LD},
we need the following Assumption 2.

\textbf{Assumption 2} (Positivity Condition): $\conv{\varphi}_{X}(\tau)$
is strictly positive for all sufficiently small (equivalently for
all) $\tau>0$, where 
\begin{equation}
\varphi_{X}(\tau):=\inf_{Q_{X}:\C(P_{X},Q_{X})>\tau}D(Q_{X}\|P_{X}).\label{eq:varphi_X}
\end{equation}
In particular, if the cost $c$ is set to $d^{p}$ with $p\ge1$
and $d$ denoting a metric, and define 
\begin{equation}
\varphi_{X,\ge}(\tau):=\inf_{Q_{X}:\C(P_{X},Q_{X})\ge\tau}D(Q_{X}\|P_{X}),\label{eq:varphi_X-2}
\end{equation}
then the assumption is equivalent to saying that $\conv{\varphi}_{X,\ge}(\tau)$
is strictly increasing in $\tau\ge0$ (since $\conv{\varphi}_{X,\ge}(0)=0$).

An equivalent statement of Assumption 2 is that given $P_{X}$, if
$\C(P_{X},Q_{X})$ is bounded away from zero, then so is $D(Q_{X}\|P_{X})$.
In other words, given $P_{X}$, convergence in information (i.e.,
$D(Q_{X}\|P_{X})\to0$) implies convergence in optimal transport (i.e.,
$\C(P_{X},Q_{X})\to0$). 
\begin{thm}[Asymptotics of $E_{1}^{(n)}$ for Complete Metrics]
\label{thm:LD-2} Assume that $\mathcal{X}=\mathcal{Y}$ equipped
with a metric $d$ is a Polish metric space, and the cost function
is set to $c=d^{p}$ for $p\ge1$. Then, the following hold. 
\begin{enumerate}
\item For any $\alpha,\tau\ge0$ and any positive integer $n$, $E_{1}^{(n)}(\alpha,\tau)\geq\conv{\phi}(\alpha,\tau).$ 
\item For any $(\alpha,\tau)$ in the interior of $\mathrm{dom}\conv{\phi}$,
it holds that $\lim_{n\to\infty}E_{1}^{(n)}(\alpha,\tau)=\conv{\phi}(\alpha,\tau).$ 
\item Let $(a_{n})$ be a sequence such that $e^{-o(n)}\le a_{n}\le1-e^{-o(n)}$
(and hence $\alpha_{n}=-\frac{1}{n}\log a_{n}\to0$). Then, for any
$\tau$ in the interior of $\mathrm{dom}\conv{\varphi}$, it holds
that 
\begin{equation}
\limsup_{n\to\infty}E_{1}^{(n)}(\alpha_{n},\tau)\le\conv{\varphi}(\tau),\label{eq:-41}
\end{equation}
and under Assumption 2, 
\begin{equation}
\lim_{n\to\infty}E_{1}^{(n)}(\alpha_{n},\tau)=\conv{\varphi}(\tau),\label{eq:-40}
\end{equation}
where $\varphi_{X}$ is defined in \eqref{eq:varphi_X}. In particular,
for the case of $P_{X}=P_{Y}$, $\liminf_{n\to\infty}E_{1}^{(n)}(\alpha_{n},\tau)>0$
holds for all sufficiently small (equivalently for all) $\tau>0$
(i.e., exponential convergence) if and only if Assumption 2 holds. 
\end{enumerate}
\end{thm}
Theorem \ref{thm:LD-2} is a consequence of Theorem \ref{thm:LD}
and proven in Section \ref{sec:Proof-of-Theorem-1-1}. Statement 1
in Theorem \ref{thm:LD-2} is a restatement of Statement 1 in Theorem
\ref{thm:LD} for the case of $c=d^{p}$. Statement 3 is not new;
see Proposition 4.6 and Theorem 5.4 in \cite{gozlan2010transport}.
Statements 2 and 3 in Theorem \ref{thm:LD-2} might be proven alternatively
by the large deviation theorems on the Wassernstein metric in \cite{gozlan2009characterization,wang2010sanov}.
In fact, for this setting of $a=\frac{1}{2}$, Alon, Boppana, and
Spencer in \cite{alon1998asymptotic} provided an alternative expression
for $\lim_{n\to\infty}E_{1}^{(n)}(\alpha_{n},\tau)$ when $\mathcal{X}$
is finite (Assumption 2 automatically is satisfied for this case).
The equivalence between theirs and ours is discussed in details in
Section \ref{subsec:dual}.

By Talagrand's transport inequality, the function $\conv{\phi}$ can
be derived for the case of Gaussian distribution and Euclidean distance. 
\begin{example}[Gaussian Distribution and Euclidean Distance]
 For Gaussian distributions $P_{X}=P_{Y}=\mathcal{N}(0,1)$ and $c(x,y)=(x-y)^{2}$
(with $p=2$), the function 
\[
\conv{\phi}(\alpha,\tau)=\phi(\alpha,\tau)=\begin{cases}
\frac{1}{2}\left(\sqrt{\tau}-\sqrt{2\alpha}\right)^{2}, & \tau>2\alpha\\
0, & \textrm{otherwise}
\end{cases}.
\]
Theorem \ref{thm:LD-2} for this case verifies a consequence of the
Gaussian isoperimetric inequality. 
\end{example}

\subsubsection{\label{subsec:Hamming-Metric}Hamming Metric}

The Hamming metric was one of the metrics first considered in the
field of concentration of measure; see, e.g., \cite{margulis1974probabilistic,Ahls76,Marton86,talagrand1995concentration}.
Note that a countable space with the Hamming metric must be a Polish
metric space, but an uncountable space with the Hamming metric must
not be a Polish metric space. In fact, even so, we next show that
the asymptotics of the concentration exponent in the latter case is
still $\phi(\alpha,\tau)$ for any $\alpha>0,\tau\in(0,1)$. Let $\mathcal{X}=\mathcal{Y}$
be Polish space and let $c$ be the Hamming metric, i.e., $c(x,y)=\bone_{\{x\neq y\}}$.
By the Kantorovich duality, the OT cost in this case is equal to the
TV distance $\|Q_{X}-Q_{Y}\|_{\mathrm{TV}}=\sup_{A}Q_{X}(A)-Q_{Y}(A)$,
with the supremum here attained by $A=\big\{ x:\d Q_{X}/\d R(x)>\d Q_{Y}/\d R(x)\big\}$
where $R$ is an arbitrary probability measure such that $Q_{X},Q_{Y}\ll R$.
Define for $(p,q)\in[0,1]^{2}$, 
\[
\theta(p,q):=\theta_{\alpha,\tau}(p,q):=\inf_{s,t\in[0,1]:D(s\|p)\le\alpha,\,s-t>\tau}D(t\|q).
\]
Here recall that $D(p\|q):=D(\Bern(p)\|\Bern(q))$ for $(p,q)\in[0,1]^{2}$.
For $p\in(0,1)$, denote $s^{*}(p)$ as the solution in $[p,1]$ to
the equation $D(s\|p)=\alpha$ with $s$ unknown; denote $s^{*}(p)=1$
if there is no such solution. For $p\in\{0,1\}$, denote $s^{*}(p)=p$.
It is easy to see that $s^{*}(p)$ is nondecreasing in $p$. Then,
\begin{equation}
\theta(p,q)=\begin{cases}
0 & q\le s^{*}(p)-\tau\\
D(s^{*}(p)-\tau\|q) & q>s^{*}(p)-\tau>0\\
\infty & s^{*}(p)-\tau\le0
\end{cases}.\label{eq:-82}
\end{equation}

\begin{thm}[Asymptotics of $E_{1}^{(n)}$ for Hamming Metric]
\label{thm:LD-3} Assume that $\mathcal{X}=\mathcal{Y}$ is a Polish
space and $c$ is the Hamming metric, i.e., $c(x,y)=\bone_{\{x\neq y\}}$.
The following hold.  
\begin{enumerate}
\item For any $\alpha\ge0,\tau\in[0,1]$, it holds that $E_{1}^{(n)}(\alpha,\tau)\ge\phi(\alpha,\tau)$,
where 
\begin{align}
\phi(\alpha,\tau) & =\inf_{A}\theta_{\alpha,\tau}(P_{X}(A),P_{Y}(A)).\label{eq:-64}
\end{align}
In particular, if $P_{X}$ is finitely-supported or atomless, then
\begin{align}
\phi(\alpha,\tau) & =\inf_{p\in[0,1]:\omega(p)<\infty}\theta_{\alpha,\tau}(p,\omega(p)),\label{eq:-64-1}
\end{align}
where 
\begin{equation}
\omega(p):=\inf_{A:P_{X}(A)=p}P_{Y}(A).\label{eq:omega}
\end{equation}
 
\item For any $\alpha>0,\tau\in(0,1]$, it holds that $\lim_{n\to\infty}E_{1}^{(n)}(\alpha,\tau)=\phi(\alpha,\tau).$ 
\item Let $(a_{n})$ be a sequence such that $e^{-o(n)}\le a_{n}\le1-e^{-o(n)}$
(and hence $\alpha_{n}=-\frac{1}{n}\log a_{n}\to0$). Then, it holds
that for any $\tau\in(0,1]$, 
\[
\lim_{n\to\infty}E_{1}^{(n)}(\alpha_{n},\tau)=\varphi(\tau),
\]
where $\varphi(\tau)=\phi(0,\tau)$. 
\end{enumerate}
\end{thm}
 
\begin{rem}
This theorem implies that for the Hamming metric, the asymptotic concentration
exponent for a pair of arbitrary distributions $(P_{X},P_{Y})$ is
the same as that of $(\Bern(p),\Bern(q))$, some quantized version
of $(P_{X},P_{Y})$. 
\end{rem}
In fact, given an arbitrary $P_{X}$, it can be obtained from \eqref{eq:-64}
that 
\begin{align}
\phi(\alpha,\tau) & \ge\inf_{p\in[0,1]:\omega(p)<\infty}\theta(p,\omega(p))\ge\inf_{p\in[0,1]}\theta(p,\conv\omega(p)).\label{eq:-56-1}
\end{align}
Compared with determining the function $\omega$ itself, it is much
easier to determine $\conv\omega$, since by the Neyman--Pearson
lemma, the graph of $\conv\omega$ coincides with the lower convex
envelope of the curve $\{(P_{X}(A_{r}),P_{Y}(A_{r})):r\ge0\}$, where
$A_{r}:=\big\{ x:\d P_{Y}/\d R(x)\le r\d P_{X}/\d R(x)\big\}$ with
$R$ denoting an arbitrary probability measure such that $P_{X},P_{Y}\ll R$.
Moreover, $\omega$ coincides with $\conv\omega$ if $P_{X}$ is atomless,
and for this case, the lower bound in \eqref{eq:-56-1} is tight,
as shown in \eqref{eq:-64-1}. 

For the case of $P_{X}=P_{Y}$, 
\begin{align}
\phi(\alpha,\tau) & =\inf_{A}\theta(P_{X}(A))\label{eq:-56-2}\\
 & \ge\underline{\phi}(\alpha,\tau):=\inf_{p\in[0,1]}\theta(p),\label{eq:-79}
\end{align}
where 
\begin{align*}
\theta(p) & :=\theta(p,p)\\
 & =\inf_{s,t\in[0,1]:D(s\|p)\le\alpha,\,s-t>\tau}D(t\|p)\\
 & =\begin{cases}
0 & p\le s^{*}(p)-\tau\\
D(s^{*}(p)-\tau\|p) & p>s^{*}(p)-\tau>0\\
\infty & s^{*}(p)-\tau\le0
\end{cases}.
\end{align*}
By the convexity of the relative entropy, it is easy to see that $\underline{\phi}$
is convex. Moreover, the equality in \eqref{eq:-79} holds when $P_{X}=P_{Y}$
is atomless. Hence, for this case, $\conv\phi(\alpha,\tau)=\underline{\phi}(\alpha,\tau)$.
In other words, for any $\alpha>0,\tau\in(0,1)$, all atomless distributions
admit the same smallest asymptotic concentration exponent $\underline{\phi}(\alpha,\tau)$.
The graph of $\underline{\phi}$ is shown in Fig. \ref{fig:concentrationHamming}.
In particular, for the case in Statement 3 of Theorem \ref{thm:LD-3}
with $P_{X}=P_{Y}$, it was shown in \cite{vajda1970note} that $\varphi(\tau)\ge\min_{p\in[\tau,1]}D(p-\tau\|p)$,
with equality if $P_{X}=P_{Y}$ is atomless \cite{berend2014minimum}. 

\begin{figure}
\begin{centering}
\includegraphics[width=0.45\textwidth]{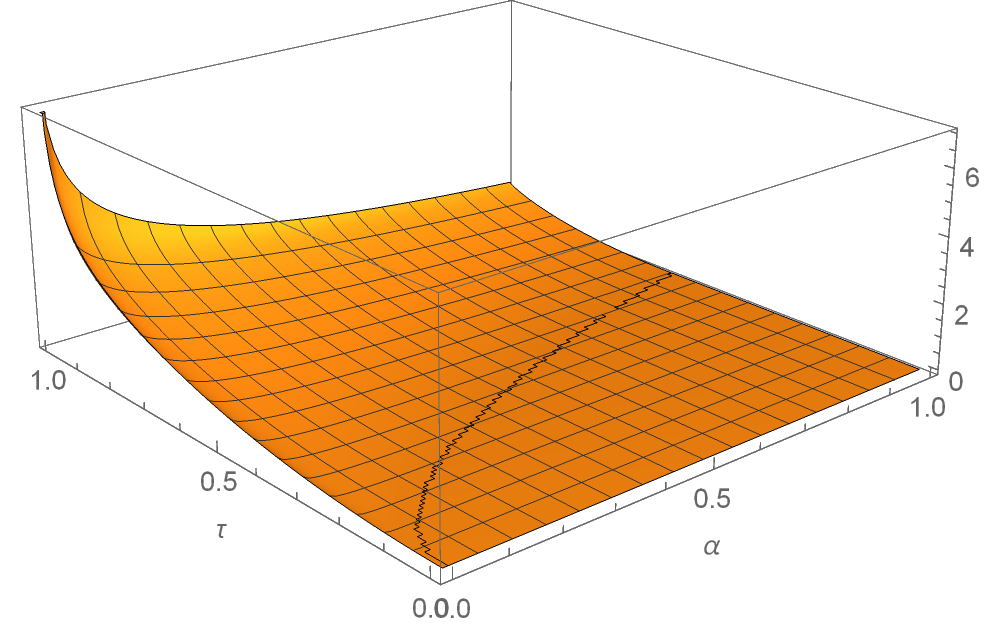} 
\par\end{centering}
\caption{\label{fig:concentrationHamming}Illustration of the function $\underline{\phi}$
which corresponds to the asymptotic concentration exponent of atomless
measures $P_{X}=P_{Y}$ under the Hamming metric. In this graph, the
bases of the logarithm and exponent are changed to $2$. Given each
$\alpha$, the function $\underline{\phi}$ is zero when $\tau$ is
smaller than the value on the black curve. (The curve looks not so
smooth in the figure due to the precision of computation, but should
be smooth in theory.)}
 
\end{figure}

\subsection{\label{subsec:Asymptotic-Isoperimetric-Exponen}Asymptotic Isoperimetric
Exponent }

We next derive the asymptotic expression of $E_{0}^{(n)}(\alpha,\tau)$.
Define 
\begin{align}
\psi(\alpha,\tau) & :=\sup_{Q_{XW}:D(Q_{X|W}\|P_{X}|Q_{W})\le\alpha}\nonumber \\
 & \qquad\inf_{Q_{Y|XW}:\mathbb{E}[c(X,Y)]\le\tau}D(Q_{Y|W}\|P_{Y}|Q_{W}),\label{eq:psi}
\end{align}
with the supremum taken over all $W$ defined on finite alphabets. 
\begin{thm}
\label{thm:The-alphabet-size}The alphabet size of $W$ in \eqref{eq:psi}
can be restricted to be no larger than $2$. 
\end{thm}
We will restate this theorem in Theorem \ref{thm:psi} in Section
\ref{subsec:dual}, and the proof of Theorem \ref{thm:psi} is provided
in Section \ref{sec:dual}. It is worth noting that bounding the alphabet
size of $W$ is not obvious as that for the function $\conv{\phi}$
in \eqref{eq:phi_lce-1}, since the auxiliary random variable $W$
here does no longer play the role of the convex combination in the
lower convex envelope. So, Carathéodory's theorem cannot be applied.
Instead, our proof of Theorem \ref{thm:The-alphabet-size} is based
on the dual expression of $\psi$. 

Based on $\psi$, the asymptotic expression of $E_{0}^{(n)}$ is characterized
in the following theorem. Define the $(\mathcal{X},\epsilon)$-smooth
 cost function w.r.t. $c$ as 
\[
c_{\mathcal{X},\epsilon}(x,y):=\inf_{x':d(x,x')\le\epsilon}c(x',y).
\]
By definition, $c_{\mathcal{X},\epsilon}(x,y)\le c_{\mathcal{X},0}(x,y)=c(x,y)$,
and by the lower semicontinuity of $c$, $\lim_{\epsilon\downarrow0}c_{\mathcal{X},\epsilon}(x,y)\ge c(x,y)$.
So, $\lim_{\epsilon\downarrow0}c_{\mathcal{X},\epsilon}(x,y)=c(x,y)$
pointwise. 

\textbf{Assumption 3} (Uniform Convergence of $(\mathcal{X},\epsilon)$-Smooth
Cost Function): We assume that there is a function $\delta(\epsilon):(0,\infty)\to(0,\infty)$
vanishing as $\epsilon\downarrow0$ such that 
\begin{equation}
c_{\mathcal{X},\epsilon}(x,y)\ge c(x,y)-\delta(\epsilon)\label{eq:-39}
\end{equation}
holds for all $(x,y)$. In other words, $c_{\mathcal{X},\epsilon}(x,y)\to c(x,y)$
as $\epsilon\downarrow0$ uniformly for all $(x,y)$. 

Assumption 3 is automatically satisfied if $\mathcal{X}=\mathcal{Y}$
and $c=d$. Moreover, Assumption 3 is implied by Assumption 1. By
choosing $Q_{X},Q_{X}',Q_{Y}$ as Dirac measures $\delta_{x},\delta_{x'},\delta_{y}$
in Assumption 1 and by the fact that $\Levy(\delta_{x},\delta_{x'})=d(x,x')$
when $d(x,x')\le1$, it is easy to verify that Assumption 3 holds
for this case.  
\begin{thm}[Asymptotics of $E_{0}^{(n)}$]
\label{thm:LD-1}{} Assume that $\mathcal{X}$ and $\mathcal{Y}$
are Polish spaces. Then the following hold. 
\begin{enumerate}
\item Assume that $c(x,y)\le c_{\mathcal{X}}(x)+c_{\mathcal{Y}}(y)$ for
some measurable functions $c_{\mathcal{X}}:\mathcal{X}\to\mathbb{R},\,c_{\mathcal{Y}}:\mathcal{Y}\to\mathbb{R}$.
Assume that $P_{X}$ concentrates on a compact set and $P_{Y}$ satisfies
$\mathbb{E}[\exp(c_{\mathcal{Y}}^{2}(Y))]<\infty$. Then, under Assumption
3, for any $\alpha\ge0,\tau\ge0$, it holds that 
\begin{equation}
\limsup_{n\to\infty}E_{0}^{(n)}(\alpha,\tau)\le\lim_{\tau'\uparrow\tau}\psi(\alpha,\tau').\label{eq:-38-1-1}
\end{equation}
\item If $c$ is bounded and satisfies Assumption 3, then for any $\alpha\ge0,\tau\ge0$,
it holds that  
\begin{equation}
\limsup_{n\to\infty}E_{0}^{(n)}(\alpha,\tau)\le\lim_{\alpha'\downarrow\alpha}\lim_{\tau'\uparrow\tau}\psi(\alpha',\tau').\label{eq:-38}
\end{equation}
\item Under Assumption 1 (given in Section \ref{subsec:General-Cost}),
for any $(\alpha,\tau)$ in the interior of $\mathrm{dom}\psi$, it
holds that 
\[
\liminf_{n\to\infty}E_{0}^{(n)}(\alpha,\tau)\ge\psi(\alpha,\tau).
\]
\item Assume that $\mathcal{X}=\mathcal{Y}$ equipped with a metric $d$
is a Polish metric space, and the cost function is set to $c=d^{p}$
for $p\ge1$.  Then, for any $(\alpha,\tau)$ in the interior of
$\mathrm{dom}\psi$, it holds that 
\[
\liminf_{n\to\infty}E_{0}^{(n)}(\alpha,\tau)\ge\psi(\alpha,\tau).
\]
\end{enumerate}
\end{thm}
\begin{rem}
It is not straightforward to derive upper bound $\lim_{n\to\infty}E_{0}^{(n)}(\alpha,\tau)$
for the case in which the cost is unbounded and $P_{X}$ does not
concentrate on a compact set. One may wonder if it is possible to
generalize the result for the compact $\mathcal{X}$ to the noncompact
(Polish) $\mathcal{X}$ by truncating the noncompact space into a
compact one. In fact, this idea is adopted in the proof of Statement
2 in Theorem \ref{thm:LD-1}; see Section \ref{subsec:Statement-2}.
As shown in this proof, the set $A\subseteq\mathcal{X}^{n}$ is projected
to a space of dimension $n'$  where $n'=(1-\epsilon')n$ for small
$\epsilon'$. Such an idea seems not to work for unbounded costs,
since in this case, the remaining space of dimension $\epsilon n$
cannot be omitted by paying only a finite cost. Another possible way
is to generalize the inherently typical subset lemma \cite{ahlswede1997identification}
to infinite (countably infinite or uncountable) spaces. The continuity
of information quantities in the weak topology is the key point in
the proof of the inherently typical subset lemma \cite{ahlswede1997identification}.
However, it is well known that in an infinite space, convergence in
weak topology does not implies convergence in Shannon information
quantities in general, i.e., Shannon information quantities are discontinuous
\cite{ho2009discontinuity}.  So, certain assumptions must be posed
in this method. 
\end{rem}
\begin{rem}
\label{rem:dimensionfree} In fact, we can obtain the following ``dimension-free''
bound: For arbitrary Polish $\mathcal{X}$ and $\mathcal{Y}$, it
holds that for any $(\alpha,\tau)$, 
\begin{align*}
E_{0}^{(n)}(\alpha,\tau) & \le\lim_{\alpha'\downarrow\alpha}\sup_{\substack{Q_{XW|K}:\\
D(Q_{X|WK}\|P_{X}|Q_{WK})\le\alpha
}
}\\
 & \qquad\inf_{\substack{(\tau_{k})_{k\in[n]},Q_{Y|XWK}:\\
c(X,Y)\le\tau_{K}\textrm{ a.s.}\\
\tau_{k}\ge0,\,\mathbb{E}[\tau_{K}]=\tau
}
}D(Q_{Y|WK}\|P_{Y}|Q_{WK}),
\end{align*}
where $K\sim\Unif[n]$ and there is no restriction on the alphabet
size of $W$. To prove this bound, we redefine $Q_{X^{n}}$ in Step
2 of Section \ref{subsec:Finite} as the uniform distribution on the
set $A$ itself, instead on an inherently typical subset of $A$,
and rechoose $Q_{Y^{n}|X^{n}}$ in Step 3 of Section \ref{subsec:Finite}
as $Q_{Y^{n}|X^{n}}=\prod_{k=1}^{n}Q_{Y_{k}|X^{k}}$ where $Q_{Y_{k}|X^{k}},k\in[n]$
are transition probability measures such that $c(x_{k},Y_{k})\le\tau_{k}$
a.s. under $Q_{Y_{k}|X^{k}=x^{k}}$ for any $x^{k}$. Then, following
the proof steps in Section \ref{subsec:Finite}, the ``dimension-free''
bound is obtained. Note that the inherently typical subset lemma is
not involved here. However, by comparing this bound with the upper
bound in \eqref{eq:-38} (or \eqref{eq:-38-1-1}), it is easy to see
that this ``dimension-free'' bound is not asymptotically tight.
It is not obvious to see whether our bound $\lim_{\alpha'\downarrow\alpha}\lim_{\tau'\uparrow\tau}\psi(\alpha',\tau')$
is a dimension-free bound for $E_{0}^{(n)}(\alpha,\tau)$. If yes,
 finding a proof is an interesting but challenging task. 
\end{rem}

The following is an example that satisfies all the conditions in Statement
1 in Theorem \ref{thm:LD-1}. 
\begin{example}
\label{exa:4}The space $\mathcal{X}=\mathcal{Y}$ equipped with a
metric $d$ is a Polish metric space, and the cost function is set
to $c=d$. Moreover, $P_{X}$ concentrates on a compact set and $P_{Y}$
satisfies $\mathbb{E}[\exp(d^{2}(x_{0},Y))]<\infty$ for some (and
hence all) $x_{0}$. In this case, by the inequality $d(x,y)\le d(x,x_{0})+d(y,x_{0})$,
we can choose $c_{\mathcal{X}}(x)=d(x,x_{0})$ and $c_{\mathcal{Y}}(y)=d(y,x_{0})$. 
\end{example}
Statement 3 in Theorem \ref{thm:LD-1} only requires Assumption 1.
So, Statement 3 in Theorem \ref{thm:LD-1} holds for Examples \ref{exa:1}
and \ref{exa:2} given below Assumption 1. 

 Assumption 3 is satisfied by Example  \ref{exa:2}. So, Statement
2 in Theorem \ref{thm:LD-1} holds for Example  \ref{exa:2}. We
now verify this point. It suffices to consider small enough $\epsilon$
such that $d(x,x')\le\epsilon<d(x,y)$.
\begin{align*}
d^{p}(x',y) & \ge(d(x,y)-d(x',x))^{p}\ge(d(x,y)-\epsilon)^{p}.
\end{align*}
So, 
\begin{align*}
d^{p}(x,y)-d^{p}(x',y) & \le d^{p}(x,y)-(d(x,y)-\epsilon)^{p}.
\end{align*}
Since $t\in[0,d_{\sup}]\mapsto t^{p}$ is continuous and hence uniformly
continuous, there is a function $\delta(\epsilon):(0,\infty)\to(0,\infty)$
 vanishing as $\epsilon\downarrow0$ such that $t^{p}-(t-\epsilon)^{p}\le\delta(\epsilon)$
for all $t\in[\epsilon,M]$. 

If $\psi$ is continuous at $(\alpha,\tau)$, then all the inequalities
in \eqref{eq:-38} turn into equalities. Given $Q_{XW}$, the infimization
in \eqref{eq:-38}, $g(\tau):=\inf_{Q_{Y|XW}:\mathbb{E}[c(X,Y)]\le\tau}D(Q_{Y|W}\|P_{Y}|Q_{W})$,
is convex and nonincreasing in $\tau$, and hence, it is only possible
to be discontinuous at the point $\tau_{0}:=\inf\{\tau:g(\tau)<\infty\}$.
The proof of Theorem \ref{thm:LD-1} is provided in Section \ref{sec:Proof-of-Theorem-2}.
Furthermore, to make it consistent with the expression of $\phi$,
 the infimization in \eqref{eq:psi} can be written as the infimization
over $Q_{Y|W}$ such that $\C(Q_{X|W},Q_{Y|W}|Q_{W})\le\tau$.

Theorem \ref{thm:LD-1} generalizes Ahlswede and Zhang's result \cite{ahlswede1999asymptotical}
from finite spaces to Polish spaces. Similar to Ahlswede and Zhang's,
our proof is also based on the inherently typical subset lemma, but
requires more technical treatments since the spaces are much more
general. Furthermore, previously, there was no bound on the alphabet
size of $W$ in the definition of $\psi$, even for the finite alphabet
case. For the finite alphabet case, Ahlswede and Zhang \cite{ahlswede1997identification,ahlswede1999asymptotical}
showed that 
\[
\psi_{N}(\alpha,\tau)\le\psi(\alpha,\tau)\le\psi_{N}(\alpha,\tau)+O\Big(\frac{\log^{2}N}{N^{1/|\calX|}}\Big),
\]
where $\psi_{N}$ is defined similarly as $\psi$ but with $W$ restricted
to concentrate on the alphabet $\mathcal{W}$ satisfying $|\mathcal{W}|=N$.
Theorem \ref{thm:psi} shows that $\psi(\alpha,\tau)=\psi_{N}(\alpha,\tau)$
for any $N\ge2$, which does not only sharpen Ahlswede and Zhang's
result, but also makes $\psi(\alpha,\tau)$ ``computable'' for the
finite alphabet case in the sense that $\psi(\alpha,\tau)$ can be
evaluated by a finite-dimensional program.

\subsection{\label{subsec:dual}Dual Formulas}

We now provide dual formulas for $\psi$ in \eqref{eq:psi} and variants
of $\phi$ in \eqref{eq:phi} and $\varphi$ in \eqref{eq:varphi}.
Our motivations for this part are two-fold: One is to verify the equivalence
between our formula $\conv{\varphi}_{X}(\tau)$ and Alon, Boppana,
and Spencer's in \cite{alon1998asymptotic} for the asymptotic concentration
exponent; the other is to prove the bound on the alphabet size of
$W$ given in Theorem \ref{thm:The-alphabet-size}. The main tool
used in deriving dual formulas is the Kantorovich duality for the
optimal transport cost and the duality for the I-projection. In the
following, for a measurable function $f:\mathcal{X}\to\mathbb{R}$,
we adopt the notation $P_{X}(f)=\int_{\mathcal{X}}f\ \mathrm{d}P_{X}$.

We define a variant of $\phi$ as for $\alpha\ge0,\tau\ge0$, 
\begin{align}
\phi_{\ge}(\alpha,\tau) & :=\inf_{\substack{Q_{X}\in\mathcal{P}(\mathcal{X}),Q_{Y}\in\mathcal{P}(\mathcal{Y}):\\
D(Q_{X}\|P_{X})\le\alpha,\C(Q_{X},Q_{Y})\ge\tau
}
}D(Q_{Y}\|P_{Y}).\label{eq:phi-1}
\end{align}
Then, $\phi_{\ge}(\alpha,\tau)\le\phi(\alpha,\tau)\le\lim_{\tau'\downarrow\tau}\phi_{\ge}(\alpha,\tau').$
Hence, for all $(\alpha,\tau)$ in the interior of $\mathrm{dom}\conv{\phi}$,
$\conv{\phi}_{\ge}(\alpha,\tau)=\conv{\phi}(\alpha,\tau).$ We next
derive a dual formula for $\phi_{\ge}$. 
\begin{thm}
\label{thm:phi}For all $\alpha\ge0,\tau\ge0$, 
\begin{align}
\phi_{\ge}(\alpha,\tau) & =\inf_{\substack{(f,g)\in C_{\mathrm{b}}(\mathcal{X})\times C_{\mathrm{b}}(\mathcal{Y}):\\
f+g\le c
}
}\sup_{\lambda>0,\eta>0}\lambda\tau-\log P_{Y}(e^{\lambda g})\nonumber \\
 & \qquad\qquad-\eta\alpha-\eta\log P_{X}(e^{\frac{\lambda}{\eta}f}).\label{eq:-48}
\end{align}
Moreover, for all $(\alpha,\tau)$ in the interior of $\mathrm{dom}\conv{\phi}$,
$\conv{\phi}_{\ge}(\alpha,\tau)=\conv{\phi}(\alpha,\tau).$ 
\end{thm}
Define a variant of $\varphi$ as 
\begin{equation}
\varphi_{\ge}(\tau):=\phi_{\ge}(0,\tau)=\inf_{Q_{Y}:\C(P_{X},Q_{Y})\ge\tau}D(Q_{Y}\|P_{Y}).\label{eq:varphi-1}
\end{equation}
As a consequence of Theorem \ref{thm:phi}, we have a dual formula
for $\varphi_{\ge}$. 
\begin{cor}
For all $\tau\ge0$, 
\begin{align}
\varphi_{\ge}(\tau) & =\inf_{\substack{(f,g)\in C_{\mathrm{b}}(\mathcal{X})\times C_{\mathrm{b}}(\mathcal{Y}):\\
f+g\le c
}
}\sup_{\lambda\ge0}\lambda(\tau-P_{X}(f))\nonumber \\
 & \qquad\qquad-\log P_{Y}(e^{\lambda g}).\label{eq:-6-2}
\end{align}
Moreover, for all $\tau$ in the interior of $\mathrm{dom}\conv{\varphi}$,
$\conv{\varphi}_{\ge}(\tau)=\conv{\varphi}(\tau).$ 
\end{cor}
When $P_{X}=P_{Y}$, the function $\varphi_{\ge}$ reduces to the
function $\varphi_{X,\ge}$ defined in \eqref{eq:varphi_X-2}: 
\begin{equation}
\varphi_{X,\ge}(\tau)=\inf_{Q_{X}:\C(P_{X},Q_{X})\ge\tau}D(Q_{X}\|P_{X}).\label{eq:varphi_X-1}
\end{equation}
For this case, we can write $\varphi_{X,\ge}$ as follows. 
\begin{prop}
\label{prop:varphi}When $P_{X}=P_{Y}$ and $c=d$ with $d$ being
a metric, we have for any $0\le\tau<\tau_{\max}$, 
\begin{align}
\varphi_{X,\ge}(\tau) & =\inf_{\textrm{1-Lip }f:P_{X}(f)=0}\sup_{\lambda\ge0}\lambda\tau-\log P_{X}(e^{\lambda f}).\label{eq:-46}
\end{align}
Moreover, for all $\tau$ in the interior of $\mathrm{dom}\conv{\varphi}_{X}$,
$\conv{\varphi}_{X,\ge}(\tau)=\conv{\varphi}_{X}(\tau).$ 
\end{prop}
Based on the dual formula in \eqref{eq:-46}, we next show the equivalence
between our formula $\conv{\varphi}_{X}(\tau)$ and Alon, Boppana,
and Spencer's in \cite{alon1998asymptotic}. When $(\mathcal{X},P_{X})$
and $(\mathcal{Y},P_{Y})$ are the same finite metric probability
space, the cost function $c$ is set to the metric $d$ on this space,
and $a$ is set to $\frac{1}{2}$ (equivalently, $\alpha_{n}=\frac{1}{n}\log2$),
Alon, Boppana, and Spencer in \cite{alon1998asymptotic} proved an
alternative expression for $\lim_{n\to\infty}E_{1}^{(n)}(\alpha_{n},\tau)$
which is 
\[
r(\tau):=\sup_{\lambda\ge0}\lambda\tau-L_{G}(\lambda).
\]
Here $G=(\mathcal{X},d,P_{X})$ denotes the metric probability space
we consider, and $L_{G}(\lambda)$ denotes the maximum of $\log P_{X}(e^{\lambda f})$
over all $1$-Lipschitz functions\footnote{Call $f:\mathcal{X}\to\mathbb{R}$ $1$-Lipschitz if $|f(x)-f(x')|\le d(x,x')$
for all $x,x'\in\mathcal{X}$. } $f:\mathcal{X}\to\mathbb{R}$ with $P_{X}(f)=0$. 
\begin{thm}
\label{thm:equivalence} For a finite metric probability space $G=(\mathcal{X},d,P_{X})$
and all $\tau>0$, $\conv{\varphi}_{X}(\tau)=r(\tau)$. 
\end{thm}
Lastly, we provide a dual formula for $\psi$. 
\begin{thm}
\label{thm:psi} For all $\alpha\ge0,\tau\ge0$, 
\begin{align}
\psi(\alpha,\tau) & =\sup_{f_{w}+g_{w}\le c,\forall w\in\{0,1\}}\sup_{\lambda\ge0}\inf_{\eta>0}\max_{w\in\{0,1\}}\eta\alpha\nonumber \\
 & \qquad+\eta\log P_{X}(e^{\frac{\lambda}{\eta}f_{w}})-\lambda\tau-\log P_{Y}(e^{-\lambda g_{w}}),\label{eq:-48-1}
\end{align}
where $(f_{w},g_{w})\in C_{\mathrm{b}}(\mathcal{X})\times C_{\mathrm{b}}(\mathcal{Y}),\forall w$.
Moreover, the alphabet size of $W$ in the definition of $\psi$ (in
\eqref{eq:psi}) can be restricted to be no larger than $2$. 
\end{thm}
The second statement of Theorem \ref{thm:psi} is exactly Theorem
\ref{thm:The-alphabet-size}.

\begin{cor}
\label{cor:uppersemicontinuity} For $\alpha>0$, $\lim_{\alpha'\uparrow\alpha}\lim_{\tau'\downarrow\tau}\psi(\alpha',\tau')=\psi(\alpha,\tau).$ 
\end{cor}

\subsection{Applications to Other Problems}

\subsubsection{Strassen's Optimal Transport}

We have characterized or bounded the concentration and isoperimetric
exponents. Our results extend Alon, Boppana, and Spencer's in \cite{alon1998asymptotic},
Gozlan and Léonard's \cite{gozlan2007large}, and Ahlswede and Zhang's
in \cite{ahlswede1999asymptotical}. Furthermore, the concentration
or isoperimetric function is closely related to Strassen's optimal
transport problem, for which we aim at characterizing 
\begin{equation}
\S_{t}^{(n)}(P_{X},P_{Y}):=\min_{P_{X^{n}Y^{n}}\in\cC(P_{X}^{\otimes n},P_{Y}^{\otimes n})}\P\{c_{n}(X^{n},Y^{n})>t\}\label{eq:strassen-1}
\end{equation}
for $t\ge0$. By Strassen's duality \cite{yu2020asymptotics}, 
\begin{align}
\S_{t}^{(n)}(P_{X},P_{Y}) & =\sup_{\textrm{closed }A\subseteq\mathcal{X}}\left\{ P_{X}^{\otimes n}(A)-P_{Y}^{\otimes n}(A^{t})\right\} \label{eq:-20-1}\\
 & =\sup_{a\in[0,1]}\{a-\Gamma^{(n)}(a,t)\}.
\end{align}
Therefore, if $\Gamma^{(n)}(a,t)$ is characterized, then so is $\S_{t}^{(n)}(P_{X},P_{Y})$.
In fact, the asymptotic exponents of $\S_{t}^{(n)}(P_{X},P_{Y})$
were already characterized by the author in \cite{yu2020asymptotics}.
Moreover, it has been shown in \cite{yu2020asymptotics} that it suffices
to restrict $A$ in the supremum in \eqref{eq:-20-1} to be ``exchangeable''
(or ``permutation-invariant''). In other words, $A$ could be specified
by a set $\calA$ of empirical measures in the way $A=\L_{n}^{-1}(\calA)$.
Hence, the supremum in \eqref{eq:-20-1} can be written as an optimization
over empirical measures. From this point, we observe that if $a\mapsto\Gamma^{(n)}(a,t)$
is convex, then computing $\Gamma^{(n)}(a,t)$ for $a\in[0,1]$ is
equivalent to computing $\inf_{\textrm{closed }A\subseteq\mathcal{X}}P_{Y}^{\otimes n}(A^{t})-\lambda P_{X}^{\otimes n}(A)$
for $\lambda\ge0$. Similarly to the argument in \cite{yu2020asymptotics},
the set $A$ in the definition of $\Gamma^{(n)}(a,t)$ (see \eqref{eq:Gamma})
can be also restricted to be ``exchangeable''. In this case, central
limit theorems can be applied to derive the limit of $\Gamma^{(n)}(a,t_{n})$
with $a$ fixed and $t_{n}$ set to a sequence approaching $\C(P_{X},P_{Y})$
in the order of $1/\sqrt{n}$, just like central limit results in
derived in \cite{yu2020asymptotics}.

\subsubsection{Classic Isoperimetric Problem}

The isoperimetric problem considered in Section \ref{subsec:Asymptotic-Isoperimetric-Exponen}
concerns thick boundaries. In contrast, in the classic isoperimetric
problem, the boundary is extremely thin. We assume that $\mathcal{X}=\mathcal{Y}$
equipped with a metric $d$ is a Polish metric space, and moreover,
$P_{X}=P_{Y}=:P$ and $c=d^{p}$ with $p\ge1$. Recall the boundary
measure defined in \eqref{eq:perimeter}. Obviously, the boundary
measure do not change if the metric $d$ is replaced by $d_{s}:=\min\{d,s\}$
for a number $s>0$. So, without loss of generality, we assume that
$d$ is bounded. The boundary measure can be alternatively expressed
as 
\begin{align}
(P^{\otimes n})^{+}(A) & =\liminf_{r\downarrow0}\frac{P^{\otimes n}(A^{r^{p}})-P^{\otimes n}(A)}{\log P^{\otimes n}(A^{r^{p}})-\log P^{\otimes n}(A)}\nonumber \\
 & \qquad\qquad\times\frac{\log P^{\otimes n}(A^{r^{p}})-\log P^{\otimes n}(A)}{r}\\
 & =P^{\otimes n}(A)\liminf_{r\downarrow0}\frac{\log[P^{\otimes n}(A^{r^{p}})/P^{\otimes n}(A)]}{r}\\
 & =n^{1-1/p}P^{\otimes n}(A)\liminf_{r\downarrow0}F_{r}^{(n)}(A),\label{eq:-22}
\end{align}
where
\[
F_{r}^{(n)}(A):=\frac{\frac{1}{n}\log[P^{\otimes n}(A^{nr^{p}})/P^{\otimes n}(A)]}{r}
\]
is the slope of the line through two points at $s=0$ and $s=r$ on
the curve $s\mapsto\frac{1}{n}\log P^{\otimes n}(A^{ns^{p}})$. Note
that $\liminf_{r\downarrow0}F_{r}^{(n)}(A)$ is the lower right-hand
derivative (i.e., the lower Dini derivative) of $s\mapsto\frac{1}{n}\log P^{\otimes n}(A^{ns^{p}})$.

\textbf{ }\textbf{Assumption 4 }(Isoperimetric Stability): (a).
Given $\alpha>0$, there are a sequence of sets $B_{n}\subseteq\mathcal{X}^{n}$
of probability $e^{-n\alpha}$ and a function $\delta:(0,\infty)\times\mathbb{N}\to[0,\infty)$
such that $B_{n}$ minimizes the boundary measure $(P^{\otimes n})^{+}(A)$
over all sets $A$ of probability $e^{-n\alpha}$, $\limsup_{\epsilon\downarrow0}\limsup_{n\to\infty}\delta(\epsilon,n)=0$,
and meanwhile 
\begin{equation}
\liminf_{r\downarrow0}F_{r}^{(n)}(B_{n})\ge F_{\epsilon}^{(n)}(B_{n})-\delta(\epsilon,n),\quad\forall\epsilon>0,n\in\mathbb{N}.\label{eq:-60-2}
\end{equation}
(b). Given $\alpha>0$, there are a family of sets $A_{n,\epsilon}\subseteq\mathcal{X}^{n}$
of probability $e^{-n\alpha}$ and a function $\delta:(0,\infty)\times\mathbb{N}\to[0,\infty)$
such that $A_{n,\epsilon}$ minimizes $P^{\otimes n}(A^{n\epsilon^{p}})$
over all sets $A$ of probability $e^{-n\alpha}$, $\limsup_{\epsilon\downarrow0}\limsup_{n\to\infty}\delta(\epsilon,n)=0$,
and meanwhile 
\begin{equation}
\liminf_{r\downarrow0}F_{r}^{(n)}(A_{n,\epsilon})\le F_{\epsilon}^{(n)}(A_{n,\epsilon})+\delta(\epsilon,n),\quad\forall\epsilon>0,n\in\mathbb{N}.\label{eq:-60-2-1}
\end{equation}

Part (a) of Assumption 4 is true if the probability of the $n^{1/p}\epsilon$-enlargement
of $B_{n}$ under the product metric $c_{n}^{1/p}$ does not change
dramatically as $\epsilon\downarrow0$ for all sufficiently large
$n$. Part (b) is true if $A_{n,\epsilon}$ has a similar property.
 Assumption 4 is satisfied by the tuple of the standard Gaussian
measure, Euclidean distance, and $p=2$. In this case, the Gaussian
isoperimetric inequality states that half-spaces minimizes the Gaussian
boundary measure \cite{sudakov1978extremal,borell1975brunn}. Moreover,
for half-spaces $B_{n}$ of probability $e^{-n\alpha}$, $P^{\otimes n}(B_{n}^{nr^{2}})=\Phi(\Phi^{-1}(e^{-n\alpha})+r\sqrt{n})$
which is log-concave in $r$. Hence, it can be seen that 
\begin{align*}
\lim_{r\downarrow0}F_{r}^{(n)}(B_{n}) & \sim\sqrt{2\alpha},\\
F_{\epsilon}^{(n)}(B_{n}) & \sim\sqrt{2\alpha}-\frac{\epsilon}{2}.
\end{align*}
So, Part (a) of Assumption 4 holds in this case. Note that the Gaussian
isoperimetric inequality also implies that a half-space minimizes
$P^{\otimes n}(A^{nr^{2}})$ over all sets with the probability same
as that of the half-space. So, Part (b) of Assumption 4 follows. 

Define 
\begin{equation}
\xi(\alpha):=\liminf_{r\downarrow0}\frac{\alpha-\lim_{\alpha'\downarrow\alpha}\psi(\alpha',r^{p})}{r},\label{eq:xi}
\end{equation}
where $\psi$ is defined in \eqref{eq:psi} but with both $P_{X}$
and $P_{Y}$ therein set to $P$. 
\begin{thm}[Isoperimetric Inequality]
\label{thm:bound-2-1} Assume that $\mathcal{X}=\mathcal{Y}$ is
a Polish space and the metric $d$ is bounded. Let $\alpha>0$. Then,
under Part (a) of Assumption 4, it holds that for any set $A$ of
probability $e^{-n\alpha}$, 
\begin{align}
(P^{\otimes n})^{+}(A) & \ge n^{1-1/p}e^{-n\alpha}(\xi(\alpha)+o_{n}(1)),\label{eq:-67}
\end{align}
where $o_{n}(1)$ is a term vanishing as $n\to\infty$ which is independent
of $A$, but depends on $(\alpha,p,P)$.  Moreover, under Part (b)
of Assumption 4, if $\alpha\mapsto\psi(\alpha,r^{p})$ is continuous
at $\alpha$ for all sufficiently small $r>0$, then the inequality
in \eqref{eq:-67} is asymptotically sharp in the sense that there
is a sequence of sets $A_{n}\subseteq\mathcal{X}^{n}$ of probability
$e^{-n\alpha}$ such that 
\begin{align*}
(P^{\otimes n})^{+}(A_{n}) & \le n^{1-1/p}e^{-n\alpha}(\xi(\alpha)+o_{n}(1)).
\end{align*}
\end{thm}

\begin{rem}
The ``dimension-free'' bound given in Remark \ref{rem:dimensionfree}
can be used to derive an isoperimetric inequality similar to the one
in \eqref{eq:-67} but without the Assumption 4, which will not be
given here, since this inequality is not expected to be asymptotically
sharp. 
\end{rem}
The proof of this theorem is provided in Section \ref{sec:Proof-of-Theorem}.
 Removing Assumption 4 for the inequality in \eqref{eq:-67}  is
left to be investigated in the future. Furthermore, the equivalence
between the isoperimetric problem with thick boundaries and the one
with thin boundaries under other certain conditions is investigated
by E. Milman \cite{milman2010isoperimetric}. However, E. Milman only
focuses on complete Riemannian manifolds, while our setting concerns
general Polish spaces. 

The inequality in \eqref{eq:-67} can be seen as a generalization
of Gaussian isoperimetric inequality \cite{sudakov1978extremal,borell1975brunn}.
In the setting of the standard Gaussian measure and the Euclidean
distance, 
\begin{align}
(P^{\otimes n})^{+}(A) & \ge\varphi(\Phi^{-1}(e^{-n\alpha}))\sim e^{-n\alpha}\sqrt{2n\alpha},\label{eq:-67-2}
\end{align}
where $\varphi$ is the probability density function of the standard
Gaussian, and the asymptotic equality follows by the fact that $\varphi(\Phi^{-1}(a))\sim a\sqrt{2\log(1/a)}$
as $a\downarrow0$. Half-spaces are exactly optimal in the Gaussian
setting, and intuitively close to optimal in other product probability
measures on Euclidean spaces if the volume is fixed,  which follows
by the functional central limit theorem. In contrast, when the volume
is exponentially small, as indicated by Theorem \ref{thm:bound-2-1},
the empirically typical sets are conjectured to be asymptotically
optimal. 

\section{\label{sec:Proof-of-Theorem-1}Proof of Theorem \ref{thm:LD}}

\subsection{Statement 1}

The proof idea is essentially due to Marton \cite{Marton86,marton1996bounding}.
Our proof relies on the subadditivity of OT costs or the tensorization
of a new kind of transport-entropy inequalities given in \eqref{eq:phi-2},
instead of traditional transport-entropy inequalities. 

Let $A\subseteq\mathcal{X}$ be a measurable subset. Denote $t=n\tau$.
Denote $Q_{X^{n}}=P_{X}^{\otimes n}(\cdot|A)$ and $Q_{Y^{n}}=P_{Y}^{\otimes n}(\cdot|(A^{t})^{c})$.
For two sets $A,B$, denote $c_{n}(A,B)=\inf_{x^{n}\in A,y^{n}\in B}c_{n}(x^{n},y^{n})$.
We first claim that 
\begin{align*}
\C(Q_{X^{n}},Q_{Y^{n}}) & >t.
\end{align*}
We now prove it. If $c_{n}(A,(A^{t})^{c})$ is attained by some pair
$(x^{*n},y^{*n})$, then 
\[
\C(Q_{X^{n}},Q_{Y^{n}})\ge c_{n}(A,(A^{t})^{c})=c_{n}(x^{*n},y^{*n})>t.
\]
We next consider the case that $c_{n}(A,(A^{t})^{c})$ is not attained.
Denote the optimal coupling that attains the infimum in the definition
of $\C(Q_{X^{n}},Q_{Y^{n}})$ as $Q_{X^{n}Y^{n}}$ (the existence
of this coupling is well known). Therefore, 
\[
\C(Q_{X^{n}},Q_{Y^{n}})=\mathbb{E}_{Q}c_{n}(X^{n},Y^{n}).
\]
By definition, $c_{n}(x^{n},y^{n})>t$ for all $x^{n}\in A,y^{n}\in B$.
Since any probability measure on a Polish space is tight, we have
that for any $\epsilon>0$, there exists a compact set $F$ such that
$Q_{X^{n}Y^{n}}(F)>1-\epsilon$. By the lower semi-continuity of $c$
and compactness of $F$, we have that $\inf_{(x^{n},y^{n})\in F}c_{n}(x^{n},y^{n})$
is attained, and hence, $\inf_{(x^{n},y^{n})\in F}c_{n}(x^{n},y^{n})>t$,
i.e., there is some $\delta>0$ such that $c_{n}(x^{n},y^{n})\ge t+\delta$
for all $(x^{n},y^{n})\in F$. This further implies that $\C(Q_{X^{n}},Q_{Y^{n}})\ge(1-\epsilon)(t+\delta)+\epsilon t>t.$
Hence, the claim above is true.

Furthermore, by definition of $Q_{X^{n}},Q_{Y^{n}}$, we then have
\begin{align*}
\frac{1}{n}D(Q_{X^{n}}\|P_{X}^{\otimes n}) & =-\frac{1}{n}\log P_{X}^{\otimes n}(A)\\
\frac{1}{n}D(Q_{Y^{n}}\|P_{Y}^{\otimes n}) & =-\frac{1}{n}\log P_{Y}^{\otimes n}((A^{t})^{c}).
\end{align*}
Therefore, 
\begin{align}
E_{1}^{(n)}(\alpha,\tau) & =-\frac{1}{n}\log\left(1-\inf_{A:P_{X}^{\otimes n}(A)\ge e^{-n\alpha}}P_{Y}^{\otimes n}(A^{t})\right)\nonumber \\
 & \geq\inf_{\substack{Q_{X^{n}},Q_{Y^{n}}:\\
\frac{1}{n}D(Q_{X^{n}}\|P_{X}^{\otimes n})\le\alpha,\\
\frac{1}{n}\C(Q_{X^{n}},Q_{Y^{n}})>\tau
}
}\frac{1}{n}D(Q_{Y^{n}}\|P_{Y}^{\otimes n}).\label{eq:-26}
\end{align}

Note that this lower bound depends on the dimension $n$. We next
single-letterize this bound, i.e., make it independent of $n$. To
this end, we need the chain rule for relative entropies and the chain
rule for OT costs. For relative entropies, we have the chain rule:
\begin{align}
 & D(Q_{X^{n}}\|P_{X}^{\otimes n})=\sum_{k=1}^{n}D(Q_{X_{k}|X^{k-1}}\|P_{X}|Q_{X^{k-1}})\label{eq:}\\
 & D(Q_{Y^{n}}\|P_{Y}^{\otimes n})=\sum_{k=1}^{n}D(Q_{Y_{k}|Y^{k-1}}\|P_{Y}|Q_{Y^{k-1}}).\nonumber 
\end{align}
For OT costs, we have a similar ``chain rule''. 
\begin{lem}[``Subadditivity'' for OT Costs]
\cite[Lemma A.1]{gozlan2017kantorovich} \label{lem:coupling-1}
For any transition probability measures $Q_{X_{i}|X^{i-1}},Q_{Y_{i}|Y^{i-1}},i\in[n]$,
it holds that 
\[
\C(Q_{X^{n}},Q_{Y^{n}})\le\sum_{k=1}^{n}\C(Q_{X_{k}|X^{k-1}},Q_{Y_{k}|Y^{k-1}}|Q_{X^{k-1}},Q_{Y^{k-1}}),
\]
where $Q_{X^{n}}:=\prod_{i=1}^{n}Q_{X_{i}|X^{i-1}},\,Q_{Y^{n}}:=\prod_{i=1}^{n}Q_{Y_{i}|Y^{i-1}}$,
and 
\begin{align}
 & \C(Q_{X_{k}|X^{k-1}},Q_{Y_{k}|Y^{k-1}}|Q_{X^{k-1}},Q_{Y^{k-1}})\nonumber \\
 & :=\sup_{Q_{X^{k-1}Y^{k-1}}\in\cC(Q_{X^{k-1}},Q_{Y^{k-1}})}\nonumber \\
 & \qquad\qquad\C(Q_{X_{k}|X^{k-1}},Q_{Y_{k}|Y^{k-1}}|Q_{X^{k-1}Y^{k-1}}).\label{eq:-35}
\end{align}
\end{lem}
For completeness, we provide the proof of Lemma \ref{lem:coupling-1}
since it is very short. 
\begin{IEEEproof}[Proof of Lemma \ref{lem:coupling-1}]
We need the following lemma on composition of couplings, which is
well-known in OT theory; see the proof in, e.g., \cite[Lemma 9]{YuTan2020_exact}. 
\begin{lem}[Composition of Couplings]
\label{lem:coupling} For any transition probability measures $(P_{X_{i}|X^{i-1}W},P_{Y_{i}|Y^{i-1}W}),i\in[n]$
and any $Q_{X_{i}Y_{i}|X^{i-1}Y^{i-1}W}\in\cC(P_{X_{i}|X^{i-1}W},P_{Y_{i}|Y^{i-1}W}),i\in[n]$,
we have 
\[
\prod_{i=1}^{n}Q_{X_{i}Y_{i}|X^{i-1}Y^{i-1}W}\in\cC\Big(\prod_{i=1}^{n}P_{X_{i}|X^{i-1}W},\prod_{i=1}^{n}P_{Y_{i}|Y^{i-1}W}\Big).
\]
\end{lem}
By the lemma above, we have 
\begin{align}
 & \C(Q_{X^{n}},Q_{Y^{n}})\nonumber \\
 & =\inf_{Q_{X^{n}Y^{n}}\in\cC(Q_{X^{n}},Q_{Y^{n}})}\sum_{k=1}^{n}\mathbb{E}c(X_{k},Y_{k})\nonumber \\
 & \leq\inf_{\substack{Q_{X^{n-1}Y^{n-1}}\in\\
\cC(Q_{X^{n-1}},Q_{Y^{n-1}})
}
}\Bigl[\sum_{k=1}^{n-1}\mathbb{E}c(X_{k},Y_{k})\nonumber \\
 & \qquad+\inf_{\substack{Q_{X_{n}Y_{n}|X^{n-1}Y^{n-1}}\in\\
\cC(Q_{X_{n}|X^{n-1}},Q_{Y_{n}|Y^{n-1}})
}
}\mathbb{E}c(X_{n},Y_{n})\Bigr]\label{eq:-3}\\
 & \leq\inf_{\substack{Q_{X^{n-1}Y^{n-1}}\in\\
\cC(Q_{X^{n-1}},Q_{Y^{n-1}})
}
}\Bigl[\sum_{k=1}^{n-1}\mathbb{E}c(X_{k},Y_{k})\nonumber \\
 & \qquad+\sup_{\substack{Q_{X^{n-1}Y^{n-1}}\in\\
\cC(Q_{X^{n-1}},Q_{Y^{n-1}})
}
}\inf_{\substack{Q_{X_{n}Y_{n}|X^{n-1}Y^{n-1}}\in\\
\cC(Q_{X_{n}|X^{n-1}},Q_{Y_{n}|Y^{n-1}})
}
}\mathbb{E}c(X_{n},Y_{n})\Bigr]\\
 & =\inf_{\substack{Q_{X^{n-1}Y^{n-1}}\in\\
\cC(Q_{X^{n-1}},Q_{Y^{n-1}})
}
}\Bigl[\sum_{k=1}^{n-1}\mathbb{E}c(X_{k},Y_{k})\Bigr]\nonumber \\
 & \qquad+\C(Q_{X_{n}|X^{n-1}},Q_{Y_{n}|Y^{n-1}}|Q_{X^{n-1}},Q_{Y^{n-1}})\\
 & \cdots\cdots\\
 & \leq\sum_{k=1}^{n}\C(Q_{X_{k}|X^{k-1}},Q_{Y_{k}|Y^{k-1}}|Q_{X^{k-1}},Q_{Y^{k-1}}),\label{eq:chainrule_OT}
\end{align}
where in \eqref{eq:-3}, Lemma \ref{lem:coupling} is applied. 
\end{IEEEproof}
We continue the proof of \eqref{eq:bound2}. From \eqref{eq:}, we
know that for any $Q_{X^{n}}$ such that $\frac{1}{n}D(Q_{X^{n}}\|P_{X}^{\otimes n})\le\alpha$,
there must exist nonnegative numbers $(\alpha_{k})$ such that 
\[
D(Q_{X_{k}|X^{k-1}}\|P_{X}|Q_{X^{k-1}})\le\alpha_{k}
\]
and $\frac{1}{n}\sum_{k=1}^{n}\alpha_{k}=\alpha$. Similarly, from
Lemma \ref{lem:coupling-1}, we know that for $(Q_{X^{n}},Q_{Y^{n}})$
such that $\frac{1}{n}\C(Q_{X^{n}},Q_{Y^{n}})>\tau$, there must exist
nonnegative numbers $(\tau_{k})$ such that 
\[
\C(Q_{X_{k}|X^{k-1}},Q_{Y_{k}|Y^{k-1}}|Q_{X^{k-1}},Q_{Y^{k-1}})>\tau_{k}
\]
and $\frac{1}{n}\sum_{k=1}^{n}\tau_{k}=\tau$. These lead to that
for some sequence of nonnegative pairs $\left((\alpha_{k},\tau_{k})\right)$
such that $\frac{1}{n}\sum_{k=1}^{n}\alpha_{k}=\alpha,\frac{1}{n}\sum_{k=1}^{n}\tau_{k}=\tau$,
we have 
\[
E_{1}^{(n)}(\alpha,\tau)\ge\frac{1}{n}\sum_{k=1}^{n}\phi_{k}(\alpha_{k},\tau_{k},Q_{X^{k-1}},Q_{Y^{k-1}}),
\]
where 
\begin{align*}
 & \phi_{k}(\alpha_{k},\tau_{k},Q_{X^{k-1}},Q_{Y^{k-1}})\\
 & :=\underset{\substack{Q_{X_{k}|X^{k-1}},Q_{Y_{k}|Y^{k-1}}:\\
D(Q_{X_{k}|X^{k-1}}\|P_{X}|Q_{X^{k-1}})\le\alpha_{k},\\
\C(Q_{X_{k}|X^{k-1}},Q_{Y_{k}|Y^{k-1}}|Q_{X^{k-1}},Q_{Y^{k-1}})>\tau_{k}
}
}{\qquad\qquad\qquad\inf\;D(Q_{Y_{k}|Y^{k-1}}\|P_{Y}|Q_{Y^{k-1}}).}
\end{align*}

We now simplify the expression of $\phi_{k}(\alpha_{k},\tau_{k},Q_{X^{k-1}},Q_{Y^{k-1}})$.
Note that 
\[
\C(Q_{X_{k}|X^{k-1}},Q_{Y_{k}|Y^{k-1}}|Q_{X^{k-1}},Q_{Y^{k-1}})>\tau_{k}
\]
if and only if there exists a coupling $Q_{X^{k-1}Y^{k-1}}$ of $(Q_{X^{k-1}},Q_{Y^{k-1}})$
such that 
\[
\C(Q_{X_{k}|X^{k-1}},Q_{Y_{k}|Y^{k-1}}|Q_{X^{k-1}Y^{k-1}})>\tau_{k}.
\]
Therefore, 
\begin{align}
 & \phi_{k}(\alpha_{k},\tau_{k},Q_{X^{k-1}},Q_{Y^{k-1}})\nonumber \\
 & =\underset{\substack{Q_{X_{k}|X^{k-1}},Q_{Y_{k}|Y^{k-1}},Q_{X^{k-1}Y^{k-1}}\in\cC(Q_{X^{k-1}},Q_{Y^{k-1}}):\\
D(Q_{X_{k}|X^{k-1}}\|P_{X}|Q_{X^{k-1}})\le\alpha_{k},\\
\C(Q_{X_{k}|X^{k-1}},Q_{Y_{k}|Y^{k-1}}|Q_{X^{k-1}Y^{k-1}})>\tau_{k}
}
}{\qquad\qquad\inf\;D(Q_{Y_{k}|Y^{k-1}}\|P_{Y}|Q_{Y^{k-1}})}\nonumber \\
 & \ge\underset{\substack{Q_{X_{k}|X^{k-1}},Q_{Y_{k}|Y^{k-1}},Q_{X^{k-1}Y^{k-1}}:\\
D(Q_{X_{k}|X^{k-1}Y^{k-1}}\|P_{X}|Q_{X^{k-1}Y^{k-1}})\le\alpha_{k},\\
\C(Q_{X_{k}|X^{k-1}Y^{k-1}},Q_{Y_{k}|X^{k-1}Y^{k-1}}|Q_{X^{k-1}Y^{k-1}})>\tau_{k}
}
}{\qquad\qquad\inf\;D(Q_{Y_{k}|X^{k-1}Y^{k-1}}\|P_{Y}|Q_{X^{k-1}Y^{k-1}})}\label{eq:-1}\\
 & \ge\underset{\substack{Q_{X_{k}|X^{k-1}Y^{k-1}},Q_{Y_{k}|X^{k-1}Y^{k-1}},Q_{X^{k-1}Y^{k-1}}:\\
D(Q_{X_{k}|X^{k-1}Y^{k-1}}\|P_{X}|Q_{X^{k-1}Y^{k-1}})\le\alpha_{k},\\
\C(Q_{X_{k}|X^{k-1}Y^{k-1}},Q_{Y_{k}|X^{k-1}Y^{k-1}}|Q_{X^{k-1}Y^{k-1}})>\tau_{k}
}
}{\qquad\qquad\inf\;D(Q_{Y_{k}|X^{k-1}Y^{k-1}}\|P_{Y}|Q_{X^{k-1}Y^{k-1}})},\label{eq:-2}
\end{align}
where 
\begin{itemize}
\item in \eqref{eq:-1}, we denote 
\begin{align}
Q_{X_{k}|X^{k-1}Y^{k-1}} & =Q_{X_{k}|X^{k-1}},\,\label{eq:-19}\\
Q_{Y_{k}|X^{k-1}Y^{k-1}} & =Q_{Y_{k}|Y^{k-1}},\label{eq:-69}
\end{align}
and at the same time, we relax the coupling $Q_{X^{k-1}Y^{k-1}}$
of $(Q_{X^{k-1}},Q_{Y^{k-1}})$ to any joint distribution; 
\item in \eqref{eq:-2} we optimize over $(Q_{X_{k}|X^{k-1}Y^{k-1}},Q_{Y_{k}|X^{k-1}Y^{k-1}})$
directly, instead over $(Q_{X_{k}|X^{k-1}},Q_{Y_{k}|Y^{k-1}})$. (In
other words, we remove the constraints given in \eqref{eq:-19} and
\eqref{eq:-69} from the optimization in \eqref{eq:-2}.) 
\end{itemize}
Recall the expression of $\conv{\phi}(\alpha,\tau)$ in \eqref{eq:phi_lce-1}.
If we substitute $W\leftarrow(X^{k-1},Y^{k-1}),X\leftarrow X_{k},Y\leftarrow Y_{k}$
into \eqref{eq:-2}, then we obtain the expression in \eqref{eq:phi_lce-1}.
In other words, \eqref{eq:-2} is further lower bounded by $\conv{\phi}(\alpha_{k},\tau_{k})$.
Therefore, 
\begin{align*}
E_{1}^{(n)}(\alpha,\tau) & \ge\frac{1}{n}\sum_{k=1}^{n}\conv{\phi}(\alpha_{k},\tau_{k})\ge\conv{\phi}(\alpha,\tau).
\end{align*}

\begin{rem}
In fact, the single-letterization technique here was also used by
the author in \cite{YuTan2020_exact,yu2021strong_article,yu2023dimension}. 
\end{rem}

\subsection{Statement 2}

From the dimension-free bound in \eqref{eq:bound2}, $\liminf_{n\to\infty}E_{1}^{(n)}(\alpha,\tau)\ge\conv\phi(\alpha,\tau).$
We next prove $\limsup_{n\to\infty}E_{1}^{(n)}(\alpha,\tau)\leq\conv\phi(\alpha,\tau)$
by large deviations theory. Specifically, we choose $A$ and $B$
as conditional empirically typical sets and then analyze their exponents
by apply Sanov's theorem and estimate their distance by definition.

We assume that $\mathcal{W}$ is finite, and without loss of generality,
we assume $\supp(Q_{W})=[m]=\{1,2,\cdots,m\}$. (In fact, by the cardinality
bound for $\conv\phi$, we can assume $m\le3$.) Let $\epsilon>0$.
Let $(Q_{W},Q_{X|W},Q_{Y|W})$ be an optimal pair attaining $\conv\phi(\alpha-\epsilon,\tau+\epsilon)+\epsilon$.
That is, 
\begin{align*}
D(Q_{X|W}\|P_{X}|Q_{W}) & \le\alpha-\epsilon\\
\C(Q_{X|W},Q_{Y|W}|Q_{W}) & >\tau+\epsilon\\
D(Q_{Y|W}\|P_{Y}|Q_{W}) & \le\conv\phi(\alpha-\epsilon,\tau+\epsilon)+\epsilon.
\end{align*}

For each $n$, let $Q_{W}^{(n)}$ be an $n$-type (i.e., the empirical
measure of an $n$-length sequence) such that $\supp(Q_{W}^{(n)})\subseteq[m]$
and $Q_{W}^{(n)}\to Q_{W}$ as $n\to\infty$. Let $Q_{XW}^{(n)}:=Q_{W}^{(n)}Q_{X|W},\,Q_{YW}^{(n)}:=Q_{W}^{(n)}Q_{Y|W}$.
Let $w^{n}=(1,\cdots,1,2,\cdots,2,\cdots,m,\cdots,m)$ be an $n$-length
sequence, where $i$ appears $n_{i}:=nQ_{W}^{(n)}(i)$ times. Hence,
the empirical measure of $w^{n}$ is $Q_{W}^{(n)}$.

We now choose $A$ and $B$ as conditional empirically typical sets.
Specifically, for $\epsilon'>0$, 
\begin{align*}
A & =\mathcal{T}_{\epsilon'}^{(n)}(Q_{X|W}|w^{n})=\L_{n}^{-1}(\mathcal{A}|w^{n})=\prod_{w=1}^{m}\L_{n_{w}}^{-1}(\mathcal{A}_{w}),\\
B & =\mathcal{T}_{\epsilon'}^{(n)}(Q_{Y|W}|w^{n})=\L_{n}^{-1}(\mathcal{B}|w^{n})=\prod_{w=1}^{m}\L_{n_{w}}^{-1}(\mathcal{B}_{w}),
\end{align*}
where $\mathcal{A}_{w}:=B_{\epsilon']}(Q_{X|W=w})$, $\mathcal{B}_{w}:=B_{\epsilon']}(Q_{Y|W=w})$
for $w\in[m]$, $\mathcal{A}=B_{\epsilon']}(Q_{X|W})$, and $\mathcal{B}=B_{\epsilon']}(Q_{Y|W})$.
For each $w$, $\mathcal{A}_{w}$ is closed. Since the empirical measure
map $\L$ is continuous under the weak topology, $\L_{n_{w}}^{-1}(\mathcal{A}_{w})$
is closed in $\mathcal{X}^{n_{w}}$. Therefore, $A$ is closed in
$\mathcal{X}^{n}$. Similarly, $B$ is closed in $\mathcal{Y}^{n}$.

By Sanov's theorem, 
\begin{align}
 & \limsup_{n\to\infty}-\frac{1}{n}\log P_{X}^{\otimes n}(A)\nonumber \\
 & =\sum_{w}Q_{W}(w)\limsup_{n\to\infty}-\frac{1}{n_{w}}\log P_{X}^{\otimes n_{w}}(\L_{n_{w}}^{-1}(\mathcal{A}_{w}))\nonumber \\
 & \leq\sum_{w}Q_{W}(w)\inf_{R_{X}\in\mathcal{A}_{w}^{o}}D(R_{X}\|P_{X})\nonumber \\
 & \le\sum_{w}Q_{W}(w)D(Q_{X|W=w}\|P_{X})\nonumber \\
 & =D(Q_{X|W}\|P_{X}|Q_{W})\nonumber \\
 & \le\alpha-\epsilon.\label{eq:-101}
\end{align}
Hence, $-\frac{1}{n}\log P_{X}^{\otimes n}(A)\le\alpha$ for all sufficiently
large $n$. Similarly, 
\begin{align}
-\frac{1}{n}\log P_{Y}^{\otimes n}(B) & \le D(Q_{Y|W}\|P_{Y}|Q_{W})\nonumber \\
 & \le\conv\phi(\alpha-\epsilon,\tau+\epsilon)+2\epsilon\label{eq:-31-4}
\end{align}
for all sufficiently large $n$.

We next estimate the distance between $A$ and $B$, and show that
\begin{equation}
c_{n}(x^{n},y^{n})>n\tau,\quad\forall x^{n}\in A,y^{n}\in B.\label{eq:-116}
\end{equation}
Observe that for $\L_{x^{n}|w^{n}}\in\mathcal{A},\L_{y^{n}|w^{n}}\in\mathcal{B},$
\begin{align*}
\frac{1}{n}c_{n}(x^{n},y^{n}) & =\mathbb{E}_{\L_{x^{n},y^{n},w^{n}}}c(X,Y)\\
 & \ge\C(\L_{x^{n}|w^{n}},\L_{y^{n}|w^{n}}|\L_{w^{n}})\\
 & \ge\inf_{R_{X|W}\in\mathcal{A},R_{Y|W}\in\mathcal{B}}\C(R_{X|W},R_{Y|W}|\L_{w^{n}})\\
 & =\mathbb{E}_{W\sim\L_{w^{n}}}\inf_{R_{X}\in\mathcal{A}_{W},R_{Y}\in\mathcal{B}_{W}}\C(R_{X},R_{Y})\\
 & \to\mathbb{E}_{W\sim Q_{W}}\inf_{R_{X}\in\mathcal{A}_{W},R_{Y}\in\mathcal{B}_{W}}\C(R_{X},R_{Y})\\
 & =\inf_{R_{X|W}\in\mathcal{A},R_{Y|W}\in\mathcal{B}}\C(R_{X|W},R_{Y|W}|Q_{W})\\
 & =:\eta.
\end{align*}
So, it remains to show that $\eta>\tau.$

By Assumption 1, we obtain that 
\begin{align}
\eta & \geq\inf_{R_{Y|W}\in\mathcal{B}}\C(Q_{X|W},R_{Y|W}|Q_{W})-\delta(\epsilon'),\label{eq:-29-2}
\end{align}
where $\delta(\epsilon')$ is positive and vanishes as $\epsilon'\downarrow0$.

Observe that\footnote{This is because $R_{Y|W}\mapsto\C(Q_{X|W},R_{Y|W}|Q_{W})$ is a convex
combination of lower semi-continuous functions $R_{Y|W=w}\mapsto\C(Q_{X|W=w},R_{Y|W=w})$.
So, $\C(Q_{X|W},R_{Y|W}|Q_{W})$ is lower semi-continuous as well
in $\mathcal{P}(\mathcal{Y})^{m}$ equipped with the product topology.
Hence, its strict superlevel sets are open.} 
\begin{align*}
\mathcal{B}_{0} & :=\{(R_{Y|W=w})_{w\in[m]}\in\mathcal{P}(\mathcal{Y})^{m}:\\
 & \qquad\qquad\C(Q_{X|W},R_{Y|W}|Q_{W})>\tau+\epsilon\}
\end{align*}
is open in $\mathcal{P}(\mathcal{Y})^{m}$ equipped with the product
topology. Since $Q_{Y|W}\in\mathcal{B}_{0}$, $\mathcal{B}_{0}$ contains
the product of $\mathcal{F}_{w},w\in[m]$ for some open sets $\mathcal{F}_{w}\subseteq\mathcal{P}(\mathcal{Y})$
such that $Q_{Y|W=w}\in\mathcal{F}_{w}$. So $\mathcal{B}_{w}\subseteq\mathcal{F}_{w},\forall w$,
for sufficiently small $\epsilon'$ (which was used in the definition
of $\mathcal{B}_{w}$), which means in this case, $\mathcal{B}\subseteq\mathcal{B}_{0}$.
This implies that the RHS of \eqref{eq:-29-2} is further lower bounded
by $\tau+\epsilon-\delta(\epsilon')$. So, if we let $\epsilon>0$
be fixed and $\epsilon'>0$ be sufficiently small such that $\epsilon>\delta(\epsilon')$,
then for sufficiently large $n$, we have \eqref{eq:-116}.

Lastly, combining \eqref{eq:-101}, \eqref{eq:-31-4}, and \eqref{eq:-116}
yields that $\limsup_{n\to\infty}E_{1}^{(n)}(\alpha,\tau)\leq\conv\phi(\alpha-\epsilon,\tau+\epsilon)+\epsilon.$
Since $\conv\phi$ is convex, it is continuous on the interior of
$\mathrm{dom}\conv\phi$. We hence have that for all $\alpha,\tau>0$,
$\limsup_{n\to\infty}E_{1}^{(n)}(\alpha,\tau)\leq\conv\phi(\alpha,\tau).$

\subsection{Statement 3}

The lower bound follows by the dimension-free bound in \eqref{eq:bound2}.
We next prove the upper bound. For this case, we set $\alpha=0$ in
the proof above, and re-choose $(Q_{W},Q_{Y|W})$ as an optimal pair
attaining $\conv\varphi(\tau+\epsilon)+\epsilon$. That is, 
\begin{align*}
\C(P_{X},Q_{Y|W}|Q_{W}) & >\tau+\epsilon\\
D(Q_{Y|W}\|P_{Y}|Q_{W}) & \le\conv\varphi(\tau+\epsilon)+\epsilon.
\end{align*}
On one hand, we choose $\mathcal{A}:=B_{\epsilon']}(P_{X})$ for $\epsilon'>0$,
and $A=\L_{n}^{-1}(\mathcal{A}|w^{n})$. Then, we have

\begin{align*}
 & \limsup_{n\to\infty}-\frac{1}{n}\log(1-P_{X}^{\otimes n}(A))\\
 & \geq\inf_{Q_{X}\in\overline{\mathcal{A}^{c}}}D(Q_{X}\|P_{X})\\
 & \geq\inf_{Q_{X}:\Levy(Q_{X},P_{X})\ge\epsilon'/2}D(Q_{X}\|P_{X})\\
 & \ge\epsilon'^{2}/2,
\end{align*}
where the last inequality follows since $D(Q_{X}\|P_{X})\ge2\Levy(Q_{X},P_{X})^{2}$
(see \eqref{eq:D-TV-L}). Hence, for fixed $\epsilon'>0$, $P_{X}^{\otimes n}(A)\to1$
as $n\to+\infty$ exponentially fast.

On the other hand, we retain the choices of $\mathcal{B}_{w}$ and
$\mathcal{B}$. Similarly to \eqref{eq:-31-4}, we obtain 
\begin{equation}
-\frac{1}{n}\log P_{Y}^{\otimes n}(B)\le\conv\varphi(\tau+\epsilon)+2\epsilon\label{eq:-31-1-2}
\end{equation}
for all sufficiently large $n$.

Similarly to the above, it can be shown that $\frac{1}{n}c_{n}(x^{n},y^{n})>\tau$
for sufficiently large $n$. We hence have that for all $\tau>0$,
$\limsup_{n\to\infty}E_{1}^{(n)}(\alpha_{n},\tau)\leq\conv\varphi(\tau).$

\section{\label{sec:Proof-of-Theorem-1-1}Proof of Theorem \ref{thm:LD-2}}

Statement 1 in Theorem \ref{thm:LD-2} is a restatement of Statement
1 in Theorem \ref{thm:LD} for the case of $c=d^{p}$. We next prove
Statements 2 and 3.

Statement 2 (Case $\alpha>0$): From the dimension-free bound in \eqref{eq:bound2},
$\liminf_{n\to\infty}E_{1}^{(n)}(\alpha,\tau)\ge\conv{\phi}(\alpha,\tau).$
We next prove $\limsup_{n\to\infty}E_{1}^{(n)}(\alpha,\tau)\leq\conv{\phi}(\alpha,\tau).$

Let $s>0$, and $d_{s}:=\min\{d,s\}$. Then, $d_{s}$ is a bounded
metric on $\mathcal{X}$. This is just Example \ref{exa:2} given
below Assumption 1 which satisfies Assumption 1. So, by Theorem \ref{thm:LD},
when we set $c=d_{s}^{p}$, we have $\limsup_{n\to\infty}E_{1,s}^{(n)}(\alpha,\tau)\leq\conv{\phi}_{s}(\alpha,\tau)$,
where $E_{1,s}^{(n)}(\alpha,\tau)$ is the quantity $E_{1}^{(n)}(\alpha,\tau)$
given in \eqref{eq:E1} but defined for $c=d_{s}^{p}$, and similarly,
$\conv{\phi}_{s}(\alpha,\tau)$ is the $\conv{\phi}(\alpha,\tau)$
defined for $c=d_{s}^{p}$. Explicitly, 
\begin{equation}
\conv{\phi}_{s}(\alpha,\tau)=\inf_{\substack{Q_{X|W},Q_{Y|W},Q_{W}:\\
D(Q_{X|W}\|P_{X}|Q_{W})\le\alpha,\\
\C_{s}(Q_{X|W},Q_{Y|W}|Q_{W})>\tau
}
}D(Q_{Y|W}\|P_{Y}|Q_{W})\label{eq:phi_lce-1-2}
\end{equation}
where $\C_{s}(Q_{X|W},Q_{Y|W}|Q_{W})$ is the OT cost for $c=d_{s}^{p}$.

Observe that for the same $A$, 
\begin{align*}
A^{t} & =\bigcup_{x^{n}\in A}\{y^{n}\in\mathcal{Y}^{n}:\sum_{i=1}^{n}d^{p}(x_{i},y_{i})\leq t\}\\
 & \subseteq\bigcup_{x^{n}\in A}\{y^{n}\in\mathcal{Y}^{n}:\sum_{i=1}^{n}d_{s}^{p}(x_{i},y_{i})\leq t\}=:A_{s}^{t}
\end{align*}
So, $E_{1}^{(n)}(\alpha,\tau)\le E_{1,s}^{(n)}(\alpha,\tau).$ Hence,
$\limsup_{n\to\infty}E_{1}^{(n)}(\alpha,\tau)\leq\conv{\phi}_{s}(\alpha,\tau)$.
Taking limit as $s\to\infty$, we obtain $\limsup_{n\to\infty}E_{1}^{(n)}(\alpha,\tau)\leq\lim_{s\to\infty}\conv{\phi}_{s}(\alpha,\tau)$.
To prove Statement 2, it suffices to show that $\lim_{s\to\infty}\conv{\phi}_{s}(\alpha,\tau)=\conv{\phi}(\alpha,\tau)$
for $\alpha,\tau>0$. On one hand, $\conv{\phi}_{s}(\alpha,\tau)\ge\conv{\phi}(\alpha,\tau)$
since $\C_{s}(Q_{X|W},Q_{Y|W}|Q_{W})\le\C(Q_{X|W},Q_{Y|W}|Q_{W})$.
So, it suffices to prove $\lim_{s\to\infty}\conv{\phi}_{s}(\alpha,\tau)\le\conv{\phi}(\alpha,\tau)$
for $\alpha,\tau>0$.

Let $\epsilon>0$. Let $(Q_{W},Q_{X|W},Q_{Y|W})$ $\epsilon$-approximately
attain $\conv{\phi}(\alpha,\tau)$ in the sense that 
\begin{align*}
D(Q_{X|W}\|P_{X}|Q_{W}) & \le\alpha\\
\C(Q_{X|W},Q_{Y|W}|Q_{W}) & >\tau\\
D(Q_{Y|W}\|P_{Y}|Q_{W}) & \le\conv{\phi}(\alpha,\tau)+\epsilon.
\end{align*}

\begin{lem}
Given $(Q_{X},Q_{Y})$, 
\[
\lim_{s\to\infty}\C_{s}(Q_{X},Q_{Y})=\C(Q_{X},Q_{Y}).
\]
\end{lem}
\begin{proof} Obviously, $\C_{s}(Q_{X},Q_{Y})\le\C(Q_{X},Q_{Y}).$
Hence, $\lim_{s\to\infty}\C_{s}(Q_{X},Q_{Y})\le\C(Q_{X},Q_{Y}).$

By Kantorovich duality \cite[Theorem 5.10]{villani2008optimal} (also
given in Lemma \ref{lem:Kantorovich}), 
\begin{align}
\C(Q_{X},Q_{Y}) & =\sup_{\substack{(f,g)\in C_{\mathrm{b}}(\mathcal{X})\times C_{\mathrm{b}}(\mathcal{Y}):\\
f+g\leq c
}
}\int_{\mathcal{X}}f\ \mathrm{d}Q_{X}+\int_{\mathcal{Y}}g\ \mathrm{d}Q_{Y}\label{eq:-111}
\end{align}
where $C_{\mathrm{b}}(\mathcal{X})$ denotes the collection of bounded
continuous functions $f:\mathcal{X}\to\mathbb{R}$. Given $\epsilon>0$,
let $(f^{*},g^{*})\in C_{\mathrm{b}}(\mathcal{X})\times C_{\mathrm{b}}(\mathcal{Y})$
$\epsilon$-approximately attain the supremum above in the sense that
\begin{align*}
f^{*}+g^{*} & \leq c\\
\int_{\mathcal{X}}f^{*}\ \mathrm{d}Q_{X}+\int_{\mathcal{Y}}g^{*}\ \mathrm{d}Q_{Y} & \ge\C(Q_{X},Q_{Y})-\epsilon.
\end{align*}
Then, by the boundness, $f^{*}+g^{*}\leq c_{s}$ for all sufficiently
large $s$. By Kantorovich duality again, 
\begin{align}
\C_{s}(Q_{X},Q_{Y}) & =\sup_{\substack{(f,g)\in C_{\mathrm{b}}(\mathcal{X})\times C_{\mathrm{b}}(\mathcal{Y}):\\
f+g\leq c_{s}
}
}\int_{\mathcal{X}}f\ \mathrm{d}Q_{X}+\int_{\mathcal{Y}}g\ \mathrm{d}Q_{Y}.\label{eq:-111-1}
\end{align}
For sufficiently large $s$, $(f^{*},g^{*})$ is a feasible solution
to \eqref{eq:-111-1}. Hence, 
\[
\C_{s}(Q_{X},Q_{Y})\ge\int_{\mathcal{X}}f^{*}\ \mathrm{d}Q_{X}+\int_{\mathcal{Y}}g^{*}\ \mathrm{d}Q_{Y}\ge\C(Q_{X},Q_{Y})-\epsilon.
\]
Since $\epsilon>0$ is arbitrary, $\lim_{s\to\infty}\C_{s}(Q_{X},Q_{Y})\ge\C(Q_{X},Q_{Y})$,
completing the proof. \end{proof} Since by definition, the conditional
OT cost is the weighted sum of the unconditional version, given $\left(Q_{W},Q_{X|W},Q_{Y|W}\right)$,
we immediately have 
\[
\lim_{s\to\infty}\C_{s}(Q_{X|W},Q_{Y|W}|Q_{W})=\C(Q_{X|W},Q_{Y|W}|Q_{W})>\tau.
\]
So, for sufficiently large $s$, $\C_{s}(Q_{X|W},Q_{Y|W}|Q_{W})>\tau$
which means that $\left(Q_{W},Q_{X|W},Q_{Y|W}\right)$ is a feasible
solution to the infimization in \eqref{eq:phi_lce-1-2} with $\alpha$
substituted by $\alpha-\epsilon$. Therefore, 
\[
\lim_{s\to\infty}\conv{\phi}_{s}(\alpha,\tau)\le D(Q_{Y|W}\|P_{Y}|Q_{W})\le\conv{\phi}(\alpha,\tau)+\epsilon.
\]
Letting $\epsilon\downarrow0$, we obtain $\lim_{s\to\infty}\conv{\phi}_{s}(\alpha,\tau)\le\conv{\phi}(\alpha,\tau).$
This completes the proof.

Statement 3 (Case $\alpha_{n}\to0$): The proof for the upper bound
is similar to the above for Statement 2, and hence is omitted here.

We next prove $\liminf_{n\to\infty}E_{1}^{(n)}(\alpha_{n},\tau)\ge\conv\varphi_{X}(\tau)$.
The proof is essentially same as Marton's in \cite{Marton86} or Gozlan's
in \cite{gozlan2005principe}. From the dimension-free bound in \ref{thm:LD},
we have for fixed $\tau$, $E_{1}^{(n)}(\alpha_{n},\tau)\geq\conv\phi(\alpha_{n},\tau).$
Under the condition $D(Q_{X|W}\|P_{X}|Q_{W})\le\alpha_{n}$, we have
\[
\C(Q_{X|W}\|P_{X}|Q_{W})\le\conc\kappa_{X}(\alpha_{n}),
\]
where $\conc\kappa_{X}$ is the upper concave envelope of 
\begin{equation}
\kappa_{X}(\alpha):=\sup_{Q_{X}:D(Q_{X}\|P_{X})<\alpha}\C(P_{X},Q_{X}).\label{eq:-32-2-1}
\end{equation}
The generalized inverse of $\conv{\varphi}_{X}$ is for $\alpha\ge0$,
\begin{align}
\conv{\varphi}_{X}^{-}(\alpha) & :=\inf\left\{ \tau\ge0:\varphi_{X}(\tau)\ge\alpha\right\} \\
 & =\conc\kappa_{X}(\alpha).
\end{align}

By Assumption 2, $\conc\kappa_{X}(\alpha)\to0$ as $\alpha\to0$.
By the triangle inequality (since for this case, $\C^{1/p}(\cdot,\cdot)$
is a Wasserstein metric), we then have that for $(Q_{X|W},Q_{Y|W},Q_{W})$
satisfying the constraints in \eqref{eq:phi_lce-1}, 
\begin{align*}
 & \C^{1/p}(P_{X},Q_{Y|W}|Q_{W})\\
 & \ge\C^{1/p}(Q_{X|W},Q_{Y|W}|Q_{W})-\C^{1/p}(Q_{X|W},P_{X}|Q_{W})\\
 & >\tau^{1/p}-\conc\kappa_{X}(\alpha_{n})^{1/p}.
\end{align*}
We finally obtain 
\[
\conv\phi(\alpha_{n},\tau)\ge\conv\varphi\big((\tau^{1/p}-\conc\kappa_{X}(\alpha_{n})^{1/p})^{p}\big).
\]
Letting $n\to\infty$, $\liminf_{n\to\infty}\conv\phi(\alpha_{n},\tau)\ge\conv\varphi(\tau)$
for $\tau>0$. Hence, $\liminf_{n\to\infty}E_{1}^{(n)}(\alpha_{n},\tau)\ge\conv\varphi(\tau).$

\section{\label{sec:Proof-of-Theorem-1-1-1}Proof of Theorem \ref{thm:LD-3}}

Statement 1: Observe that given any $A$, by the DPI, it holds that
$D(Q_{X}\|P_{X})\ge D(Q_{X}(A)\|P_{X}(A))$ and $D(Q_{Y}\|P_{Y})\ge D(Q_{Y}(A)\|P_{Y}(A))$.
Therefore, for $\alpha\ge0,\tau\in[0,1]$, it holds that 
\begin{align}
\phi(\alpha,\tau) & =\inf_{A}\inf_{\substack{Q_{X}(A),Q_{Y}(A):\\
D(Q_{X}(A)\|P_{X}(A))\le\alpha,\\
Q_{X}(A)-Q_{Y}(A)>\tau
}
}D(Q_{Y}(A)\|P_{Y}(A))\label{eq:phi-4}\\
 & =\inf_{A}\theta_{\alpha,\tau}(P_{X}(A),P_{Y}(A)).
\end{align}
That is, for the Hamming metric, \eqref{eq:-64} holds. By \eqref{eq:bound2},
it holds that $E_{1}^{(n)}(\alpha,\tau)\ge\phi(\alpha,\tau)$. 

If $P_{X}$ is finitely-supported or atomless, then the infimum in
\eqref{eq:omega} is attained. This is because, for the finitely-supported
case, $A$ in \eqref{eq:omega} can be restricted to be a subset of
the support of $P_{X}$ (which is a finite set); for the atomless
case, by the Neyman-Pearson lemma, the optimal set $A$ in \eqref{eq:omega}
is any set such that $P_{X}(A)=p$ and 
\begin{align*}
 & \big\{ x:\d P_{Y}/\d R(x)<r\d P_{X}/\d R(x)\big\}\\
 & \subseteq A\\
 & \subseteq\big\{ x:\d P_{Y}/\d R(x)\le r\d P_{X}/\d R(x)\big\}
\end{align*}
for some $r\ge0$, where $R$ is an arbitrary probability measure
such that $P_{X},P_{Y}\ll R$. For finitely-supported or atomless
$P_{X}$, 
\begin{align}
\phi(\alpha,\tau) & =\inf_{A}\theta_{\alpha,\tau}(P_{X}(A),P_{Y}(A))\label{eq:-64-2}\\
 & =\inf_{p\in[0,1]}\inf_{A:P_{X}(A)=p}\theta_{\alpha,\tau}(P_{X}(A),P_{Y}(A))\\
 & =\inf_{p\in[0,1]}\inf_{A:P_{X}(A)=p}\theta_{\alpha,\tau}(p,\omega(p))\label{eq:-146}\\
 & =\inf_{p\in[0,1]:\omega(p)<\infty}\theta_{\alpha,\tau}(p,\omega(p)),
\end{align}
where \eqref{eq:-146} follows since on one hand, from \eqref{eq:-82},
it is observed that given $p$, $\theta_{\alpha,\tau}(p,q)$ is nondecreasing
in $q$, and on the other hand, the infimum in \eqref{eq:omega} is
attained. 

Statement 2: By Statement 1, if $\phi(\alpha,\tau)=\infty$, then
$E_{1}^{(n)}(\alpha,\tau)=\infty$. So, it suffices to consider the
case $\phi(\alpha,\tau)<\infty$. Moreover, if $\tau=1$, then for
any nonempty set $A\subseteq\mathcal{X}^{n}$, its $n$-enlargement
is always $\mathcal{X}^{n}$. So, $\Gamma^{(n)}(a,n)=1$ for any $a>0$,
and hence $E_{1}^{(n)}(\alpha,1)=\infty$ for any finite $\alpha$,
which coincides with $\phi(\alpha,1)=1$. It remains to consider the
case of $\alpha>0,\tau\in(0,1)$ and $\phi(\alpha,\tau)<\infty$. 

Given $(\alpha,\tau)$ and $\epsilon>0$, let $A$ be a set that $\epsilon$-approximately
attains the infimum in \eqref{eq:-64}. That is, $A$ satisfies 
\begin{equation}
\theta_{\alpha,\tau}(p,q)\le\phi(\alpha,\tau)+\epsilon,\label{eq:-4}
\end{equation}
where $p:=P_{X}(A)$, $q:=P_{Y}(A)$. We partition $\mathcal{X}$
into $\{A,A^{c}\}$. Let $X\sim P_{X},Y\sim P_{Y}$. Denote $I=1$
if $X\in A$; $I=0$ otherwise. Denote $J=1$ if $Y\in A$; $J=0$
otherwise. So, $I$ and $J$ are random variables, whose distributions
are given by $P_{I}=\Bern(p)$ and $P_{J}=\Bern(q)$. Define the isoperimetric
function for $P_{I}$, $P_{J}$, and Hamming metric as for $a\in[0,1],t\ge0$,
\begin{equation}
\Gamma_{\mathrm{b}}^{(n)}(a,t):=\inf_{B\subseteq\mathbb{N}^{n}:P_{I}^{\otimes n}(B)\ge a}P_{J}^{\otimes n}(B^{t}),\label{eq:Gamma-4-2}
\end{equation}
and concentration exponent as for $\alpha\ge0$, 
\begin{align}
E_{\mathrm{b},1}^{(n)}(\alpha,\tau) & :=-\frac{1}{n}\log(1-\Gamma_{\mathrm{b}}^{(n)}(e^{-n\alpha},n\tau)).\label{eq:E1-1}
\end{align}
Here the subscript ``$\mathrm{b}$'' denotes ``Bernoulli'' or
``binary''. By definition, $\Gamma^{(n)}(a,t)\le\Gamma_{\mathrm{b}}^{(n)}(a,t)$,
and hence, $E_{1}^{(n)}(\alpha,\tau)\le E_{\mathrm{b},1}^{(n)}(\alpha,\tau)$. 

Since the space $\{0,1\}$ with the Hamming metric satisfies Example
\ref{exa:1} given below Assumption 1, by Theorem \ref{thm:LD}, for
distributions $P_{I},P_{J}$, and any $(\alpha,\tau)$ in the interior
of $\mathrm{dom}\conv\phi_{\mathrm{b}}$, it holds that $\lim_{n\to\infty}E_{\mathrm{b},1}^{(n)}(\alpha,\tau)=\conv\phi_{\mathrm{b}}(\alpha,\tau)$,
where 
\[
\phi_{\mathrm{b}}(\alpha,\tau):=\inf_{s,t\in[0,1]:D(s\|p)\le\alpha,\,|s-t|>\tau}D(t\|q).
\]
By definition, $\phi_{\mathrm{b}}(\alpha,\tau)\le\theta_{\alpha,\tau}(p,q)$.
Combining it with $\theta_{\alpha,\tau}(p,q)\le\phi(\alpha,\tau)+\epsilon$
and $E_{1}^{(n)}(\alpha,\tau)\le E_{\mathrm{b},1}^{(n)}(\alpha,\tau)$
yields $\limsup_{n\to\infty}E_{1}^{(n)}(\alpha,\tau)\le\phi_{\mathrm{b}}(\alpha,\tau)+\epsilon$.
Since $\epsilon>0$ is arbitrary, it holds that $\limsup_{n\to\infty}E_{1}^{(n)}(\alpha,\tau)\le\phi_{\mathrm{b}}(\alpha,\tau)$,
which combined with $E_{1}^{(n)}(\alpha,\tau)\ge\phi(\alpha,\tau)$
yields $\lim_{n\to\infty}E_{1}^{(n)}(\alpha,\tau)=\phi(\alpha,\tau)$. 

Lastly, we verify that $(\alpha,\tau)$ is in the interior of $\mathrm{dom}\conv\phi_{\mathrm{b}}$.
It is easy to see that 
\begin{align*}
\phi_{\mathrm{b}}(\alpha,\tau) & \le\bar{\phi}_{\mathrm{b}}(\alpha,\tau):=\theta_{\alpha,\tau}(p,q)\\
 & =\inf_{s,t\in[0,1]:D(s\|p)\le\alpha,\,s-t>\tau}D(t\|q),
\end{align*}
and $\bar{\phi}_{\mathrm{b}}$ is convex. So, $\conv\phi_{\mathrm{b}}\le\bar{\phi}_{\mathrm{b}}$.
By the finiteness of $\phi(\alpha,\tau)$ and \eqref{eq:-4}, it holds
that $(\alpha,\tau)\in\mathrm{dom}\bar{\phi}_{\mathrm{b}}$. 

If $(p,q)\in(0,1)^{2}$, then by the expression in \eqref{eq:-82},
$\mathrm{dom}\bar{\phi}_{\mathrm{b}}=\{(\alpha,\tau)\in[0,\infty)\times[0,1):D(\tau\|p)<\alpha\}$.
So, any $(\alpha,\tau)\in\mathrm{dom}\bar{\phi}_{\mathrm{b}}$ such
that $\alpha>0,\tau\in(0,1)$ must be in the interior of $\mathrm{dom}\bar{\phi}_{\mathrm{b}}$,
and hence, in the interior of $\mathrm{dom}\conv\phi_{\mathrm{b}}$. 

If $p=0$ or $p=q=1$, then $\bar{\phi}_{\mathrm{b}}(\alpha,\tau)=\theta_{\alpha,\tau}(p,q)=\infty$
since $\tau>0$. This contradicts with $(\alpha,\tau)\in\mathrm{dom}\bar{\phi}_{\mathrm{b}}$.
So, this case cannot occur. 

If $p=1$ and $q\in[0,1)$, then $\mathrm{dom}\bar{\phi}_{\mathrm{b}}=\{(\alpha,\tau)\in[0,\infty)\times[0,1)\}$.
So, any $(\alpha,\tau)$ such that $\alpha>0,\tau\in(0,1)$ must be
in the interior of $\mathrm{dom}\bar{\phi}_{\mathrm{b}}$, and hence,
in the interior of $\mathrm{dom}\conv\phi_{\mathrm{b}}$. 

Combining all the cases above implies $(\alpha,\tau)$ is in the interior
of $\mathrm{dom}\conv\phi_{\mathrm{b}}$, completing the proof of
Statement 2. 

Statement 3: The proof of Statement 3 is similar to that of Statement
2. We first quantize the random variables $X,Y$ into Bernoulli random
variables $I,J$, and then apply Statement 3 of Theorem \ref{thm:LD-2}
(or Statement 3 of Theorem \ref{thm:LD}) to $P_{I},P_{J}$, yielding
the desired formula. We omit the proof details. 

\section{\label{sec:Proof-of-Theorem-2}Proof of Theorem \ref{thm:LD-1}}

\subsection{Statement 1 }

We first prove the upper bound for the case of finite $\mathcal{X}$
(and Polish $\mathcal{Y}$), and then generalize it to compact $\mathcal{X}$
and further to Polish $\mathcal{X}$. 

\subsubsection{\label{subsec:Finite}Finite $\mathcal{X}$}

We first consider that $\mathcal{X}$ is a finite metric space. For
this case, we extend Ahlswede, Yang, and Zhang's method \cite{ahlswede1997identification,ahlswede1999asymptotical}
to the case in which $\mathcal{Y}$ is an arbitrary Polish space (but
$\mathcal{X}$ is still a finite metric space). We divide the proof
into four steps. 

For this case, we prove 
\begin{equation}
\limsup_{n\to\infty}E_{0}^{(n)}(\alpha,\tau)\leq\psi(\alpha,\tau).\label{eq:-17-1-1}
\end{equation}
We assume $\psi(\alpha,\tau)<\infty$, since otherwise, the inequality
holds trivially. 

\textbf{Step 1: Inherently Typical Subset Lemma }

In our proof, we utilize the inherently typical subset lemma in \cite{ahlswede1997identification,ahlswede1999asymptotical}.
We now introduce this lemma. Let $A$ be any subset of $\mathcal{X}^{n}$.
For any $0\le i\le n-1$, define 
\[
A_{i}=\left\{ x^{i}\in\mathcal{X}^{i}:x^{i}\textrm{ is a prefix of some element of }A\right\} ,
\]
which is the projection of $A$ to the space $\mathcal{X}^{i}$ of
the first $i$ components. 
\begin{defn}
$A\subseteq\mathcal{X}^{n}$ is called $m$-inherently typical if
there exist a set $\mathcal{W}_{m}$ with $|\mathcal{W}_{m}|\le(m+1)^{|\mathcal{X}|}$
and $n$ mappings $\phi_{i}:A_{i}\to\mathcal{W}_{m},\;i\in[0:n-1]$
such that the following hold: \\
 (i) There exists a distribution (empirical measure) $Q_{XW}$ such
that for any $x^{n}\in A$, 
\[
\L_{x^{n}w^{n}}=Q_{XW}
\]
where $w^{n}$ is a sequence defined by $w_{i}=\phi_{i}(x^{i-1})$
for all $1\le i\le n$. Such a sequence is called a sequence associated
with $x^{n}$ through $(\phi_{i})$. \\
 (ii) 
\begin{equation}
H_{Q}(X|W)-\frac{\log^{2}m}{m}\le\frac{1}{n}\log|A|\le H_{Q}(X|W).\label{eq:-8}
\end{equation}
\end{defn}
For an $m$-inherently typical set $A$, let $Q_{X^{n}}$ be the uniform
distribution on $A$. We now give another interpretation of the $m$-inherently
typical set in the language of sufficient statistics. Let $W_{i}=\phi_{i}(X^{i-1})$.
First, observe that 
\begin{align*}
\frac{1}{n}\log|A| & =H_{Q}(X^{n})\\
 & =\sum_{i=1}^{n}H_{Q}(X_{i}|X^{i-1})\\
 & =\sum_{i=1}^{n}H_{Q}(X_{i}|X^{i-1},W_{i})\\
 & =H_{Q}(X_{K}|X^{K-1},W_{K},K)
\end{align*}
where $K$ is a random time index uniformly distributed over $[n]$
which is independent of $X^{n}$. Moreover, 
\begin{align}
Q_{X_{K},W_{K}} & =\mathbb{E}_{(X^{n},W^{n})\sim Q_{X^{n},W^{n}}}[Q_{X_{K},W_{K}|X^{n},W^{n}}]\nonumber \\
 & =\mathbb{E}_{(X^{n},W^{n})\sim Q_{X^{n},W^{n}}}[\L_{X^{n},W^{n}}]\nonumber \\
 & =Q_{X,W}.\label{eq:distletter}
\end{align}
Hence, the inequalities in \eqref{eq:-8} can be rewritten as 
\begin{equation}
0\le I_{Q}(X_{K};X^{K-1},K|W_{K})\le\frac{\log^{2}m}{m}.\label{eq:mi}
\end{equation}
The first inequality holds trivially since mutual information is nonnegative.
For sufficiently large $m$, the bound $\frac{\log^{2}m}{m}$ is sufficiently
small. Hence, $I_{Q}(X_{K};X^{K-1},K|W_{K})$ is close to zero. In
this case, $X_{K}$ and $(X^{K-1},K)$ are approximately conditionally
independent given $W_{K}$. In other words, $W_{K}$ is an approximate
sufficient statistic for ``underlying parameter'' $X_{K}$; we refer
readers to \cite[Section 2.9]{Cov06} for sufficient statistics and
\cite{hayashi2017minimum} for approximate versions.

As for $m$-inherent typical sets, one of the most important results
is the inherently typical subset lemma, which concerns the existence
of inherent typical sets. Such a lemma was proven by Ahlswede, Yang,
and Zhang \cite{ahlswede1997identification,ahlswede1999asymptotical}. 
\begin{lem}[Inherently Typical Subset Lemma]
For any $m\ge2^{16|\mathcal{X}|^{2}}$, $n$ satisfying $\left((m+1)^{5|\mathcal{X}|+4}\log(n+1)\right)/n\le1$,
and any $A\subseteq\mathcal{X}^{n}$, there exists an $m$-inherently
typical subset $\tilde{A}\subseteq A$ such that 
\[
0\le\frac{1}{n}\log\frac{|A|}{|\tilde{A}|}\leq|\mathcal{X}|(m+1)^{|\mathcal{X}|}\frac{\log(n+1)}{n}.
\]
\end{lem}
\textbf{Step 2: Multi-letter Bound}

For any $A\subseteq\mathcal{X}^{n}$, denote $A_{Q_{X}}:=A\cap\{x^{n}:\L_{x^{n}}=Q_{X}\}$
for empirical measure $Q_{X}$. Since $A=\bigcup_{Q_{X}}A_{Q_{X}}$
and the number of distinct types is no more than $(n+1)^{|\mathcal{X}|}$,
by the pigeonhole principle, we have 
\[
P_{X}^{\otimes n}(A_{Q_{X}})\geq P_{X}^{\otimes n}(A)(n+1)^{-|\mathcal{X}|}
\]
for some empirical measure $Q_{X}$.

By the lemma above, given $m\ge2^{16|\mathcal{X}|^{2}}$, for all
sufficiently large $n$, there exists an $m$-inherently typical subset
$\tilde{A}\subseteq A_{Q_{X}}$ such that 
\[
|\tilde{A}|\ge|A_{Q_{X}}|\cdot(n+1)^{-b}
\]
where $b=|\mathcal{X}|(m+1)^{|\mathcal{X}|}$. Observe that for any
$B\subseteq\{x^{n}:\L_{x^{n}}=Q_{X}\}$, we have $P_{X}^{\otimes n}(B)=|B|e^{n\sum_{x}Q_{X}(x)\log P_{X}(x)}$.
Hence, 
\[
P_{X}^{\otimes n}(\tilde{A})\ge P_{X}^{\otimes n}(A_{Q_{X}})(n+1)^{-b}\geq P_{X}^{\otimes n}(A)(n+1)^{-b'}
\]
where $b'=b+|\mathcal{X}|=|\mathcal{X}|(1+(m+1)^{|\mathcal{X}|})$.

Let $Q_{X^{n}}$ be the uniform distribution on $\tilde{A}$. Then,
\eqref{eq:distletter} and \eqref{eq:mi2} still hold, and moreover,
\[
D(Q_{X^{n}}\|P_{X}^{\otimes n})=-\frac{1}{n}\log P_{X}^{\otimes n}(\tilde{A})\le-\frac{1}{n}\log P_{X}^{\otimes n}(A)+o_{n}(1).
\]
If $P_{X}^{\otimes n}(A)\ge e^{-n\alpha}$, we have 
\begin{equation}
D(Q_{X^{n}}\|P_{X}^{\otimes n})\le\alpha+o_{n}(1).\label{eq:-9}
\end{equation}

Denote $t=n\tau$. Let $Q_{Y^{n}|X^{n}}$ be a conditional distribution
such that given each $x^{n}$, $Q_{Y^{n}|X^{n}=x^{n}}$ is concentrated
on the cost ball $B_{t}(x^{n}):=\{y^{n}:c_{n}(x^{n},y^{n})\le t\}$.
Then, we have that $Q_{Y^{n}}:=Q_{X^{n}}\circ Q_{Y^{n}|X^{n}}$ is
concentrated on $A^{t}$, which implies that $-\frac{1}{n}\log P_{Y}^{\otimes n}(A^{t})\le\frac{1}{n}D_{0}(Q_{Y^{n}}\|P_{Y}^{\otimes n})\leq\frac{1}{n}D(Q_{Y^{n}}\|P_{Y}^{\otimes n})$.
Here $D_{0}(Q\|P):=-\log P\{\frac{\mathrm{d}Q}{\mathrm{d}P}>0\}$
is the Rényi divergence of order $0$, which is no greater than the
relative entropy $D(Q\|P)$ \cite{Erven}. Since $Q_{Y^{n}|X^{n}}$
is arbitrary, we have 
\[
-\frac{1}{n}\log P_{Y}^{\otimes n}(A^{t})\le\inf_{Q_{Y^{n}|X^{n}}:c_{n}(X^{n},Y^{n})\le t\textrm{ a.s.}}\frac{1}{n}D(Q_{Y^{n}}\|P_{Y}^{\otimes n}).
\]
Taking supremum of the RHS over all $Q_{X^{n}}$ satisfying \eqref{eq:distletter},
\eqref{eq:mi2}, and \eqref{eq:-9}, we have 
\begin{align}
 & E_{0}^{(n)}(\alpha,\tau)\nonumber \\
 & \le\eta_{n}(\alpha,\tau)\nonumber \\
 & :=\sup_{\substack{Q_{X^{n}},Q_{XW}:\\
\frac{1}{n}D(Q_{X^{n}}\|P_{X}^{\otimes n})\le\alpha+o_{n}(1),\\
Q_{X^{n}W^{n}}\{(x^{n},w^{n}):\L_{x^{n},w^{n}}=Q_{XW}\}=1,\\
I_{Q}(X_{K};X^{K-1},K|W_{K})=o_{m}(1)
}
}\nonumber \\
 & \qquad\qquad\inf_{Q_{Y^{n}|X^{n}}:c_{n}(X^{n},Y^{n})\le t\textrm{ a.s.}}\frac{1}{n}D(Q_{Y^{n}}\|P_{Y}^{\otimes n}),\label{eq:-3-1-2}
\end{align}
where $W_{i}=\phi_{i}(X^{i-1})$. The condition $Q_{X^{n}W^{n}}\{(x^{n},w^{n}):\L_{x^{n},w^{n}}=Q_{XW}\}=1$
implies $Q_{X_{K},W_{K}}=Q_{XW}$.

\textbf{Step 3: Single-letterizing the Cost Constraint}

We next make a special choice of $Q_{Y^{n}|X^{n}}$. Let $\delta>0$
be sufficiently small such that $\psi_{m}(\alpha+\delta,\tau-\delta)<\infty$,
where $\psi_{m}$ is defined similarly as $\psi$ but with $W$ restricted
to concentrate on the alphabet $\mathcal{W}_{m}$ satisfying $|\mathcal{W}_{m}|\le(m+1)^{|\mathcal{X}|}$.
\begin{lem}
\label{lem:bounds} There is some $Q_{Y|XW}$ such that 
\begin{align}
\mathbb{E}_{Q}[c(X,Y)] & \le\tau-\delta,\label{eq:-16-1}\\
D(Q_{Y|W}\|P_{Y}|Q_{W}) & \le\psi_{m}(\alpha+\delta,\tau-\delta)+\delta,\label{eq:-18-2}\\
\mathbb{E}_{Q}\left[c(X,Y)^{2}\right] & \textrm{ is uniformly bounded},\label{eq:-63}
\end{align}
for all $Q_{XW}$ satisfying 
\begin{equation}
D(Q_{X|W}\|P_{X}|Q_{W})\le\alpha+\delta.\label{eq:-62}
\end{equation}
\end{lem}
The proof of this lemma is provided in Section \ref{sec:Proofs-of-Lemma}.

Let $Q_{Y|XW}$ be the conditional distribution given in Lemma \ref{lem:bounds}.
By standard information-theoretic techniques, it holds that 
\begin{align*}
 & \frac{1}{n}D(Q_{X^{n}}\|P_{X}^{\otimes n})\\
 & =\frac{1}{n}\sum_{k=1}^{n}D(Q_{X_{k}|X^{k-1}}\|P_{X}|Q_{X^{k-1}})\\
 & =D(Q_{X_{K}|X^{K-1}K}\|P_{X}|Q_{X^{K-1}K})\\
 & =I_{Q}(X_{K};X^{K-1}K|W)+D(Q_{X|W}\|P_{X}|Q_{W})\\
 & \geq D(Q_{X|W}\|P_{X}|Q_{W}).
\end{align*}
So, the condition $\frac{1}{n}D(Q_{X^{n}}\|P_{X}^{\otimes n})\le\alpha+o_{n}(1)$
in the \eqref{eq:-3-1-2} implies that \eqref{eq:-62} is satisfied
for all sufficiently large $n$. The conclusions in Lemma \ref{lem:bounds}
hold for the $Q_{XW}$ induced by $Q_{X^{n}}$ in the optimization
in \eqref{eq:-3-1-2}. 

However, the product distribution $Q_{Y|XW}^{\otimes n}$ does not
satisfy the constraint $c_{n}(X^{n},Y^{n})\le t\textrm{ a.s.}$ So,
we cannot substitute it into \eqref{eq:-3-1-2} directly. We next
construct a conditional version of $Q_{Y|XW}^{\otimes n}$ and then
substitute this conditional version into \eqref{eq:-3-1-2}.

Denote 
\[
\mu:=\mathbb{E}_{Q}c(X,Y)\le\tau-\delta.
\]
Then, 
for all $(x^{n},w^{n})$ with type $Q_{XW}$ with $w_{i}=\phi_{i}(x^{i-1})$
and for $Y^{n}\sim Q_{Y|X,W}^{\otimes n}(\cdot|x^{n},w^{n})$, it
holds that 
\begin{align*}
\mathbb{E}c_{n}(x^{n},Y^{n}) & =\sum_{k=1}^{n}\mathbb{E}c(x_{k},Y_{k})=\sum_{k=1}^{n}\mu(x_{k},w_{k})\\
 & =n\mathbb{E}_{Q_{XW}}\mu(X,W)=n\mu,
\end{align*}
where $\mu(x,w):=\mathbb{E}_{Q_{Y|(X,W)=(x,w)}}c(x,Y)$. By Chebyshev's
inequality, it holds that 
\begin{align}
\epsilon_{n} & :=\Q\{Y^{n}\notin\{x^{n}\}^{t}\}\nonumber \\
 & =\Q\{c_{n}(x^{n},Y^{n})>n\tau\}\nonumber \\
 & \le\frac{\mathbb{E}_{Q}\left[(c_{n}(x^{n},Y^{n})-n\mu)^{2}\right]}{n^{2}(\tau-\mu)^{2}}\nonumber \\
 & =\frac{\sum_{k=1}^{n}\mathbb{E}_{Q}\left[(c(x_{k},Y_{k})-\mu(x_{k},w_{k}))^{2}\right]}{n^{2}(\tau-\mu)^{2}}\nonumber \\
 & =\frac{\mathbb{E}_{Q_{XW}}\mathrm{Var}(c(X,Y)|X,W)}{n(\tau-\mu)^{2}}\nonumber \\
 & \le\frac{\mathrm{Var}_{Q}(c(X,Y))}{n(\tau-\mu)^{2}}\nonumber \\
 & \le\frac{\mathbb{E}_{Q}\left[c(X,Y)^{2}\right]}{n(\tau-\mu)^{2}}.\label{eq:-61}
\end{align}
Recall that $\Q$ denotes the underlying probability measure that
induces $Q_{Y|X,W}^{\otimes n}$. Combining \eqref{eq:-63} and \eqref{eq:-61}
yields that $\epsilon_{n}$ vanishes as $n\to\infty$ uniformly for
all $Q_{XW}$ induced by $Q_{X^{n}}$ in the optimization in \eqref{eq:-3-1-2}. 

Denote $\hat{Q}_{Y^{n}|X^{n}W^{n}}$ as a distribution given by 
\begin{align*}
 & \hat{Q}_{Y^{n}|(X^{n},W^{n})=(x^{n},w^{n})}\\
 & =\left(\prod_{k=1}^{n}Q_{Y|(X,W)=(x_{k},w_{k})}\right)(\cdot|\{x^{n}\}^{t})
\end{align*}
for all $x^{n}$ and $w_{i}=\phi_{i}(x^{i-1})$. Denote 
\begin{align*}
 & \tilde{Q}_{Y^{n}|(X^{n},W^{n})=(x^{n},w^{n})}\\
 & =\left(\prod_{k=1}^{n}Q_{Y|(X,W)=(x_{k},w_{k})}\right)(\cdot|(\{x^{n}\}^{t})^{c}).
\end{align*}
We can rewrite $Q_{Y|XW}^{\otimes n}$ as a mixture: 
\begin{align*}
Q_{Y|XW}^{\otimes n}(\cdot|x^{n},w^{n}) & =(1-\epsilon_{n})\hat{Q}_{Y^{n}|(X^{n},W^{n})=(x^{n},w^{n})}\\
 & \qquad\qquad+\epsilon_{n}\tilde{Q}_{Y^{n}|(X^{n},W^{n})=(x^{n},w^{n})}.
\end{align*}
For the same input distribution $Q_{X^{n}}$, the output distributions
of channels $Q_{Y|XW}^{\otimes n},\hat{Q}_{Y^{n}|X^{n},W^{n}}$, and
$\tilde{Q}_{Y^{n}|X^{n},W^{n}}$ are respectively denoted as $Q_{Y^{n}},\hat{Q}_{Y^{n}},$
and $\tilde{Q}_{Y^{n}}$, which satisfy 
\[
Q_{Y^{n}}=(1-\epsilon_{n})\hat{Q}_{Y^{n}}+\epsilon_{n}\tilde{Q}_{Y^{n}}.
\]
Denote $J\sim Q_{J}:=\mathrm{Bern}(\epsilon_{n})$, and $Q_{Y^{n}|J=1}=\hat{Q}_{Y^{n}},Q_{Y^{n}|J=0}=\tilde{Q}_{Y^{n}}$.
Then, 
\[
Q_{Y^{n}}=Q_{J}(1)Q_{Y^{n}|J=1}+Q_{J}(0)Q_{Y^{n}|J=0}.
\]

Observe that 
\begin{align*}
 & D(Q_{Y^{n}|J}\|P_{Y}^{\otimes n}|Q_{J})\\
 & =(1-\epsilon_{n})D(\hat{Q}_{Y^{n}}\|P_{Y}^{\otimes n})+\epsilon_{n}D(\tilde{Q}_{Y^{n}}\|P_{Y}^{\otimes n})\\
 & \ge(1-\epsilon_{n})D(\hat{Q}_{Y^{n}}\|P_{Y}^{\otimes n}).
\end{align*}
On the other hand, 
\[
D(Q_{Y^{n}|J}\|P_{Y}^{\otimes n}|Q_{J})=D(Q_{Y^{n}}\|P_{Y}^{\otimes n})+D(Q_{J|Y^{n}}\|Q_{J}|Q_{Y^{n}}),
\]
and 
\[
D(Q_{J|Y^{n}}\|Q_{J}|Q_{Y^{n}})=I_{Q}(J;Y^{n})\le H_{Q}(J)\le\log2.
\]
Hence, 
\begin{equation}
D(\hat{Q}_{Y^{n}}\|P_{Y}^{\otimes n})\le\frac{D(Q_{Y^{n}}\|P_{Y}^{\otimes n})+\log2}{1-\epsilon_{n}}.\label{eq:-65}
\end{equation}

By choosing $Q_{Y^{n}|X^{n}}$ in \eqref{eq:-3-1-2} as a feasible
solution such that $Q_{Y^{n}|X^{n}=x^{n}}=\hat{Q}_{Y^{n}|(X^{n},W^{n})=(x^{n},w^{n})}$
for all $x^{n}$ where $w_{i}=\phi_{i}(x^{i-1})$, we then have that
the objective function $\frac{1}{n}D(\hat{Q}_{Y^{n}}\|P_{Y}^{\otimes n})$
in \eqref{eq:-3-1-2} is upper bounded as shown in \eqref{eq:-65}.
It means that for the fixed distribution $Q_{Y|XW}$ given in Lemma
\ref{lem:bounds}, it holds that 
\begin{align*}
\eta_{n}\left(\alpha,\tau\right) & \le\sup_{\substack{Q_{X^{n}},Q_{XW}:\\
\frac{1}{n}D(Q_{X^{n}}\|P_{X}^{\otimes n})\le\alpha+o_{n}(1),\\
Q_{X_{K},W_{K}}=Q_{XW},\\
I_{Q}(X_{K};X^{K-1},K|W_{K})=o_{m}(1)
}
}\frac{D(Q_{Y^{n}}\|P_{Y}^{\otimes n})}{n(1-\epsilon_{n})}+o_{n}(1).
\end{align*}

\textbf{Step 4: Single-letterizing Divergences}

We next complete the single-letterization. By standard information-theoretic
techniques, we obtain that 
\begin{align}
 & \frac{1}{n}D(Q_{Y^{n}}\|P_{Y}^{\otimes n})\\
 & =\frac{1}{n}\sum_{k=1}^{n}D(Q_{Y_{k}|Y^{k-1}}\|P_{Y}|Q_{Y^{k-1}})\\
 & \le\frac{1}{n}\sum_{k=1}^{n}D(Q_{Y_{k}|X^{k-1}Y^{k-1}}\|P_{Y}|Q_{X^{k-1}Y^{k-1}})\\
 & =\frac{1}{n}\sum_{k=1}^{n}D(Q_{Y_{k}|X^{k-1}}\|P_{Y}|Q_{X^{k-1}})\\
 & \qquad+\frac{1}{n}\sum_{k=1}^{n}D(Q_{Y_{k}|X^{k-1}Y^{k-1}}\|Q_{Y_{k}|X^{k-1}}|Q_{X^{k-1}Y^{k-1}})\\
 & =\frac{1}{n}\sum_{k=1}^{n}D(Q_{Y_{k}|X^{k-1}}\|P_{Y}|Q_{X^{k-1}})\\
 & \qquad+\frac{1}{n}\sum_{k=1}^{n}I_{Q}(Y_{k};Y^{k-1}|X^{k-1})\\
 & =\frac{1}{n}\sum_{k=1}^{n}D(Q_{Y_{k}|X^{k-1}}\|P_{Y}|Q_{X^{k-1}})\label{eq:mi2}\\
 & =D(Q_{Y_{K}|X^{K-1}K}\|P_{Y}|Q_{X^{K-1}K})\\
 & =D(Q_{Y_{K}|X^{K-1}KW_{K}}\|Q_{Y_{K}|W_{K}}|Q_{X^{K-1}KW_{K}})\\
 & \qquad+D(Q_{Y_{K}|W_{K}}\|P_{Y}|Q_{W_{K}})\\
 & =I_{Q}(Y_{K};X^{K-1},K|W_{K})+D(Q_{Y_{K}|W_{K}}\|P_{Y}|Q_{W_{K}})\\
 & =D(Q_{Y_{K}|W_{K}}\|P_{Y}|Q_{W_{K}})+o_{m}(1)\label{eq:-37}\\
 & =D(Q_{Y|W}\|P_{Y}|Q_{W})+o_{m}(1)\label{eq:-66}\\
 & \le\psi_{m}(\alpha+\delta,\tau-\delta)+\delta+o_{m}(1),
\end{align}
where 
\begin{itemize}
\item \eqref{eq:mi2} follows since under the distribution $Q_{X^{n}W^{n}}Q_{Y|XW}^{\otimes n}$,
$W^{k}$ is a function of $X^{k-1}$, and moreover, $Y_{k}$ and $Y^{k-1}$
are conditionally independent given $(X^{k-1},W^{k})$ for each $k$; 
\item \eqref{eq:-37} follows since under the distribution $Q_{K}\otimes Q_{X^{n}W^{n}}Q_{Y|XW}^{\otimes n}$
with $Q_{K}=\Unif[n]$, $(K,X^{K-1})$ and $Y_{K}$ are conditionally
independent given $(X_{K},W_{K})$, and hence, 
\begin{align*}
 & I_{Q}(Y_{K};X^{K-1},K|W_{K})\\
 & \le I_{Q}(X_{K};X^{K-1},K|W_{K})=o_{m}(1);
\end{align*}
\item in \eqref{eq:-66}, $Q_{Y|W}$ is induced by the distribution $Q_{XW}Q_{Y|XW}$,
and \eqref{eq:-66} follows since $Q_{Y_{K}|W}$ is induced by the
distribution $Q_{X_{K}W_{K}}Q_{Y|XW}$, and hence, $Q_{Y_{K}|W}=Q_{Y|W}$
(recall that $Q_{X_{K}W_{K}}=Q_{XW}$);
\item the last line follows by  Lemma \ref{lem:bounds}. 
\end{itemize}
Hence, 
\begin{align*}
\eta_{n}\left(\alpha,\tau\right) & \le\frac{\psi_{m}(\alpha+\delta,\tau-\delta)+\delta+o_{m}(1)}{1-\epsilon_{n}}+o_{n}(1).
\end{align*}

Letting $n\to\infty$ first and $\delta\downarrow0$ then, we obtain
\begin{equation}
\limsup_{n\to\infty}E_{0}^{(n)}(\alpha,\tau)\leq\limsup_{\alpha'\downarrow\alpha,\tau'\uparrow\tau}\psi_{m}(\alpha',\tau').\label{eq:-17}
\end{equation}
Since $\mathcal{P}(\mathcal{X}\times\mathcal{W}_{m})$ is a probability
simplex, by the standard technique of passing a sequence to a convergent
subsequence, one can prove that $\psi_{m}$ is upper semicontinuous,
i.e., $\limsup_{\alpha'\downarrow\alpha,\tau'\uparrow\tau}\psi_{m}(\alpha',\tau')=\psi(\alpha,\tau)$.
 By \eqref{eq:-17} and the upper semicontinuity of $\psi_{m}$,
and letting $m\to\infty$, we obtain 
\begin{equation}
\limsup_{n\to\infty}E_{0}^{(n)}(\alpha,\tau)\leq\psi(\alpha,\tau).\label{eq:-17-1}
\end{equation}

\subsubsection{\label{subsec:Compact}Compact $\mathcal{X}$}

We next generalize the result from finite $\mathcal{X}$ to compact
$\mathcal{X}$ by the standard quantization technique. Since $\mathcal{X}$
is compact, for any $r>0$, it can be covered by a finite number of
open balls $\left\{ B_{r}(x_{i})\right\} _{i=1}^{k}$. Denote $E_{i}:=B_{r}(x_{i})\backslash\bigcup_{j=1}^{i-1}B_{r}(x_{j}),i\in[k]$,
which are measurable. Hence, $\left\{ E_{i}\right\} _{i=1}^{k}$ forms
a partition of $\mathcal{X}$, and $E_{i}$ is a subset of $B_{r}(x_{i})$.
For each $i$, we choose a point $z_{i}\in E_{i}$. Consider $\mathcal{Z}:=\{z_{1},z_{2},\cdots,z_{k}\}$
as a sample space, and define a probability mass function $P_{Z}$
on $\mathcal{Z}$ given by $P_{Z}(z_{i})=P_{X}(E_{i}),\forall i\in[k].$
In other words, $Z\sim P_{Z}$ is a quantized version of $X\sim P_{X}$
in the sense that $Z=z_{i}$ if $X\in E_{i}$ for some $i$.

For a vector $i^{n}:=(i_{1},i_{2},...,i_{n})\in[k]^{n}$, denote $E_{i^{n}}:=\prod_{l=1}^{n}E_{i_{l}}$.
Consequently, $\left\{ E_{i^{n}}:i^{n}\in[k]^{n}\right\} $ forms
a partition of $\mathcal{X}^{n}$. Similarly, for $X^{n}\sim P_{X}^{\otimes n}$,
we denote $Z^{n}$ as a random vector where $Z_{i}$ is the quantized
version of $X_{i},i\in[n]$. Obviously, $Z^{n}\sim P_{Z}^{\otimes n}$.

For any measurable set $A\subseteq\mathcal{X}^{n}$, denote $\mathcal{I}:=\{i^{n}\in[k]^{n}:E_{i^{n}}\cap A\neq\emptyset\}$.
Denote $\hat{A}:=\bigcup_{i^{n}\in\mathcal{I}}E_{i^{n}}$ which is
a superset of $A$, i.e., $A\subseteq\hat{A}$. On the other hand,
for each $i^{n}\in\mathcal{I}$ and any $\hat{\tau}>0$, the $\hat{t}$-enlargement
of $E_{i^{n}}$ with $\hat{t}:=n\hat{\tau}$ satisfies that 
\begin{align}
E_{i^{n}}^{\hat{t}} & =\{y^{n}:c_{n}(x^{n},y^{n})\le\hat{t},\exists x^{n}\in E_{i^{n}}\}\nonumber \\
 & \subseteq\{y^{n}:c_{n}(x^{n},y^{n})\le\hat{t},d(x_{i},\hat{x}_{i})\le r,\forall i\in[n],\\
 & \qquad\qquad\exists\hat{x}^{n}\in A,x^{n}\in\mathcal{X}^{n}\}\label{eq:-13}\\
 & =\{y^{n}:\inf_{x^{n}:d(x_{i},\hat{x}_{i})\le r,\forall i\in[n]}c_{n}(x^{n},y^{n})\le\hat{t},\exists\hat{x}^{n}\in A\}\nonumber \\
 & =\{y^{n}:\sum_{i=1}^{n}\inf_{x_{i}:d(x_{i},\hat{x}_{i})\le r}c(x_{i},y_{i})\le\hat{t},\exists\hat{x}^{n}\in A\}\nonumber \\
 & \subseteq\{y^{n}:\sum_{i=1}^{n}c(\hat{x}_{i},y_{i})\le n(\hat{\tau}+\delta(r)),\exists\hat{x}^{n}\in A\}\label{eq:-10}\\
 & =A^{n(\hat{\tau}+\delta(r))},\nonumber 
\end{align}
where 
\begin{itemize}
\item \eqref{eq:-13} follows from the fact that $\exists x^{n}\in E_{i^{n}}$
implies $d(x_{i},\hat{x}_{i})\le r,\forall i\in[n]$ for some $\hat{x}^{n}\in A,x^{n}\in\mathcal{X}^{n}$; 
\item in \eqref{eq:-10} $\delta(r)$ is a positive function of $r$ which
vanishes as $r\downarrow0$, and \eqref{eq:-10} follows by Assumption
3 (i.e., \eqref{eq:-39}). 
\end{itemize}
Hence, 
\[
\hat{A}^{\hat{t}}=\bigcup_{i^{n}\in\mathcal{I}}E_{i^{n}}^{\hat{t}}\subseteq A^{n(\hat{\tau}+\delta(r))}
\]
If we choose $\hat{\tau}=\tau-\delta(r)$, then $\hat{A}^{n\hat{\tau}}\subseteq A^{n\tau}.$
Combining this with $A\subseteq\hat{A}$ implies 
\begin{align*}
P_{Y}^{\otimes n}(A^{n\tau}) & \ge P_{Y}^{\otimes n}(\hat{A}^{n\hat{\tau}})\\
P_{X}^{\otimes n}(A) & \le P_{X}^{\otimes n}(\hat{A}),
\end{align*}
which further imply that 
\begin{align*}
 & \inf_{A:P_{X}^{\otimes n}(A)\ge a}P_{Y}^{\otimes n}(A^{n\tau})\\
 & \ge\inf_{A:P_{X}^{\otimes n}(\hat{A})\ge a}P_{Y}^{\otimes n}(\hat{A}^{n\hat{\tau}})\\
 & =\inf_{\mathcal{I}\subseteq[k]^{n}:P_{X}^{\otimes n}(\bigcup_{i^{n}\in\mathcal{I}}E_{i^{n}})\ge a}P_{Y}^{\otimes n}((\bigcup_{i^{n}\in\mathcal{I}}E_{i^{n}})^{n\hat{\tau}})\\
 & =\inf_{B\subseteq\mathcal{Z}^{n}:P_{Z}^{\otimes n}(B)\ge a}P_{Y}^{\otimes n}(B^{n\hat{\tau}}),
\end{align*}
where $B^{n\hat{\tau}}=\{y^{n}:c_{n}(z^{n},y^{n})\le n\hat{\tau},\exists z^{n}\in B\}$.
Therefore, 
\[
E_{0}^{(n)}(\alpha,\tau|P_{X})\leq E_{0}^{(n)}(\alpha,\hat{\tau}|P_{Z}),
\]
where $E_{0}^{(n)}(\cdot,\cdot|P_{X})$ is the exponent $E_{0}^{(n)}$
defined for distribution pair $(P_{X},P_{Y})$, and $E_{0}^{(n)}(\cdot,\cdot|P_{Z})$
is the exponent $E_{0}^{(n)}$ defined for $(P_{Z},P_{Y})$.

Denote $\psi(\cdot,\cdot|P_{X})$ as the function $\psi$ defined
for $(P_{X},P_{Y})$, and $\psi(\cdot,\cdot|P_{Z})$ as the one defined
for $(P_{Z},P_{Y})$. Since $\mathcal{Z}$ is a finite metric space
(with the discrete/Hamming metric), by the result proven in Section
\ref{subsec:Finite}, we have 
\[
\limsup_{n\to\infty}E_{0}^{(n)}(\alpha,\hat{\tau}|P_{Z})\leq\psi(\alpha,\hat{\tau}|P_{Z}).
\]
Therefore,

\begin{align}
\limsup_{n\to\infty}E_{0}^{(n)}(\alpha,\tau|P_{X}) & \leq\psi(\alpha,\hat{\tau}|P_{Z})\nonumber \\
 & =\psi(\alpha,\tau-\delta(r)|P_{Z}).\label{eq:-14}
\end{align}

We next show that $\psi(\alpha',\tau'+\delta(r)|P_{Z})\le\psi(\alpha',\tau'|P_{X})$
for any $\alpha'\ge0,\tau'>0$. For any $Q_{Z|W}$, we define a mixture
distribution $Q_{X|W}$ such that for each $w$, 
\[
Q_{X|W=w}=\sum_{i=1}^{k}Q_{Z|W}(z_{i}|w)P_{X}(\cdot|E_{i}),
\]
which implies 
\begin{align}
\frac{\mathrm{d}Q_{X|W}}{\mathrm{d}P_{X}}(x|w) & =\sum_{i=1}^{k}Q_{Z|W}(z_{i}|w)\frac{\bone_{E_{i}}(x)}{P_{X}(E_{i})}\nonumber \\
 & =\sum_{i=1}^{k}Q_{Z|W}(z_{i}|w)\frac{\bone_{E_{i}}(x)}{P_{Z}(z_{i})},\forall x.\label{eq:-11}
\end{align}
For such $Q_{X|W}$, 
\begin{equation}
D(Q_{X|W}\|P_{X}|Q_{W})=D(Q_{Z|W}\|P_{Z}|Q_{W}).\label{eq:-12}
\end{equation}
Note that for such a construction, $Z\sim Q_{Z}$ can be seen as a
quantized version of $X\sim Q_{X}$.

By Assumption 3, we have that $c(X,Y)\ge c(Z,Y)-\delta(r)$ a.s. where
$Z$ is the quantized version of (and also a function of) $X$. We
hence have that for $Q_{X|W}$ constructed above, 
\begin{align*}
 & \C(Q_{X|W},Q_{Y|W}|Q_{W})\\
 & =\min_{Q_{XY|W}\in\cC(Q_{X|W},Q_{Y|W})}\mathbb{E}_{Q_{W}Q_{XY|W}}[c(X,Y)]\\
 & \geq\min_{Q_{XY|W}\in\cC(Q_{X|W},Q_{Y|W})}\mathbb{E}_{Q_{W}Q_{XY|W}}[c(Z,Y)]-\delta(r)\\
 & \geq\min_{Q_{ZY|W}\in\cC(Q_{Z|W},Q_{Y|W})}\mathbb{E}_{Q_{W}Q_{ZY|W}}[c(Z,Y)]-\delta(r)\\
 & =\C(Q_{Z|W},Q_{Y|W}|Q_{W})-\delta(r).
\end{align*}
Therefore, 
\begin{align}
 & \inf_{Q_{Y|W}:\C(Q_{X|W},Q_{Y|W}|Q_{W})\le\tau'}D(Q_{Y|W}\|P_{Y}|Q_{W})\\
 & \geq\inf_{Q_{Y|W}:\C(Q_{Z|W},Q_{Y|W}|Q_{W})\le\tau'+\delta(r)}D(Q_{Y|W}\|P_{Y}|Q_{W}).
\end{align}
Taking supremum over $Q_{ZW}$ such that $D(Q_{Z|W}\|P_{Z}|Q_{W})\le\alpha'$,
we obtain 
\begin{align*}
 & \sup_{Q_{ZW}:D(Q_{Z|W}\|P_{Z}|Q_{W})\le\alpha'}\\
 & \qquad\inf_{Q_{Y|W}:\C(Q_{X|W},Q_{Y|W}|Q_{W})\le\tau'}D(Q_{Y|W}\|P_{Y}|Q_{W})\\
 & \geq\psi(\alpha',\tau'+\delta(r)|P_{Z})
\end{align*}
where $Q_{X|W}$ at the LHS above is induced by $Q_{Z|W}$ as shown
in \eqref{eq:-11}. By \eqref{eq:-12}, the LHS above is in turn upper
bounded by $\psi(\alpha',\tau'|P_{X})$ (by replacing the supremum
above with the supremum over $Q_{XW}$ such that $D(Q_{X|W}\|P_{X}|Q_{W})\le\alpha'$).
Hence, 
\[
\psi(\alpha',\tau'+\delta(r)|P_{Z})\le\psi(\alpha',\tau'|P_{X}).
\]

For $\tau>2\delta(r)$ (when $\tau>0$ and $r$ is sufficiently small),
substituting $\alpha'\leftarrow\alpha,\,\tau'\leftarrow\tau-2\delta(r)$
into the above inequality, we have 
\begin{equation}
\psi(\alpha,\tau-\delta(r)|P_{Z})\le\psi(\alpha,\tau-2\delta(r)|P_{X}).\label{eq:-15}
\end{equation}
Combining \eqref{eq:-14} and \eqref{eq:-15} and letting $r\downarrow0$,
we have 
\[
\limsup_{n\to\infty}E_{0}^{(n)}(\alpha,\tau|P_{X})\leq\lim_{\tau'\uparrow\tau}\psi(\alpha,\tau'|P_{X}).
\]

\subsection{\label{subsec:Statement-2}Statement 2 }

Since $\mathcal{X}$ is Polish, any probability measure on it is tight.
So, for any $\epsilon\in(0,1)$, there is a compact set $B\subseteq\mathcal{X}$
such that $P_{X}(B^{c})\le\epsilon$. Let $X^{n}\sim P_{X}^{\otimes n}$
and $Z_{i}:=\bone_{B^{c}}(X_{i})$, $i\in[n]$. Then, $Z^{n}\sim\Bern(P_{X}(B^{c}))^{\otimes n}$.
By Sanov's theorem, for any $\epsilon'\in(\epsilon,1)$, 
\[
\P\{\sum_{i=1}^{n}Z_{i}\ge n\epsilon'\}\le e^{-nD(\epsilon'\|P_{X}(B^{c}))}\le e^{-nD(\epsilon'\|\epsilon)},
\]
where the second inequality follows since $\epsilon\mapsto D(\epsilon'\|\epsilon)$
is decreasing for $\epsilon<\epsilon'$. Since $\epsilon\mapsto D(\epsilon'\|\epsilon)$
goes to infinity as $\epsilon\downarrow0$, given any $\epsilon'>0$,
we can choose $\epsilon$ small enough so that $D(\epsilon'\|\epsilon)>\alpha$.
For example, we can choose $\epsilon=\epsilon'e^{-1/\epsilon'^{2}}$
and choose $\epsilon'$ small enough. For any measurable set $A$
such that $P_{X}^{\otimes n}(A)\ge e^{-n\alpha}$, it holds that 
\begin{align*}
 & \P\{X^{n}\in A,\,\sum_{i=1}^{n}Z_{i}<n\epsilon'\}\\
 & \ge\P\{X^{n}\in A\}-\P\{\sum_{i=1}^{n}Z_{i}\ge n\epsilon'\}\\
 & \ge e^{-n\alpha}-e^{-nD(\epsilon'\|\epsilon)}.
\end{align*}
Given any $\delta>0$, for all sufficiently large $n$, 
\begin{equation}
\P\{X^{n}\in A,\,\sum_{i=1}^{n}Z_{i}<n\epsilon'\}\ge e^{-n(\alpha+\delta)}.\label{eq:-32}
\end{equation}
For a subset $\mathcal{J}\subseteq[n]$, denote $C_{\mathcal{J}}$
as the event that $X_{i}\in B$ for $i\in\mathcal{J}$ and $X_{i}\in B^{c}$
for $i\in\mathcal{J}^{c}$. Then, \eqref{eq:-32} can be rewritten
as 
\begin{equation}
\P\{X^{n}\in A\cap(\bigcup_{|\mathcal{J}|\ge n(1-\epsilon')}C_{\mathcal{J}})\}\ge e^{-n(\alpha+\delta)}.\label{eq:-32-4-1}
\end{equation}
On the other hand, there are ${n \choose \le n\epsilon'}:=\sum_{i=1}^{\left\lfloor n\epsilon'\right\rfloor }{n \choose i}$
of sets $\mathcal{J}\subseteq[n]$ such that $|\mathcal{J}|\ge n(1-\epsilon')$.
Note that by Sanov's theorem, ${n \choose \le n\epsilon'}\le e^{nH(\epsilon')}$,
where $H(\epsilon')$ is the binary entropy function of $\epsilon'$.
Combining this with \eqref{eq:-32-4-1} yields that 
\begin{equation}
\max_{|\mathcal{J}|\ge n(1-\epsilon')}\P\{X^{n}\in A\cap C_{\mathcal{J}}\}\ge e^{-n(\alpha+\delta+H(\epsilon'))}.\label{eq:-32-4-1-1}
\end{equation}
Let $\mathcal{J}^{*}$ be the optimal $\mathcal{J}$ attaining the
maximum in the above equation. Without loss of generality, we assume
$\mathcal{J}^{*}=[n^{*}]$ for some $n^{*}\ge n':=\left\lceil n(1-\epsilon')\right\rceil $.
Denote $A':=\bigcup_{x^{n}\in A}\{x^{n'}\}\subseteq\mathcal{X}^{n'}$
as the projection of $A$ to the first $n'$ coordinates. Then, the
maximum in \eqref{eq:-32-4-1-1} is upper bounded by $\P\{X^{n'}\in A'\cap B^{n'}\}=P_{X}^{\otimes n'}(A'\cap B^{n'})$.
Denote $c_{\sup}:=\sup_{x,y}c(x,y)$, which by assumption is finite.
Moreover, 
\begin{align}
A^{t} & \supseteq(A\cap C_{\mathcal{J}^{*}})^{t}\nonumber \\
 & =\bigcup_{x^{n}\in A\cap C_{\mathcal{J}^{*}}}\{x^{n}\}^{t}\nonumber \\
 & \supseteq\bigcup_{x^{n}\in A\cap C_{\mathcal{J}^{*}}}\Big(\{x^{n'}\}^{t-(n-n')c_{\sup}}\times\prod_{i=n'+1}^{n}\{x_{i}\}^{c_{\sup}}\Big)\label{eq:-33}\\
 & \supseteq\bigcup_{x^{n}\in A\cap C_{\mathcal{J}^{*}}}\Big(\{x^{n'}\}^{t-(n-n')c_{\sup}}\times\mathcal{Y}^{n-n'}\Big)\label{eq:-34}\\
 & =\bigcup_{x^{n'}\in A'\cap B^{n'}}\Big(\{x^{n'}\}^{t-(n-n')c_{\sup}}\times\mathcal{Y}^{n-n'}\Big)\nonumber \\
 & =(A'\cap B^{n'})^{t-(n-n')c_{\sup}}\times\mathcal{Y}^{n-n'},
\end{align}
where  
\begin{itemize}
\item \eqref{eq:-33} follows since in the enlargement operation, $\sum_{i=1}^{n}c(x_{i},y_{i})\le t$
is relaxed to $\sum_{i=1}^{n'}c(x_{i},y_{i})\le t-(n-n')c_{\sup}$
and $c(x_{i},y_{i})\le c_{\sup}$ for $i\in[n'+1:n]$; 
\item \eqref{eq:-34} follows since $\{x\}^{c_{\sup}}=\mathcal{Y}$ for
any $x$.  
\end{itemize}
Denoting $\tilde{A}:=A'\cap B^{n'}$ and summarizing the above, it
holds that 
\begin{align*}
P_{X}^{\otimes n'}(\tilde{A}) & \ge e^{-n(\alpha+\delta+H(\epsilon'))},\\
P_{Y}^{\otimes n}(A^{t}) & \ge P_{Y}^{\otimes n'}(\tilde{A}^{t-(n-n')c_{\sup}})\ge P_{Y}^{\otimes n'}(\tilde{A}^{t-n\epsilon'c_{\sup}}).
\end{align*}
Setting $t=n\tau$, we then have that 
\begin{align*}
P_{Y}^{\otimes n}(A^{n\tau}) & \ge\inf_{\tilde{A}\subseteq B^{n'}:P_{X}^{\otimes n'}(\tilde{A})\ge e^{-n'\alpha'}}P_{Y}^{\otimes n'}(\tilde{A}^{n'\tau'}),
\end{align*}
where $\tau':=\tau-\epsilon'c_{\sup}$ and $\alpha':=\frac{\alpha+\delta+H(\epsilon')}{1-\epsilon'}$.
That is, 
\begin{align*}
E_{0}^{(n)}(\alpha,\tau|P_{X}) & \le E_{0}^{(n')}(\alpha'+\log P_{X}(B),\tau'|P_{X}(\cdot|B)).
\end{align*}
Since $B$ is compact, applying the upper bound on the isoperimetric
exponent for compact $\mathcal{X}$ (proven in Section \ref{subsec:Compact}),
we obtain that 
\begin{align*}
 & \limsup_{n\to\infty}E_{0}^{(n)}(\alpha,\tau|P_{X})\\
 & \le\limsup_{n\to\infty}E_{0}^{(n')}(\alpha'+\log P_{X}(B),\tau'|P_{X}(\cdot|B))\\
 & \leq\lim_{\tau''\uparrow\tau'}\psi(\alpha'+\log P_{X}(B),\tau''|P_{X}(\cdot|B)).
\end{align*}

Observe that 
\begin{align*}
 & \psi(\alpha'+\log P_{X}(B),\tau''|P_{X}(\cdot|B))\\
 & =\sup_{Q_{XW}:D(Q_{X|W}\|P_{X}(\cdot|B)|Q_{W})\le\alpha'+\log P_{X}(B)}\\
 & \qquad\qquad\inf_{Q_{Y|XW}:\mathbb{E}[c(X,Y)]\le\tau''}D(Q_{Y|W}\|P_{Y}|Q_{W})\\
 & \le\sup_{Q_{XW}:D(Q_{X|W}\|P_{X}|Q_{W})\le\alpha'}\\
 & \qquad\inf_{Q_{Y|XW}:\mathbb{E}[c(X,Y)]\le\tau''}D(Q_{Y|W}\|P_{Y}|Q_{W})\\
 & \le\psi(\alpha',\tau''|P_{X}),
\end{align*}
Therefore, 
\begin{align*}
\limsup_{n\to\infty}E_{0}^{(n)}(\alpha,\tau|P_{X}) & \le\lim_{\tau''\uparrow\tau'}\psi(\alpha',\tau''|P_{X}).
\end{align*}
Letting $\epsilon'\downarrow0$ first and $\delta\downarrow0$ then,
we obtain that 
\begin{align*}
\limsup_{n\to\infty}E_{0}^{(n)}(\alpha,\tau|P_{X}) & \le\lim_{\alpha''\downarrow\alpha}\lim_{\tau''\uparrow\tau}\psi(\alpha'',\tau''|P_{X}).
\end{align*}

\subsection{Statement 3 }

The proof of the lower bound is based on the large deviations theory,
which is similar to that of Statement 2 of Theorem \ref{thm:LD} given
in Section \ref{sec:Proof-of-Theorem-1}.

Let $\epsilon>0$ and $m\ge2$. Let $Q_{WX}$ be such that $|\supp(Q_{W})|\le m$
and $D(Q_{X|W}\|P_{X}|Q_{W})\le\alpha-\epsilon$. Without loss of
generality, we assume $\supp(Q_{W})=[m]$, under which the function
$\psi$ does not change by Theorem \ref{thm:The-alphabet-size}. For
each $n$, let $Q_{W}^{(n)}$ be an empirical measure of an $n$-length
sequence (i.e., $n$-type) such that $\supp(Q_{W}^{(n)})\subseteq[m]$
and $Q_{W}^{(n)}\to Q_{W}$ as $n\to\infty$. Let $Q_{XW}^{(n)}=Q_{W}^{(n)}Q_{X|W}$.
Let $w^{n}=(1,\cdots,1,2,\cdots,2,\cdots,m,\cdots,m)$ be an $n$-length
sequence, where $i$ appears $n_{i}:=nQ_{W}^{(n)}(i)$ times. Hence,
the empirical measure of $w^{n}$ is $Q_{W}^{(n)}$.

Let $\epsilon'>0$. We now choose $A$ as the conditional empirically
$\epsilon'$-typical sets. That is, $A=\L_{n}^{-1}(\mathcal{A}|w^{n})=\prod_{w=1}^{m}\L_{n_{w}}^{-1}(\mathcal{A}_{w})$,
where $\mathcal{A}_{w}:=B_{\epsilon']}(Q_{X|W=w})$ for $w\in[m]$,
and $\mathcal{A}:=\{R_{X|W}:R_{X|W=w}\in\mathcal{A}_{w},\forall w\in[m]\}$.
As shown in Section \ref{sec:Proof-of-Theorem-1}, $A$ is closed
in $\mathcal{X}^{n}$, and $-\frac{1}{n}\log P_{X}^{\otimes n}(A)\le\alpha$
for all sufficiently large $n$.

Denote $t=n\tau$. Observe that 
\begin{align*}
A^{t} & =\{y^{n}:\exists x^{n},\,\L_{x^{n}|w^{n}}\in\mathcal{A},\,c_{n}(x^{n},y^{n})\le t\}\\
 & =\{y^{n}:\exists x^{n},\,\L_{x^{n}|w^{n}}\in\mathcal{A},\,\mathbb{E}_{\L_{x^{n},y^{n},w^{n}}}c(X,Y)\le\tau\}\\
 & \subseteq\{y^{n}:\exists x^{n},\,\L_{x^{n}|w^{n}}\in\mathcal{A},\,\C(\L_{x^{n}|w^{n}},\L_{y^{n}|w^{n}}|\L_{w^{n}})\le\tau\}\\
 & \subseteq\{y^{n}:\exists R_{X|W}\in\mathcal{A},\,\C(R_{X|W},\L_{y^{n}|w^{n}}|Q_{W}^{(n)})\le\tau\}.
\end{align*}
Hence, we have $A^{t}\subseteq\L_{n}^{-1}(\mathcal{B}|w^{n}),$ where
\begin{equation}
\mathcal{B}=\{R_{Y|W}:\C(R_{X|W},R_{Y|W}|Q_{W}^{(n)})\le\tau,\exists R_{X|W}\in\mathcal{A}\}.\label{eq:-33-2}
\end{equation}

By a conditional version of Sanov's theorem, 
\begin{align}
E & :=\liminf_{n\to\infty}-\frac{1}{n}\log P_{Y}^{\otimes n}(\L_{n}^{-1}(\mathcal{B}|w^{n}))\nonumber \\
 & \ge\inf_{R_{WY}\in\overline{\mathcal{B}'}}D(R_{YW}\|P_{Y}\otimes Q_{W}),\label{eq:-18-1}
\end{align}
where $\mathcal{B}':=\{R_{WY}:R_{W}\in B_{\epsilon']}(Q_{W}),R_{Y|W}\in\mathcal{B}\}$.
To simplify this lower bound, denoting 
\[
\hat{\mathcal{B}}:=\{R_{Y|W}:\C(Q_{X|W},R_{Y|W}|Q_{W})\le\tau+2\epsilon''\},
\]
we have the following lemma. 
\begin{lem}
\label{lem:For-sufficiently-small}For sufficiently small $\epsilon'$,
it holds that 
\begin{equation}
\mathcal{B}'\subseteq\hat{\mathcal{B}}':=\{R_{WY}:R_{W}\in B_{\epsilon']}(Q_{W}),R_{Y|W}\in\hat{\mathcal{B}}\},\label{eq:-36-3}
\end{equation}
and $\hat{\mathcal{B}}'$ is closed (in the weak topology). 
\end{lem}
\begin{IEEEproof}[Proof of Lemma \ref{lem:For-sufficiently-small}]
By Assumption 1, for any $R_{Y|W}$, it holds that given $\epsilon''>0$,
for sufficiently small $\epsilon'$, 
\[
\inf_{R_{X|W}\in A}\C(R_{X|W},R_{Y|W}|Q_{W}^{(n)})\geq\C(Q_{X|W},R_{Y|W}|Q_{W}^{(n)})-\epsilon''.
\]
Note that the minimization in the conditional optimal transport can
be taken in a pointwise way for each condition $W=w$. Combining this
with the condition that $c$ is bounded, we have that $R_{W}\mapsto\C(R_{X|W},R_{Y|W}|R_{W})$
is continuous. So, given $\epsilon''>0$, for sufficiently large $n$,
\[
\C(Q_{X|W},R_{Y|W}|Q_{W}^{(n)})\ge\C(Q_{X|W},R_{Y|W}|Q_{W})-\epsilon''.
\]
This implies that given $\epsilon''$, for sufficiently small $\epsilon'$,
$\mathcal{B}\subseteq\hat{\mathcal{B}}$. Hence, $\mathcal{B}'\subseteq\hat{\mathcal{B}}'$.

We next prove that for sufficiently small $\epsilon$, $\hat{\mathcal{B}}'$
is closed. Let $(R_{WY}^{(k)})$ be an arbitrary sequence drawn from
$\hat{\mathcal{B}}'$, which converges to $R_{WY}^{*}$ (under the
weak topology). Obviously, $R_{W}^{(k)}\to R_{W}^{*}=Q_{W}$ and $R_{Y|W=w}^{(k)}\to R_{Y|W=w}^{*}$
for each $w$. By the lower semi-continuity of $R_{Y}\mapsto\C(R_{X},R_{Y})$,
we have that 
\begin{align*}
\liminf_{k\to\infty}\C(Q_{X|W=w},R_{Y|W=w}^{(k)}) & \ge\C(Q_{X|W=w},R_{Y|W=w}^{*}).
\end{align*}
Hence, 
\begin{align*}
\liminf_{k\to\infty}\C(Q_{X|W},R_{Y|W}^{(k)}|Q_{W}) & \ge\C(Q_{X|W},R_{Y|W}^{*}|Q_{W}).
\end{align*}
On the other hand, by the choice of $(R_{WY}^{(k)})$, $\C(Q_{X|W},R_{Y|W}^{(k)}|Q_{W})\le\tau+2\epsilon''$.
Hence, $\C(Q_{X|W},R_{Y|W}^{*}|Q_{W})\le\tau+2\epsilon''$. That is,
$R_{WY}^{*}\in\hat{\mathcal{B}}'$. Hence, $\hat{\mathcal{B}}'$ is
closed. This completes the proof of Lemma \ref{lem:For-sufficiently-small}. 
\end{IEEEproof}
By Lemma \ref{lem:For-sufficiently-small} and \eqref{eq:-18-1},
\begin{align*}
E & \ge\inf_{R_{WY}\in\hat{\mathcal{B}}'}D(R_{YW}\|P_{Y}\otimes Q_{W})\\
 & =\inf_{R_{WY}:R_{W}\in B_{\epsilon']}(Q_{W}),R_{Y|W}\in\hat{\mathcal{B}}}D(R_{Y|W}\|P_{Y}|R_{W})+D(R_{W}\|Q_{W}).
\end{align*}
Letting $\epsilon'\downarrow0$ and by the continuity of $R_{W}\in\mathcal{P}([m])\mapsto D(R_{W}\|Q_{W})$,
the second term above can be removed: 
\begin{align*}
E & \ge\beta:=\lim_{\epsilon'\downarrow0}\inf_{\substack{R_{W}\in B_{\epsilon']}(Q_{W}),R_{Y|W}:\\
\C(Q_{X|W},R_{Y|W}|Q_{W})\le\tau+2\epsilon''
}
}D(R_{Y|W}\|P_{Y}|Q_{W}).
\end{align*}

Let $(R_{W}^{(k)},R_{Y|W}^{(k)})$ be such that 
\begin{align*}
 & R_{W}^{(k)}\in B_{\frac{1}{k}]}(Q_{W}),\\
 & \C(Q_{X|W},R_{Y|W}^{(k)}|Q_{W})\le\tau+2\epsilon'',\\
 & D(R_{Y|W}^{(k)}\|P_{Y}|Q_{W})\le\beta+\frac{1}{k}.
\end{align*}
Since $R_{W}^{(k)}$ is in the probability simplex, by passing to
a subsequence, we assume $R_{W}^{(k)}\to Q_{W}$. Since sublevel sets
of the relative entropy $R_{Y}\mapsto D(R_{Y}\|P_{Y})$ are compact,
by the fact that for each $w$, $D(R_{Y|W=w}\|P_{Y})$ is finite,
passing to a subsequence, we have $R_{Y|W=w}^{(k)}\to R_{Y|W=w}^{*}$.
By the lower semi-continuity of the relative entropy and the optimal
transport cost functional, we have 
\begin{align*}
\liminf_{k\to\infty}D(R_{Y|W}^{(k)}\|P_{Y}|Q_{W}) & \ge D(R_{Y|W}^{*}\|P_{Y}|Q_{W}),\\
\liminf_{k\to\infty}\C(Q_{X|W},R_{Y|W}^{(k)}|Q_{W}) & \ge\C(Q_{X|W},R_{Y|W}^{*}|Q_{W}).
\end{align*}
Hence, $R_{Y|W}^{*}$ satisfies that 
\begin{align*}
\C(Q_{X|W},R_{Y|W}^{*}|Q_{W}) & \le\tau+2\epsilon''\\
D(R_{Y|W}^{*}\|P_{Y}|Q_{W}) & \le\beta.
\end{align*}
Therefore, $E\ge g(\tau+2\epsilon'',Q_{XW})$, where 
\begin{align*}
g(t,Q_{XW}) & :=\inf_{Q_{Y|W}:\C(Q_{X|W},Q_{Y|W}|Q_{W})\le t}D(Q_{Y|W}\|P_{Y}|Q_{W})\\
 & =\inf_{Q_{Y|XW}:\mathbb{E}[c(X,Y)]\le t}D(Q_{Y|W}\|P_{Y}|Q_{W}).
\end{align*}

Since $Q_{XW}$ is arbitrary distribution on $\mathcal{X}\times\mathcal{W}$
satisfying $D(Q_{X|W}\|P_{X}|Q_{W})\le\alpha-\epsilon$, taking supremum
over all such distributions, we obtain 
\begin{align*}
 & \liminf_{n\to\infty}E_{0}^{(n)}(\alpha,\tau)\\
 & \ge\sup_{Q_{XW}:D(Q_{X|W}\|P_{X}|Q_{W})\le\alpha-\epsilon}g(\tau+2\epsilon'',Q_{XW})\\
 & =\psi(\alpha-\epsilon,\tau+2\epsilon'').
\end{align*}
Letting $\epsilon\downarrow0$ and $\epsilon''\downarrow0$, we obtain
\begin{align*}
\liminf_{n\to\infty}E_{0}^{(n)}(\alpha,\tau) & \ge\lim_{\alpha'\uparrow\alpha}\lim_{\tau'\downarrow\tau}\psi(\alpha',\tau')\\
 & =\psi(\alpha,\tau),
\end{align*}
where the last line will be proven in Corollary \ref{cor:uppersemicontinuity}.

\section{Proofs of Dual Formulas }

\label{sec:dual}

It is well known that the OT cost admits the following duality. 
\begin{lem}[Kantorovich Duality]
\cite[Theorem 5.10]{villani2008optimal} \label{lem:Kantorovich}
Let $\mathcal{X}$ and $\mathcal{Y}$ be Polish spaces. It holds that
\begin{align}
\C(Q_{X},Q_{Y}) & =\sup_{\substack{(f,g)\in C_{\mathrm{b}}(\mathcal{X})\times C_{\mathrm{b}}(\mathcal{Y}):\\
f+g\leq c
}
}Q_{X}(f)+Q_{Y}(g),\label{eq:dual}
\end{align}
where $C_{\mathrm{b}}(\mathcal{X})$ denotes the collection of bounded
continuous functions $f:\mathcal{X}\to\mathbb{R}$. 
\end{lem}
We also need the following duality for the I-projection, which is
well-known if the space is Polish since both sides in \eqref{eq:-47}
correspond to the same large deviation exponent. 
\begin{lem}[Duality for the I-Projection]
\label{lem:DE} Let $f:\mathcal{X}\to\mathbb{R}$ be a measurable
bounded above function. Then, it holds that for any real $\tau$,
\begin{equation}
\inf_{Q:Q(f)\ge\tau}D(Q\|P)=\sup_{\lambda\ge0}\lambda\tau-\log P(e^{\lambda f}),\label{eq:-47}
\end{equation}
and for any real $\alpha\ge0$, 
\begin{equation}
\sup_{Q:D(Q\|P)\le\alpha}Q(f)=\inf_{\eta>0}\eta\alpha+\eta\log P(e^{(1/\eta)f}).\label{eq:-53}
\end{equation}
The $\sup_{\lambda\ge0}$ in \eqref{eq:-47} can be replaced by $\sup_{\lambda>0}$. 
\end{lem}
This lemma is a direct consequence of the following lemma. The following
lemma can be easily verified by definition. 
\begin{lem}
\cite{csiszar1975divergence} \label{lem:identity} For a measurable
bounded above function $f:\mathcal{X}\to\mathbb{R}$ and $\lambda\ge0$,
define a probability measure $Q_{\lambda}$ with density 
\[
\frac{\mathrm{d}Q_{\lambda}}{\mathrm{d}P}=\frac{e^{\lambda f}}{P(e^{\lambda f})},
\]
then 
\begin{align}
 & D(Q\|P)-D(Q_{\lambda}\|P)\nonumber \\
 & =D(Q\|Q_{\lambda})+\lambda\left(Q(f)-Q_{\lambda}(f)\right)\\
 & \geq\lambda\left(Q(f)-Q_{\lambda}(f)\right).\label{eq:-55}
\end{align}
\end{lem}
The function $f$ in Lemmas \ref{lem:DE} and \ref{lem:identity}
can be assumed to be unbounded, but $P(e^{\lambda f})$ should be
finite for Lemma \ref{lem:identity}, $P(e^{\lambda f})$ should be
finite for $\lambda\ge0$ such that $Q_{\lambda}(f)=\tau$ for \eqref{eq:-47},
and $P(e^{(1/\eta)f})$ should be finite for $\eta>0$ such that $D(Q_{1/\eta}\|P)=\alpha$
for \eqref{eq:-53},

The conditional version of Lemma \ref{lem:DE} is as follows, which
can be proven similarly to the unconditional version. 
\begin{lem}
\label{lem:conditionalDE} Let $\mathcal{W}$ be a finite set and
$f:\mathcal{X}\times\mathcal{W}\to\mathbb{R}$ be a measurable bounded
above function. Let $P_{W}$ be a probability measure on $\mathcal{W}$.
Then, for any real $\tau$, it holds that 
\begin{align*}
 & \inf_{Q_{X|W}:P_{W}Q_{X|W}(f)\ge\tau}D(Q_{X|W}\|P_{X|W}|P_{W})\\
 & =\sup_{\lambda\ge0}\lambda\tau-P_{W}(\log P_{X|W}(e^{\lambda f})),
\end{align*}
and for any real $\alpha\ge0$, it holds that 
\begin{align*}
 & \sup_{Q_{X|W}:D(Q_{X|W}\|P_{X|W}|Q_{W})\le\alpha}P_{W}Q_{X|W}(f)\\
 & =\inf_{\eta>0}\eta\alpha+\eta P_{W}(\log P_{X|W}(e^{(1/\eta)f})).
\end{align*}
\end{lem}
Based on the duality lemmas above, we prove Theorem \ref{thm:phi},
Proposition \ref{prop:varphi}, Theorem \ref{thm:equivalence}, and
Theorem \ref{thm:psi}.

\begin{proof}[Proof of Theorem \ref{thm:phi}] By the definition
of $\phi_{\ge}$ and by the Kantorovich duality, 
\begin{align}
 & \phi_{\ge}(\alpha,\tau)\nonumber \\
 & =\inf_{\substack{Q_{X},Q_{Y},f,g:f+g\le c,\\
Q_{X}(f)+Q_{Y}(g)\ge\tau,\\
D(Q_{X}\|P_{X})\le\alpha
}
}D(Q_{Y}\|P_{Y})\\
 & =\inf_{\substack{Q_{X},f,g:f+g\le c,\\
D(Q_{X}\|P_{X})\le\alpha
}
}\inf_{Q_{Y}:Q_{X}(f)+Q_{Y}(g)\ge\tau}D(Q_{Y}\|P_{Y}).\label{eq:-27}
\end{align}

By Lemma \ref{lem:DE}, 
\begin{align}
\phi_{\ge}(\alpha,\tau) & =\inf_{f,g:f+g\le c,}\inf_{Q_{X}:D(Q_{X}\|P_{X})\le\alpha}\sup_{\lambda>0}\nonumber \\
 & \qquad\qquad\lambda(\tau-Q_{X}(f))-\log P_{Y}(e^{\lambda g}).\label{eq:-44}
\end{align}
The objective function in \eqref{eq:-44} is linear in $\lambda$
and also linear in $Q_{X}$, and moreover, $\{Q_{X}:D(Q_{X}\|P_{X})\le\alpha\}$
is compact. So, by the minimax theorem \cite[Theorem  2.10.2]{zalinescu2002convex},
the second infimization and the supremization can be swapped. Hence,
the inf-sup part in \eqref{eq:-27} is equal to 
\begin{align*}
 & \sup_{\lambda>0}\inf_{Q_{X}:D(Q_{X}\|P_{X})\le\alpha}\lambda(\tau-Q_{X}(f))-\log P_{Y}(e^{\lambda g}).
\end{align*}
which by Lemma \ref{lem:DE}, can be rewritten as 
\begin{align*}
 & \sup_{\lambda>0}\lambda(\tau-\inf_{\eta>0}(\eta\alpha+\eta\log P_{X}(e^{(1/\eta)f})))-\log P_{Y}(e^{\lambda g}).
\end{align*}
Substituting this into \eqref{eq:-44} completes the proof. \end{proof}
\begin{proof}[Proof of Proposition \ref{prop:varphi}] 
 By the Kantorovich--Rubinstein formula \cite[(5.11)]{villani2008optimal},
\begin{align}
 & \varphi_{X,\ge}(\tau)\nonumber \\
 & =\inf_{Q_{X},\textrm{1-Lip }f:P_{X}(f)=0,\,Q_{X}(f)\ge\tau}D(Q_{X}\|P_{X})\nonumber \\
 & =\inf_{\textrm{1-Lip }f:P_{X}(f)=0}\inf_{Q_{X}}\sup_{\lambda\ge0}D(Q_{X}\|P_{X})+\lambda(\tau-Q_{X}(f))\nonumber \\
 & =\inf_{\textrm{1-Lip }f:P_{X}(f)=0}\sup_{\lambda\ge0}\inf_{Q_{X}}D(Q_{X}\|P_{X})+\lambda(\tau-Q_{X}(f))\nonumber \\
 & =\inf_{\textrm{1-Lip }f:P_{X}(f)=0}\sup_{\lambda\ge0}\lambda\tau-\log P_{X}(e^{\lambda f}).\label{eq:-6-1}
\end{align}
\end{proof} \begin{proof}[Proof of Theorem \ref{thm:equivalence}]
It is easy to see that $\conv{\varphi}_{X,\ge}(\tau)=\conv{\varphi}_{X}(\tau)$.
If we swap the inf and sup in \eqref{eq:-46}, then we will obtain
$r(\tau)$. However, this is infeasible in general.

Obviously, from \eqref{eq:-46}, $\varphi_{X,\ge}(\tau)\ge r(\tau)$,
and by definition, $r(\tau)$ is convex. So, taking the lower convex
envelope, we obtain $\conv{\varphi}_{X,\ge}(\tau)\ge r(\tau)$. It
remains to prove $\conv{\varphi}_{X,\ge}(\tau)\le r(\tau)$. We next
do this.

By \cite[Theorem 3.10]{alon1998asymptotic}, given any $\tau\ge0$,
there is a $\lambda^{*}$ such that $r(\tau)=\lambda^{*}\tau-L_{G}(\lambda^{*})$.
Because the function $\lambda\mapsto\lambda\tau-L_{G}(\lambda)$ has
a maximum at $\lambda^{*}$, its right derivative at $\lambda^{*}$
is at most $0$, and its left derivative is at least $0$. In other
words, we have $L_{G}^{\mathrm{l}}(\lambda^{*})\le\tau\le L_{G}^{\mathrm{r}}(\lambda^{*})$.
Because $L_{G}^{\mathrm{r}}(\lambda^{*})\ge\tau$, there must be a
function $g:\mathcal{X}\to\mathbb{R}$ such that $L_{g}(\lambda^{*})=L_{G}(\lambda^{*})$
and $L_{g}'(\lambda^{*})\ge\tau$. Because $L_{G}^{\mathrm{l}}(\lambda^{*})\le\tau$,
there must be a function $h:\mathcal{X}\to\mathbb{R}$ such that $L_{h}(\lambda^{*})=L_{G}(\lambda^{*})$
and $L_{h}'(\lambda^{*})\le\tau$. Hence for any $\epsilon>0$, there
are positive integer $n$ and nonnegative integer $k$ such that $\left|\hat{\tau}-\tau\right|\le\epsilon,$
where 
\[
\hat{\tau}:=pL_{g}'(\lambda^{*})+(1-p)L_{h}'(\lambda^{*})
\]
and $p=\frac{k}{n}$.

Let $X^{n}\sim P_{X}^{\otimes n}$. Denote $f:\mathcal{X}^{n}\to\mathbb{R}$
by 
\[
f(x^{n})=\sum_{i=1}^{k}g(x_{i})+\sum_{i=k+1}^{n}h(x_{i}).
\]
Since $g,h$ are $1$-Lipschitz, so is $f$ (on the product space).
Then, for any $\lambda\ge0$, 
\[
L_{f}(\lambda)=kL_{g}(\lambda)+(n-k)L_{h}(\lambda).
\]
Then, 
\begin{align}
r(\tau) & =\lambda^{*}\tau-L_{G}(\lambda^{*})\nonumber \\
 & \le\lambda^{*}\hat{\tau}-\left(pL_{g}(\lambda^{*})+(1-p)L_{h}(\lambda^{*})\right)+\lambda^{*}\epsilon\nonumber \\
 & =\sup_{\lambda\ge0}\lambda\hat{\tau}-\left(pL_{g}(\lambda)+(1-p)L_{h}(\lambda)\right)+\lambda^{*}\epsilon\label{eq:-20}\\
 & =\sup_{\lambda\ge0}\lambda\hat{\tau}-\frac{1}{n}L_{f}(\lambda)+\lambda^{*}\epsilon\label{eq:-23}\\
 & \ge\inf_{\textrm{1-Lip }\hat{f}:P_{X}^{\otimes n}(\hat{f})=0}\sup_{\lambda\ge0}\lambda\hat{\tau}-\frac{1}{n}L_{\hat{f}}(\lambda)+\lambda^{*}\epsilon\nonumber \\
 & =\frac{1}{n}\varphi_{n}(n\hat{\tau})+\lambda^{*}\epsilon\label{eq:-24}\\
 & \ge\conv{\varphi}_{X,\ge}(\hat{\tau})+\lambda^{*}\epsilon,\label{eq:-25}
\end{align}
where 
\begin{itemize}
\item \eqref{eq:-20} follows since the objective function in it is strictly
convex in $\lambda$ and its derivative is zero at $\lambda^{*}$; 
\item $\varphi_{n}$ in \eqref{eq:-24} given by 
\[
\varphi_{n}(t)=\inf_{Q_{X^{n}}\in\mathcal{P}(\mathcal{X}^{n}):\C(P_{X}^{\otimes n},Q_{X^{n}})\ge t}D(Q_{X^{n}}\|P_{X}^{\otimes n})
\]
is the $n$-dimensional extension of $\varphi_{X,\ge}$, and \eqref{eq:-24}
follows by Proposition \ref{prop:varphi} for the $n$-dimensional
version $\varphi_{n}$; 
\item \eqref{eq:-25} follows the single-letterization argument same to
that used for \eqref{eq:-26}. 
\end{itemize}
Lastly, letting $\epsilon\to0$, we have $\hat{\tau}\to\tau$. By
the continuity of $\conv{\varphi}_{X,\ge}$ and \eqref{eq:-25}, we
have $r(\tau)\ge\conv{\varphi}_{X,\ge}(\tau)$.

\end{proof} \begin{proof}[Proof of Theorem \ref{thm:psi}] We
first give a dual formula for 
\[
\theta(\tau,Q_{XW}):=\inf_{Q_{Y|W}:\C(Q_{X|W},Q_{Y|W}|Q_{W})\le\tau}D(Q_{Y|W}\|P_{Y}|Q_{W}).
\]

Observe that 
\begin{align}
 & \theta(\tau,Q_{XW})\nonumber \\
 & =\inf_{Q_{Y|W}:\C(Q_{X|W},Q_{Y|W}|Q_{W})\le\tau}D(Q_{Y|W}\|P_{Y}|Q_{W})\\
 & =\inf_{Q_{Y|W}}\sup_{\lambda\ge0}D(Q_{Y|W}\|P_{Y}|Q_{W})\nonumber \\
 & \qquad+\lambda(\C(Q_{X|W},Q_{Y|W}|Q_{W})-\tau)\\
 & =\sup_{\lambda\ge0}\inf_{Q_{Y|W}}D(Q_{Y|W}\|P_{Y}|Q_{W})\nonumber \\
 & \qquad+\lambda(\mathbb{E}_{Q_{W}}[\C(Q_{X|W}(\cdot|W),Q_{Y|W}(\cdot|W))]-\tau)\label{eq:-42}\\
 & =\sup_{\lambda\ge0}\inf_{Q_{Y|W}}\mathbb{E}_{Q_{W}}[D(Q_{Y|W}(\cdot|W)\|P_{Y})\nonumber \\
 & \qquad+\lambda(\sup_{f+g\le c}Q_{X|W}(f|W)+Q_{Y|W}(g|W)-\tau)]\label{eq:-51}\\
 & =\sup_{\lambda\ge0}\sum_{w}Q_{W}(w)[\inf_{Q_{Y|W=w}}\sup_{f+g\le c}D(Q_{Y|W=w}\|P_{Y})\nonumber \\
 & \qquad+\lambda(Q_{X|W=w}(f)+Q_{Y|W=w}(g)-\tau)]\label{eq:-52}\\
 & =\sup_{\lambda\ge0}\sum_{w}Q_{W}(w)[\sup_{f+g\le c}\inf_{Q_{Y|W=w}}D(Q_{Y|W=w}\|P_{Y})\nonumber \\
 & \qquad+\lambda(Q_{X|W=w}(f)+Q_{Y|W=w}(g)-\tau)]\label{eq:-54}\\
 & =\sup_{\lambda\ge0}\sup_{f_{w}+g_{w}\le c,\forall w}\mathbb{E}_{Q_{W}}\big[\lambda(Q_{X|W}(f_{W})-\tau)\nonumber \\
 & \qquad-\log P_{Y}(e^{-\lambda g_{W}})\big],\label{eq:-56}
\end{align}
where 
\begin{itemize}
\item the inf and sup are swapped in \eqref{eq:-42} which follows by the
general minimax theorem \cite[Theorem 5.2.2]{nirenberg1974topics}
together with the convexity of the relative entropy and optimal transport
cost functional; 
\item \eqref{eq:-51} follows by the Kantorovich duality with $f,g$ denoting
bounded continuous functions; 
\item in \eqref{eq:-52} $\inf_{Q_{Y|W}}$ is taken in a pointwise way; 
\item the inf and sup are swapped in \eqref{eq:-54} which follows by the
general minimax theorem \cite[Theorem 5.2.2]{nirenberg1974topics}
by identifying that 1) the optimal value of the sup-inf in \eqref{eq:-54}
is finite (since upper bounded by $\lambda(\C(Q_{X|W=w},P_{Y})-\tau)$),
and 2) by choosing $f,g$ as zero functions, the objective function
turns to be $Q_{Y|W=w}\mapsto D(Q_{Y|W=w}\|P_{Y})-\lambda\tau$ whose
sublevels are compact under the weak topology; 
\item \eqref{eq:-56} follows by Lemma \ref{lem:identity} (and the supremum
over $f,g$ is moved outside of the expectation). 
\end{itemize}
Substituting the dual formula of $\theta$ to $\psi$, we obtain 
\begin{align}
 & \psi(\alpha,\tau)\nonumber \\
 & =\sup_{Q_{XW}:D(Q_{X|W}\|P_{X}|Q_{W})\le\alpha}\theta(\tau,Q_{XW})\\
 & =\sup_{\lambda\ge0}\sup_{f_{w}+g_{w}\le c,\forall w}\sup_{Q_{XW}:D(Q_{X|W}\|P_{X}|Q_{W})\le\alpha}\nonumber \\
 & \quad\mathbb{E}_{Q_{W}}\left[\lambda(Q_{X|W}(f_{W})-\tau)-\log P_{Y}(e^{-\lambda g_{W}})\right]\label{eq:-50}\\
 & =\sup_{\lambda\ge0}\sup_{f_{w}+g_{w}\le c,\forall w\in\{0,1\}}\sup_{Q_{X|W},p\in[0,1]:D(Q_{X|W}\|P_{X}|\Bern(p))\le\alpha}\nonumber \\
 & \quad\mathbb{E}_{W\sim\Bern(p)}\left[\lambda(Q_{X|W}(f_{W})-\tau)-\log P_{Y}(e^{-\lambda g_{W}})\right]\label{eq:-49}\\
 & =\sup_{\lambda\ge0}\sup_{f_{w}+g_{w}\le c,\forall w\in\{0,1\}}\text{\ensuremath{\sup_{p\in[0,1]}}}\inf_{\eta>0}\eta\alpha+\nonumber \\
 & \quad\eta\mathbb{E}_{W\sim\Bern(p)}\log P_{X}(e^{(1/\eta)\left(\lambda(f_{W}-\tau)-\log P_{Y}(e^{-\lambda g_{W}})\right)})\label{eq:-57}\\
 & =\sup_{\lambda\ge0}\sup_{f_{w}+g_{w}\le c,\forall w\in\{0,1\}}\inf_{\eta>0}\eta\alpha+\nonumber \\
 & \quad\eta\max_{w\in\{0,1\}}\log P_{X}(e^{(1/\eta)\left(\lambda(f_{w}-\tau)-\log P_{Y}(e^{-\lambda g_{w}})\right)})\label{eq:-58}\\
 & =\sup_{f_{w}+g_{w}\le c,\forall w\in\{0,1\}}\sup_{\lambda\ge0}\inf_{\eta>0}\max_{w\in\{0,1\}}\eta\alpha\nonumber \\
 & \quad+\eta\log P_{X}(e^{(\lambda/\eta)f_{w}})-\lambda\tau-\log P_{Y}(e^{-\lambda g_{w}}),
\end{align}
where in \eqref{eq:-49}, by Carathéodory's theorem, the alphabet
size of $Q_{W}$ can be restricted to be no larger than $2$, \eqref{eq:-57}
follows by Lemma \ref{lem:conditionalDE}, and \eqref{eq:-58} follows
by the minimax theorem since the objective function is convex in $\eta$.
\end{proof}
\begin{IEEEproof}[Proof of Corollary \ref{cor:uppersemicontinuity}]
By the monotonicity of $\psi$, $\lim_{\alpha'\uparrow\alpha}\lim_{\tau'\downarrow\tau}\psi(\alpha',\tau')\le\psi(\alpha,\tau)$.
So, we only need to focus on the case that $\lim_{\alpha'\uparrow\alpha}\lim_{\tau'\downarrow\tau}\psi(\alpha',\tau')<\infty$.
By the monotonicity of $\psi$, it holds that 
\begin{align}
 & \lim_{\alpha'\uparrow\alpha}\lim_{\tau'\downarrow\tau}\psi(\alpha',\tau')\nonumber \\
 & =\sup_{\alpha'<\alpha,\,\tau'>\tau}\psi(\alpha',\tau')\\
 & =\sup_{\alpha'<\alpha,\,\tau'>\tau}\sup_{f_{w}+g_{w}\le c,\forall w\in\{0,1\}}\sup_{\lambda\ge0}\text{\ensuremath{\sup_{p\in[0,1]}}}\inf_{\eta\ge0}\eta\alpha'-\lambda\tau'\nonumber \\
 & \quad+\mathbb{E}_{W\sim\Bern(p)}[\eta\log P_{X}(e^{\frac{\lambda}{\eta}f_{W}})-\log P_{Y}(e^{-\lambda g_{W}})]\label{eq:-31}\\
 & =\sup_{f_{w}+g_{w}\le c,\forall w\in\{0,1\}}\sup_{\lambda\ge0}\text{\ensuremath{\sup_{p\in[0,1]}}}\inf_{\eta\ge0}\sup_{\alpha'<\alpha,\,\tau'>\tau}\eta\alpha'-\lambda\tau'\nonumber \\
 & \quad+\mathbb{E}_{W\sim\Bern(p)}[\eta\log P_{X}(e^{\frac{\lambda}{\eta}f_{W}})-\log P_{Y}(e^{-\lambda g_{W}})]\label{eq:-30}\\
 & =\sup_{f_{w}+g_{w}\le c,\forall w\in\{0,1\}}\sup_{\lambda\ge0}\text{\ensuremath{\sup_{p\in[0,1]}}}\inf_{\eta\ge0}\eta\alpha-\lambda\tau\nonumber \\
 & \quad+\mathbb{E}_{W\sim\Bern(p)}[\eta\log P_{X}(e^{\frac{\lambda}{\eta}f_{W}})-\log P_{Y}(e^{-\lambda g_{W}})]\\
 & =\psi(\alpha,\tau).
\end{align}
where 
\begin{itemize}
\item by the continuous extension of $\eta\log P_{X}(e^{\frac{\lambda}{\eta}f_{w}})$
to $\eta=0$, $\inf_{\eta>0}$ in \eqref{eq:-57} is replaced by $\inf_{\eta\ge0}$
in \eqref{eq:-31};
\item the $\sup_{\alpha'<\alpha,\,\tau'>\tau}$ and $\inf_{\eta>0}$ are
swapped in \eqref{eq:-30} which follows by the general minimax theorem
\cite[Theorem 5.2.2]{nirenberg1974topics} by identifying that 1)
the optimal value of the sup-inf in \eqref{eq:-31} is finite since
it is upper bounded by $\lim_{\alpha'\uparrow\alpha}\lim_{\tau'\downarrow\tau}\psi(\alpha',\tau')$,
and 2) given $(\alpha',\tau')$ such that $\alpha'>0$, the objective
function in \eqref{eq:-31} goes to infinity as $\eta\to\infty$,
and hence, its sublevels are compact. 
\end{itemize}
\end{IEEEproof}

\section{Proof of Theorem \ref{thm:bound-2-1} }

\label{sec:Proof-of-Theorem} 

 Let $(B_{n})$ be the optimal sets given in Part (a) of Assumption
4.  By the optimality of $B_{n}$, for any $A$ it holds that 

\begin{align}
(P^{\otimes n})^{+}(A) & \ge(P^{\otimes n})^{+}(B_{n})\nonumber \\
 & =n^{1-1/p}e^{-n\alpha}\liminf_{r\downarrow0}F_{r}^{(n)}(B_{n}).\label{eq:-59}
\end{align}
By Part (a) of Assumption 4, 
\begin{equation}
\liminf_{r\downarrow0}F_{r}^{(n)}(B_{n})\ge F_{\epsilon}^{(n)}(B_{n})-\delta(\epsilon,n).\label{eq:-60}
\end{equation}
Therefore, 
\begin{align}
 & \liminf_{n\to\infty}\liminf_{r\downarrow0}F_{r}^{(n)}(B_{n})\nonumber \\
 & \ge\liminf_{n\to\infty}F_{\epsilon}^{(n)}(B_{n})-\delta(\epsilon,\infty)\\
 & \ge\liminf_{n\to\infty}\frac{\alpha-E_{0}^{(n)}(\alpha,\epsilon^{p})}{\epsilon}-\delta(\epsilon,\infty)\\
 & \ge\frac{\alpha-\lim_{\alpha'\downarrow\alpha}\lim_{r'\uparrow\epsilon}\psi(\alpha',r^{\prime p})}{\epsilon}-\delta(\epsilon,\infty)\label{eq:-43}\\
 & =\inf_{r'\in(0,\epsilon)}\frac{\alpha-\lim_{\alpha'\downarrow\alpha}\psi(\alpha',r^{\prime p})}{\epsilon}-\delta(\epsilon,\infty)\label{eq:-45}\\
 & =\inf_{r'\in(0,\epsilon)}\frac{\alpha-\lim_{\alpha'\downarrow\alpha}\psi(\alpha',r^{\prime p})}{r'}-\delta(\epsilon,\infty),
\end{align}
where $\delta(\epsilon,\infty):=\limsup_{n\to\infty}\delta(\epsilon,n)$,
\eqref{eq:-43} follows by Theorem \ref{thm:LD-1}, \eqref{eq:-45}
follows by since by the monotonicity of $\psi$, 
\begin{align*}
\lim_{\alpha'\downarrow\alpha}\lim_{r'\uparrow\epsilon}\psi(\alpha',r^{\prime p}) & =\sup_{\alpha'\in(0,\alpha)}\sup_{r'\in(0,\epsilon)}\psi(\alpha',r^{\prime p})\\
 & =\sup_{r'\in(0,\epsilon)}\sup_{\alpha'\in(0,\alpha)}\psi(\alpha',r^{\prime p})\\
 & =\sup_{r'\in(0,\epsilon)}\lim_{\alpha'\downarrow\alpha}\psi(\alpha',r^{\prime p}).
\end{align*}
Taking $\epsilon\downarrow0$, we obtain that 
\begin{align*}
\liminf_{n\to\infty}\liminf_{r\downarrow0}F_{r}^{(n)}(B_{n}) & \ge\xi(\alpha).
\end{align*}
Substituting this into \eqref{eq:-59} yields the desired inequality. 

We next prove the sharpness.  By Part (b) of Assumption 4,  there
is a family of sets $A_{n,\epsilon}\subseteq\mathcal{X}^{n}$ of probability
$e^{-n\alpha}$ such that 
\begin{equation}
\liminf_{r\downarrow0}F_{r}^{(n)}(A_{n,\epsilon})\le F_{\epsilon}^{(n)}(A_{n,\epsilon})+\delta(\epsilon,n).\label{eq:-60-3}
\end{equation}
Hence, 
\begin{align}
 & \limsup_{n\to\infty}\liminf_{r\downarrow0}F_{r}^{(n)}(A_{n,\epsilon})\nonumber \\
 & \le\limsup_{n\to\infty}F_{\epsilon}^{(n)}(A_{n,\epsilon})+\delta(\epsilon,\infty)\\
 & =\frac{\alpha-\liminf_{n\to\infty}E_{0}^{(n)}(\alpha,\epsilon^{p})}{\epsilon}+\delta(\epsilon,\infty)\nonumber \\
 & \le\frac{\alpha-\psi(\alpha,\epsilon^{p})}{\epsilon}+\delta(\epsilon,\infty),\label{eq:-43-1}
\end{align}
where \eqref{eq:-43-1} follows by Theorem \ref{thm:LD-1}.  Taking
$\epsilon\downarrow0$, we obtain that 
\begin{align*}
\limsup_{\epsilon\downarrow0}\limsup_{n\to\infty}\liminf_{r\downarrow0}F_{r}^{(n)}(A_{n,\epsilon}) & \le\xi(\alpha).
\end{align*}
Substituting this into \eqref{eq:-22} yields 
\begin{align*}
(P^{\otimes n})^{+}(A_{n,\epsilon}) & \le n^{1-1/p}e^{-n\alpha}(\xi(\alpha)+\hat{\delta}(\epsilon,n)),
\end{align*}
where 
\begin{equation}
\limsup_{\epsilon\downarrow0}\limsup_{n\to\infty}\hat{\delta}(\epsilon,n)=0.\label{eq:-68}
\end{equation}
By basic analysis, the condition in \eqref{eq:-68} implies that
there exists a sequence $\epsilon_{n}$ such that $\epsilon_{n}\to0$
and $\hat{\delta}(\epsilon_{n},n)\to0$ as $n\to\infty$. For such
a sequence, 
\begin{align*}
(P^{\otimes n})^{+}(A_{n,\epsilon_{n}}) & \le n^{1-1/p}e^{-n\alpha}(\xi(\alpha)+o_{n}(1)).
\end{align*}

\appendices{}

\section{Proofs of Lemma \ref{lem:bounds}}

\label{sec:Proofs-of-Lemma} Since $\psi_{m}(\alpha+\delta,\tau-\delta)<\infty$,
 there is some $Q_{Y|XW}$ such that 
\begin{align}
\mathbb{E}_{Q}[c(X,Y)] & \le\tau-\delta,\label{eq:-16}\\
D(Q_{Y|W}\|P_{Y}|Q_{W}) & \le\psi_{m}(\alpha+\delta,\tau-\delta)+\delta\label{eq:-18}
\end{align}
hold for all $Q_{XW}$ satisfying $D(Q_{X|W}\|P_{X}|Q_{W})\le\alpha+\delta$.

By assumption, $c(x,y)\le c_{\mathcal{X}}(x)+c_{\mathcal{Y}}(y)$.
So, 
\begin{equation}
\mathbb{E}_{Q}\left[c(X,Y)^{2}\right]\le2(\mathbb{E}_{Q}\left[c_{\mathcal{X}}(X)^{2}\right]+\mathbb{E}_{Q}\left[c_{\mathcal{Y}}(Y)^{2}\right]).\label{eq:-21}
\end{equation}
Since $\mathcal{X}$ is finite,  $c_{\mathcal{X}}$ is bounded. So,
$\mathbb{E}_{Q}\left[c_{\mathcal{X}}(X)^{2}\right]\le\max_{x}c_{\mathcal{X}}(x)^{2}$
 for all $Q_{X}$. It is well known that the relative entropy admits
the following duality: 
\[
D(Q\|P)=\sup_{g}\mathbb{E}_{Q}[g]-\log\mathbb{E}_{P}[\exp(g(Y))],
\]
where the supremum is taken over all measurable function $g$. Substituting
$(Q,P,g)\leftarrow(Q_{XW},P_{Y}Q_{W},c_{\mathcal{Y}}^{2})$ yields
that 
\[
D(Q_{Y|W}\|P_{Y}|Q_{W})\ge\mathbb{E}_{Q}[c_{\mathcal{Y}}^{2}(Y)]-\log\mathbb{E}_{P}[\exp(c_{\mathcal{Y}}^{2}(Y))].
\]
That is, 
\begin{align*}
\mathbb{E}_{Q}[c_{\mathcal{Y}}^{2}(Y)] & \le\psi_{m}(\alpha+\delta,\tau-\delta)+\delta\\
 & \qquad\qquad+\log\mathbb{E}_{P}[\exp(c_{\mathcal{Y}}^{2}(Y))].
\end{align*}
Substituting this into \eqref{eq:-21} yields that 
\begin{align}
\mathbb{E}_{Q}\left[c(X,Y)^{2}\right] & \le2(\max_{x}c_{\mathcal{X}}(x)^{2}+\psi_{m}(\alpha+\delta,\tau-\delta)\nonumber \\
 & \qquad\qquad+\delta+\log\mathbb{E}_{P}[\exp(c_{\mathcal{Y}}^{2}(Y))]).
\end{align}
That is, for the distribution $Q_{Y|XW}$, $\mathbb{E}_{Q}\left[c(X,Y)^{2}\right]$
is bounded uniformly for all $Q_{XW}$ satisfying $D(Q_{X|W}\|P_{X}|Q_{W})\le\alpha+\delta$.

\bibliographystyle{abbrv}
\bibliography{ref}

\begin{IEEEbiographynophoto}{Lei Yu} (Member, IEEE)  received the B.E. and Ph.D. degrees in electronic  
engineering from the University of Science and Technology of China (USTC)  
in 2010 and 2015, respectively. From 2015 to 2020, he worked as a 
Post-Doctoral Researcher at the USTC, National University of Singapore, and  
University of California at Berkeley. He is currently an Associate  
Professor at the School of Statistics and Data Science, LPMC, KLMDASR,  
and LEBPS, Nankai University, China. Since 2024, he has served as  Associate Editor of  the IEEE Transactions on Information Theory. His research interests lie in the  
intersection of probability theory, information theory, and combinatorics.
\end{IEEEbiographynophoto}

\end{document}